%% file: book-2021.tex
\documentclass[12pt,letterpaper,oneside,english]{amsbook}
\usepackage[T1]{fontenc}
\usepackage[latin9]{inputenc}
\usepackage{verbatim}
\usepackage{float}
\usepackage{url}
\usepackage{amstext}
\usepackage{amsthm}
\usepackage{amssymb}
\usepackage{graphicx}

\makeatletter

\newcommand{\noun}[1]{\textsc{#1}}
\providecommand{\tabularnewline}{\\}
\floatstyle{ruled}
\newfloat{algorithm}{tbp}{loa}
\providecommand{\algorithmname}{Algorithm}
\floatname{algorithm}{\protect\algorithmname}

\numberwithin{section}{chapter}
\numberwithin{equation}{section}
\numberwithin{figure}{section}
\theoremstyle{plain}
\newtheorem{thm}{\protect\theoremname}
\theoremstyle{definition}
\newtheorem{defn}[thm]{\protect\definitionname}
\newenvironment{lyxcode}
	{\par\begin{list}{}{
		\setlength{\rightmargin}{\leftmargin}
		\setlength{\listparindent}{0pt}
		\raggedright
		\setlength{\itemsep}{0pt}
		\setlength{\parsep}{0pt}
		\normalfont\ttfamily}%
	 \item[]}
	{\end{list}}
\theoremstyle{plain}
\newtheorem{lem}[thm]{\protect\lemmaname}
\theoremstyle{definition}
\newtheorem{xca}[thm]{\protect\exercisename}

\makeatother

\usepackage{babel}
\providecommand{\definitionname}{Definition}
\providecommand{\exercisename}{Exercise}
\providecommand{\lemmaname}{Lemma}
\providecommand{\theoremname}{Theorem}

\begin{document}
\title{Lecture Notes on Support Preconditioning}
\author{Sivan Toledo\\
Tel Aviv University}
\date{April 2007}
\maketitle

\chapter*{Preface}

These lecture notes were written in 2007. They were available since
then at \url{https://www.tau.ac.il/~stoledo/Support}. I compiled
them and uploaded them without any editing in December 2021.

\include{chapter1}

\include{chapter2}

\include{chapter-direct}

\include{chapter-laplacians}

\include{chapter-laplacian-embeddings}

\include{chapter-augmented-mst}

\appendix
\include{chapter-linear-algebra}

\end{document}

%% file: chapter1.tex
\chapter{Motivation and Overview}

In this chapter we examine solvers for two families of linear systems
of equations. Our goal is motivate certain ideas and techniques, not
to explain in detail how the solvers work. We focus on the measured
behaviors of the solvers and on the fundamental questions that the
solvers and their performance pose.

\section{\label{sec:model meshes}The Model Problems: Two- and Three Dimensional
Meshes}

The coefficient matrices of the linear systems that we study arise
from two families of undirected graphs: two- and three-dimensional
meshes. The vertices of the graphs, which correspond to rows and columns
in the coefficient matrix $A$, are labeled $1$ through $n$. The
edges of the graph correspond to off-diagonal nonzeros in $A$, the
value of which is always $-1$. If $(i,j)$ is an edge in the graph,
then $A_{ij}=A_{ji}=-1$. If $(i,j)$ is not an edge in the mesh for
some $i\neq j$, then $A_{ij}=A_{ji}=0$. The value of the diagonal
elements in $A$ is selected to make the sums of elements in rows
$2,3,\ldots,n$ exactly $0$, and to make the sum of the elements
in the first row exactly $1$. We always number meshes so that vertices
that differ only in their $x$ coordinates are contiguous and numbered
from small $x$ coordinates to large ones; vertices that have the
same $z$ coordinate are also contiguous, numbered from small $y$
coordinates to high ones. Figure~\ref{fig:4-by-3 mesh and matrix}
shows an example of a two-dimensional mesh and the matrix that corresponds
to it. Figure~\ref{fig: 4-by-3-by-2 mesh} shows an example of a
small three-dimensional mesh.

\begin{figure}
\begin{centering}
\begin{minipage}[c][1\totalheight][t]{0.35\columnwidth}%
\noindent \includegraphics[width=1\columnwidth]{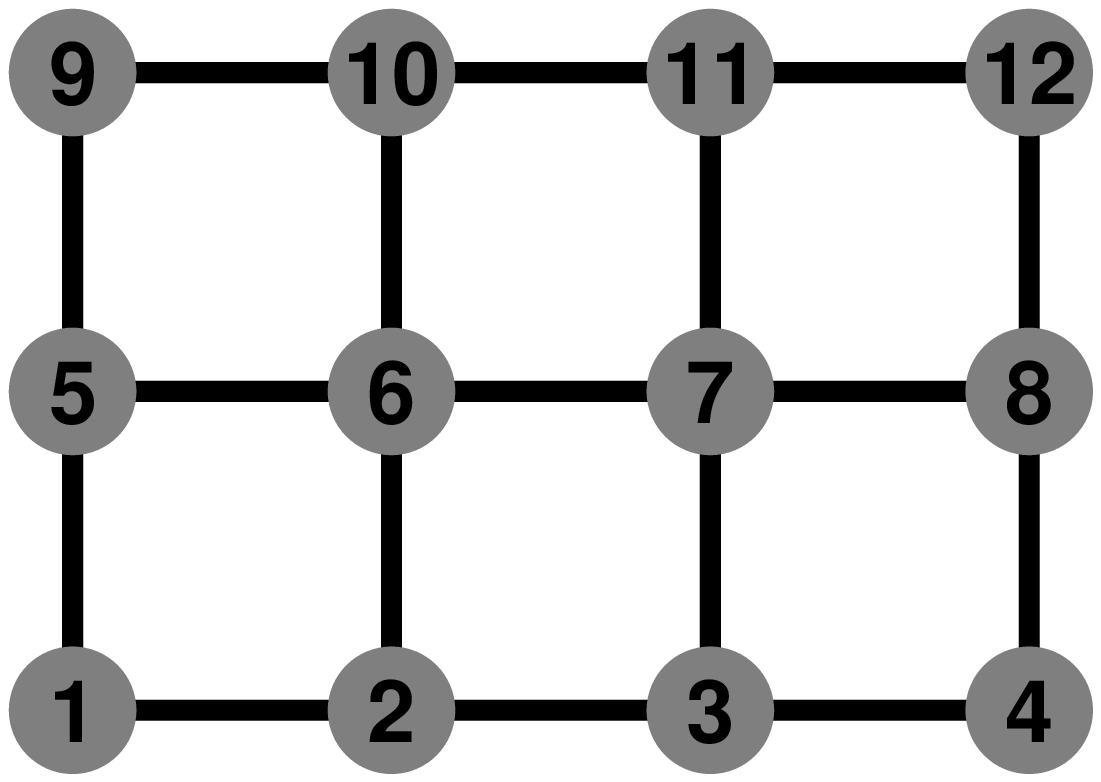}%
\end{minipage}\hfill{}%
\begin{minipage}[c][1\totalheight][t]{0.6\columnwidth}%
\noindent $A=\left[\begin{smallmatrix}3 & -1 &  &  & -1\\
-1 & 3 & -1 &  &  & -1\\
 & -1 & 3 & -1 &  &  & -1\\
 &  & -1 & 2 &  &  &  & -1\\
-1 &  &  &  & 3 & -1 &  &  & -1\\
 & -1 &  &  & -1 & 4 & -1 &  &  & -1\\
 &  & -1 &  &  & -1 & 4 & -1 &  &  & -1\\
 &  &  & -1 &  &  & -1 & 3 &  &  &  & -1\\
 &  &  &  & -1 &  &  &  & 2 & -1\\
 &  &  &  &  & -1 &  &  & -1 & 3 & -1\\
 &  &  &  &  &  & -1 &  &  & -1 & 3 & -1\\
 &  &  &  &  &  &  &  &  &  & -1 & 2
\end{smallmatrix}\right]$%
\end{minipage}
\par\end{centering}
\caption{\label{fig:4-by-3 mesh and matrix}A $4$-by-$3$ two-dimentional
mesh, with vertices labeled $1,\ldots,12$, and the coefficient matrix
$A$ that corresponds to it.}
\end{figure}

\begin{figure}
\begin{centering}
\includegraphics[width=0.8\columnwidth]{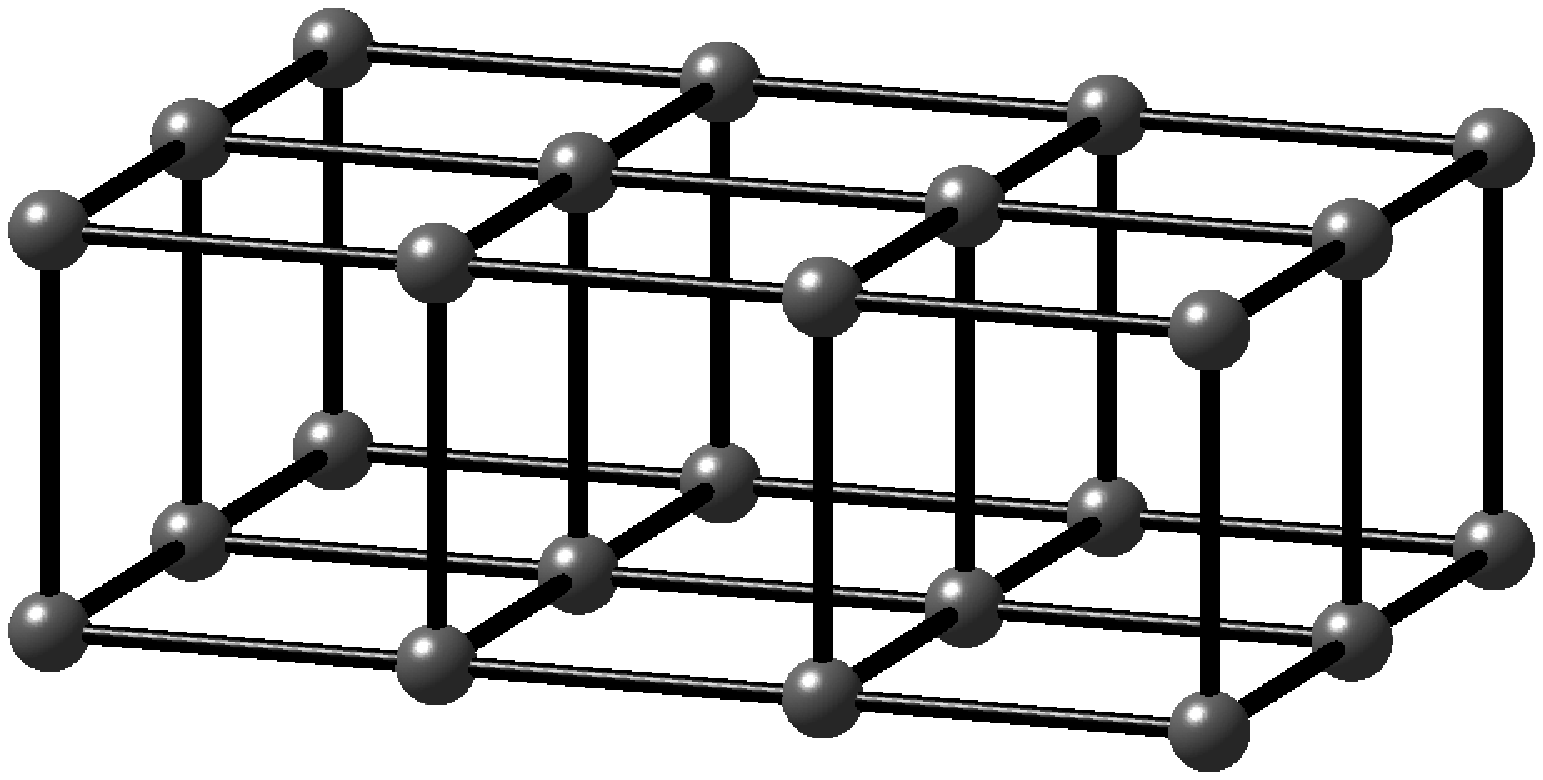}
\par\end{centering}
\caption{\label{fig: 4-by-3-by-2 mesh}A $4$-by-$3$-by-$2$ three-dimensional
mesh. The indices of the vertices are not shown.}
\end{figure}

Matrices similar to these do arise in applications, such as finite-difference
discretizations of partial differential equations. These particular
families of matrices are highly structured, and the structure allows
specialized solvers to be used, such as multigrid solvers and spectral
solvers. The methods that we present in this book are, however, applicable
to a much wider class of matrices, not only to those arising from
structured meshes. We use these matrices here only to illustrate and
motivate various linear solvers.

The experimental results that we present in the rest of the chapter
all use a solution vector $x$ with random uniformly-distributed elements
between $0$ and $1$.

\section{The Cholesky Factorization}

Our matrices are symmetric (by definition) and positive definite.
We do not show here that they are positive definite (all their eigenvalues
are positive); this will wait until a later chapter. But they are
positive definite.

The simplest way to solve a linear system $Ax=b$ with a positive
definite coefficient matrix $A$ is to factor $A$ into a produce
of a lower triangular factor $L$ and its transpose, $A=LL^{T}$.
When $A$ is symmetric positive definite, such a factorization always
exists. It is called the \emph{Cholesky} factorization, and it is
a variant of Gaussian Elimination. Once we compute the factorization,
we can solve $Ax=b$ by solving two triangular linear systems,
\begin{eqnarray*}
Ly & = & b\\
L^{T}x & = & y\;.
\end{eqnarray*}
Solvers for linear systems of equations that factor the coefficient
matrix into a product of simpler matrices are called \emph{direct
solvers}. Here the factorization is into triangular matrices, but
orthogonal, diagonal, tridiagonal, and block diagonal matrices are
also used as factors in direct solvers. 

The factorization can be computed recursively. If $A$ is $1$-by-$1$,
then $L=\sqrt{A_{11}}$. Otherwise, we can partition $A$ and $L$
into 4 blocks each, such that the diagonal blocks are square,
\[
A=\begin{bmatrix}A_{11} & A_{21}^{T}\\
A_{21} & A_{22}
\end{bmatrix}=\begin{bmatrix}L_{11}\\
L_{21} & L_{22}
\end{bmatrix}\begin{bmatrix}L_{11}^{T} & L_{21}^{T}\\
 & L_{22}^{T}
\end{bmatrix}=LL^{T}\;.
\]
This yields the following equations, which define the blocks of $L$,
\begin{eqnarray}
A_{11} & = & L_{11}L_{11}^{T}\label{eq:chol-11}\\
A_{21} & = & L_{21}L_{11}^{T}\label{eq:chol-21}\\
A_{22} & = & L_{21}L_{21}^{T}+L_{22}L_{22}^{T}\;.\label{eq:chol-22}
\end{eqnarray}
We compute $L$ by recursively solving Equation~\ref{eq:chol-11}
for $L_{11}$. Once we have computed $L_{11}$, we solve Equation~\ref{eq:chol-21}
for $L_{21}$ by substitution. The final step in the computation of
$L$ is to subtract $L_{21}L_{21}^{T}$ from $A_{22}$ and to factor
the difference recursively. 

This algorithm is reliable and numerically stable, but if implemented
naively, it is slow. Let $\phi(n)$ be the number of arithmetic operations
that this algorithm performs on a matrix of size $n$. To estimate
$\phi$, we partition $A$ so that $A_{11}$ is $1$-by-$1$. The
partitioning does not affect $\phi(n)$, but this particular partitioning
makes it easy to estimate $\phi(n)$. With this partitioning, computing
$L$ involves computing one square root, dividing an $(n-1)$-vector
by the root, computing the symmetric outer product of the scaled $(n-1)$-vector,
subtracting this outer product from a symmetric $(n-1)$-by-$(n-1)$
matrix, and recursively factoring the difference. We can set this
up as a recurrence relation,
\begin{eqnarray*}
\phi(n) & = & 1+(n-1)+(n-1)^{2}+\phi(n-1)\\
\phi(1) & = & 1\;.
\end{eqnarray*}
This yields
\begin{eqnarray*}
\phi(n) & = & n+\sum_{j=2}^{n}(j-1)+\sum_{j=2}^{n}(n-1)^{2}\\
 & = & n+\sum_{j=1}^{n-1}j+\sum_{j=1}^{n-1}j^{2}\\
 & = & n+\frac{n(n-1)}{2}+\frac{n(n-\frac{1}{2})(n-1)}{3}\\
 & = & \frac{n^{3}}{3}+o\left(n^{3}\right)\;.
\end{eqnarray*}
The little-$o$ notation means that we have neglected terms that grow
slower than $cn^{3}$ for any constant $c$.

Cubic growth means that for large meshes, a Cholesky-based solver
is slow. For a mesh with a million vertices, the solver needs to perform
more than $0.3\times10^{18}$ arithmetic operations. Even at a rate
of $10^{12}$ operations per second, the factorization takes almost
4 days. On a personal computer (say a 2005 model), the factorization
will take more than a year. A mesh with a million vertices may seem
like a large mesh, but it is not. The mesh has less than $2$ million
edges in two dimensions and less than $3$ million in three dimensions,
so its representation does not take a lot of space in memory. We can
do a lot better.

\section{Sparse Cholesky}

One way to speed up the factorization is to exploit sparsity in $A$.
If we inspect the nonzero pattern of $A$, we see that most of its
elements are zero. If $A$ is derived from a 2-dimensional mesh, only
5 diagonals contain nonzero elements. If $A$ is derived from a 3-dimensional
mesh, only 7 diagonals are nonzero. Figure~\ref{fig:11 by 15 pattern}
shows the nonzero pattern of such matrices. Only $O(n)$ elements
out of $n^{2}$ in $A$ are nonzero.

\begin{figure}
\def\mata{\vcenter{\hbox{\includegraphics[height=0.33\columnwidth]{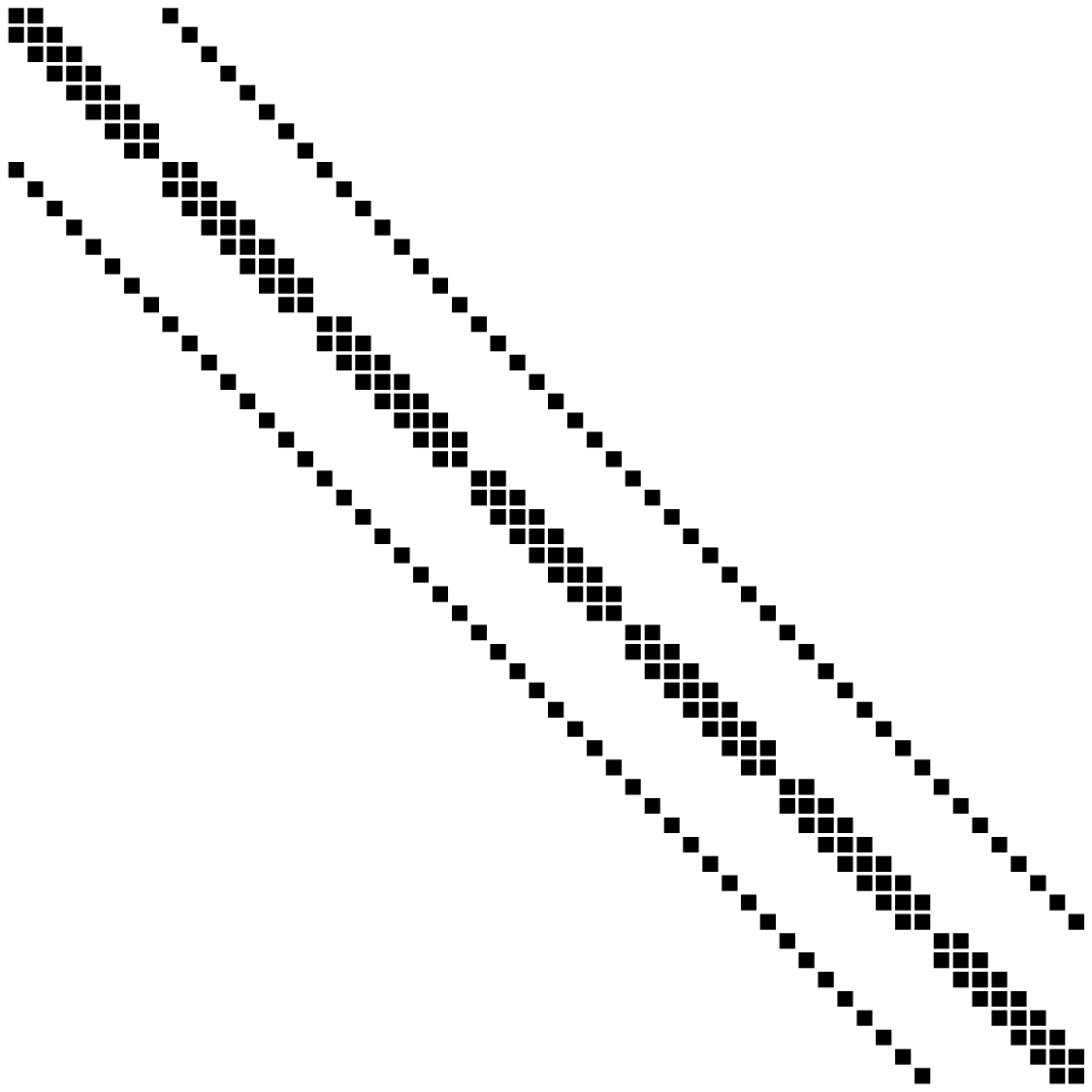}}}}

\def\matb{\vcenter{\hbox{\includegraphics[height=0.33\columnwidth]{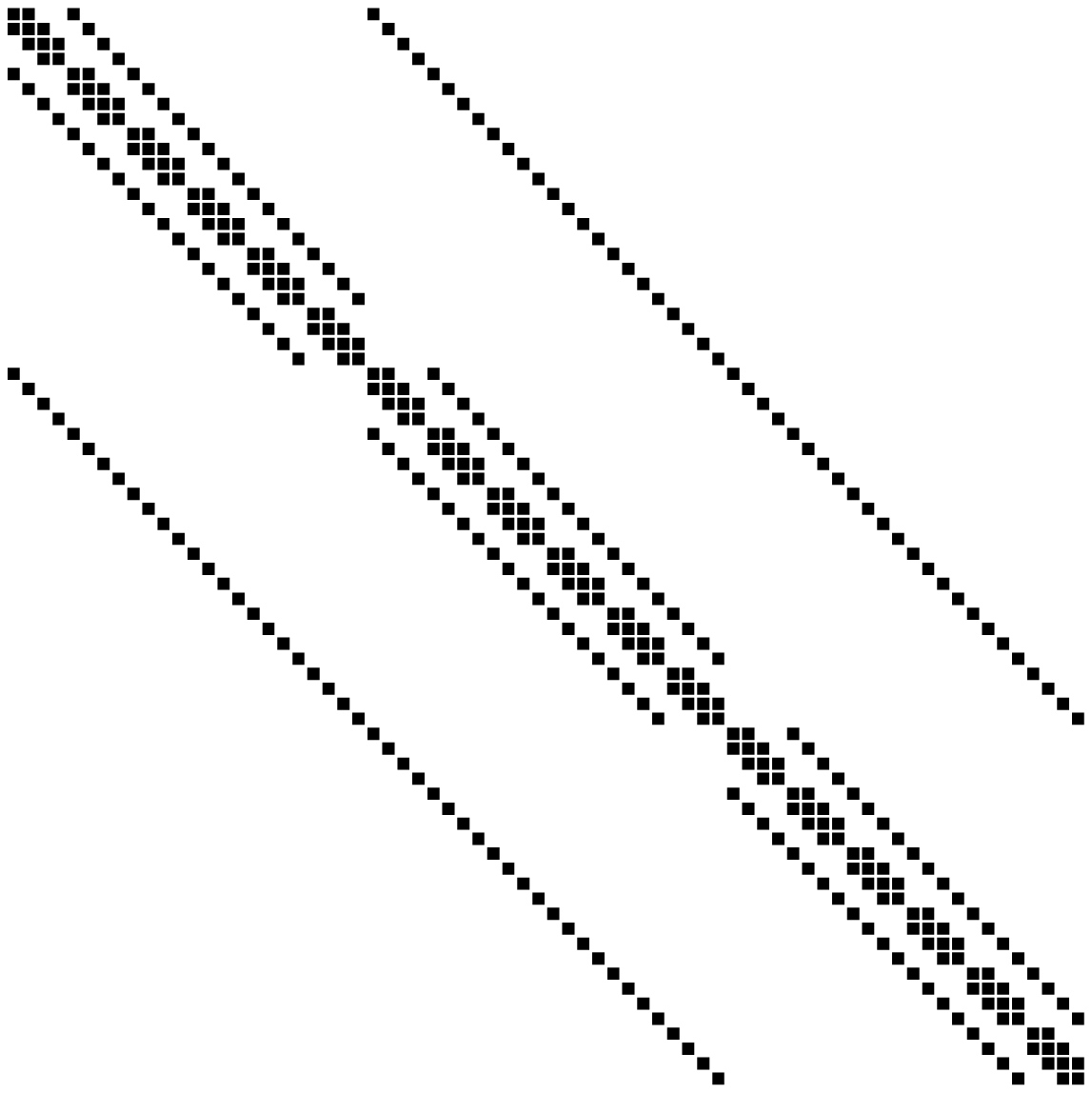}}}}

\newsavebox{\mat}

\newsavebox{\fact}

\savebox{\mat}{\raisebox{-1.45in}{\includegraphics[height=3in]{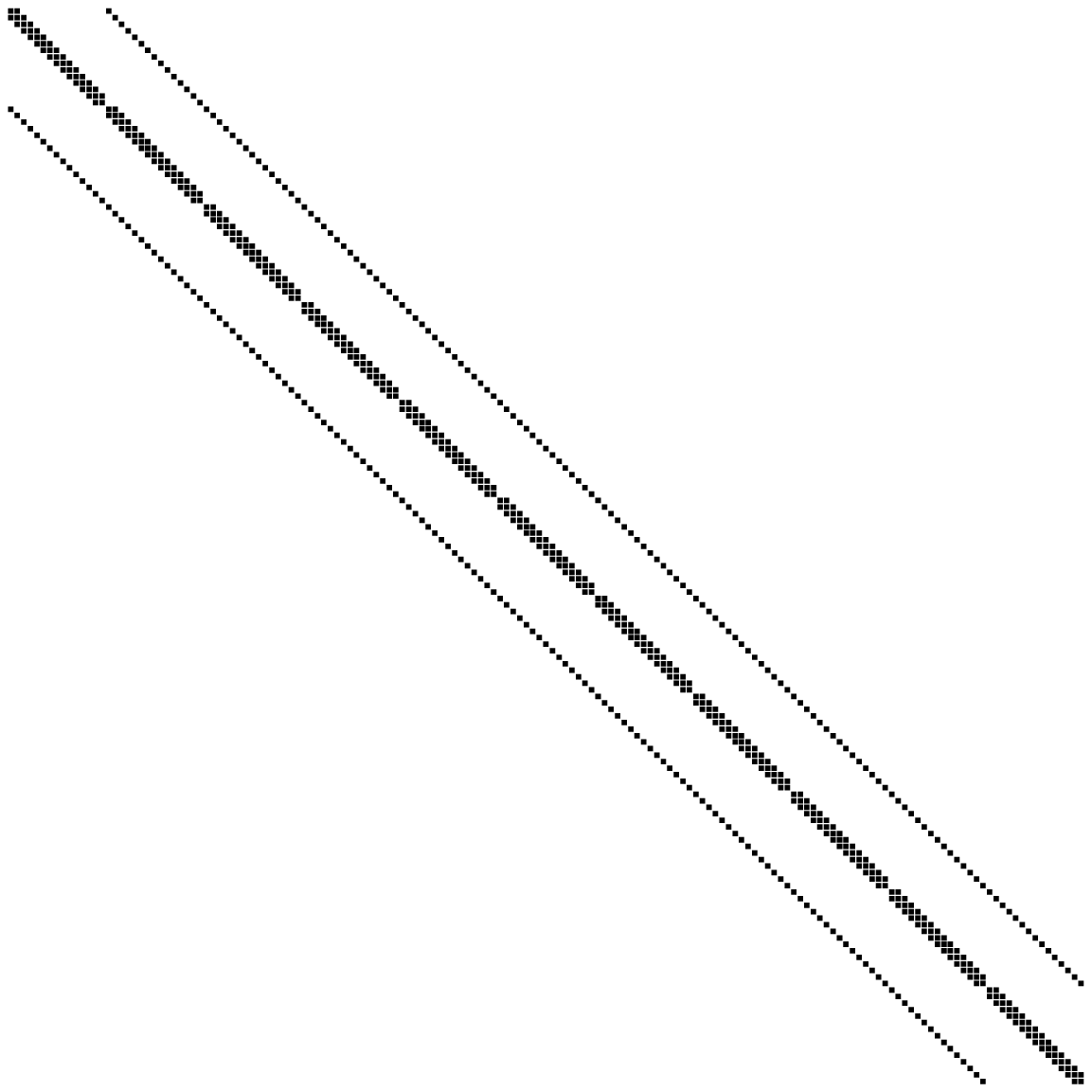}}}

\begin{centering}
$A_{\textrm{2D}}=\left[\mata\right]$\hfill{}$A_{\textrm{3D}}=\left[\matb\right]$
\par\end{centering}
\caption{\label{fig:11 by 15 pattern}The nonzero pattern of matrices corresponding
to an $8$-by-$7$ mesh (left) and to a $4$-by-$6$-by-$3$ mesh.
Nonzero elements are denoted by small squares. The other elements
are all zeros.}
\end{figure}

\emph{Sparse} factorization codes exploit the sparsity of $A$ and
avoid computations involving zero elements (except, sometimes, for
zeros that are created by exact cancellations). Figure~\ref{fig: meshes dense sparse}
compares the running times of a sparse Cholesky solver with that of
a dense Cholesky solver, which does not exploit zeros. The sparse
solver is clearly must faster, especially on 2-dimensional meshes.

\begin{figure}
\begin{centering}
\includegraphics[width=0.45\columnwidth]{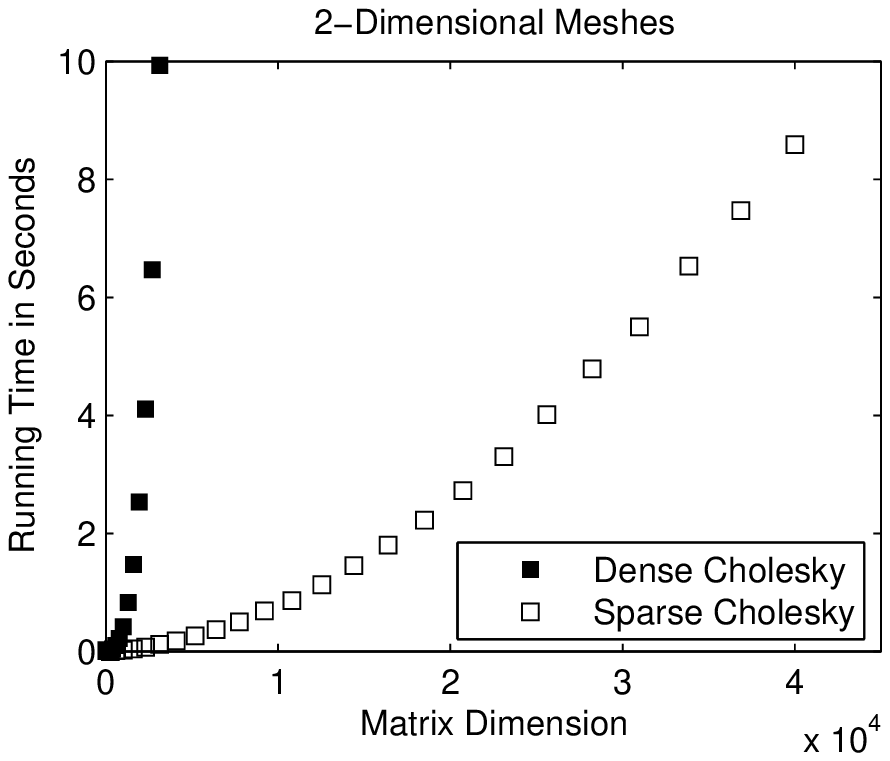}\hfill{}\includegraphics[width=0.45\columnwidth]{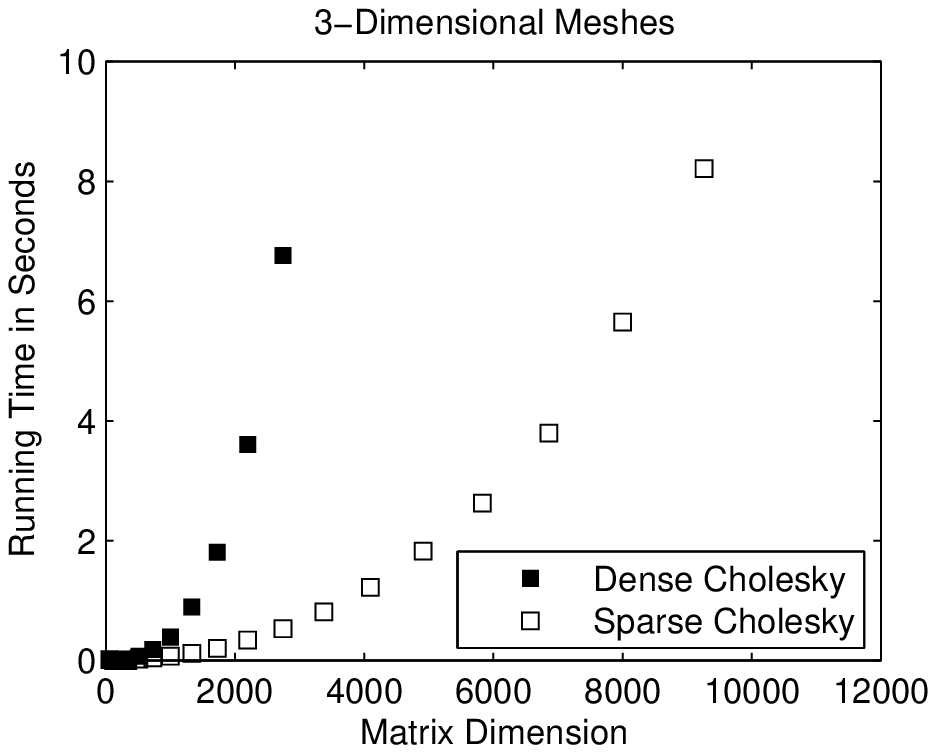}
\par\end{centering}
\caption{\label{fig: meshes dense sparse}The performance of dense and sparse
Cholesky on 2- and 3-dimensional meshes. The data in this graph, as
well as in all the other graphs in this chapter, was generated using
\noun{Matlab} 7.0.}
\end{figure}

To fully understand the performance of the sparse solver, it helps
to inspect the nonzero pattern of the Cholesky factor $L$, shown
in Figure~\ref{fig:11 by 15 with factor}. Exploiting sparsity pays
off because the factor is also sparse: it fills, but not completely.
It is easy to characterize exactly where the factor fills. It fills
almost completely within the band structure of $A$, the set of diagonals
$d$ that are closer to the main diagonal than a nonzero diagonal.
Outside this band structure, the elements of $L$ are all zero. For
a $\sqrt{n}$-by-$\sqrt{n}$ mesh, the outer diagonal is $\sqrt{n}$,
so the total amount of arithmetic in a sparse factorization is $n\cdot2\left(\sqrt{n}\right)^{2}+o(n^{2})=2n^{2}+o(n^{2})$.
For large $n$, this is a lot less than the dense $n^{3}/3+o(n^{2})$
bound. For three dimensional meshs, exploiting sparsity helps, but
not as much. In a matrix corresponding to a $\sqrt[3]{n}$-by-$\sqrt[3]{n}$-by-$\sqrt[3]{n}$
mesh, the outer nonzero diagonal is $n^{2/3}$, so the amount of arithmetic
is $2n^{7/3}+o(n^{7/3})$.

\begin{figure}
\def\mata{\vcenter{\hbox{\includegraphics[height=0.33\columnwidth]{matlab/plots/mesh2d_11_15_pattern.eps}}}}

\def\faca{\vcenter{\hbox{\includegraphics[height=0.33\columnwidth]{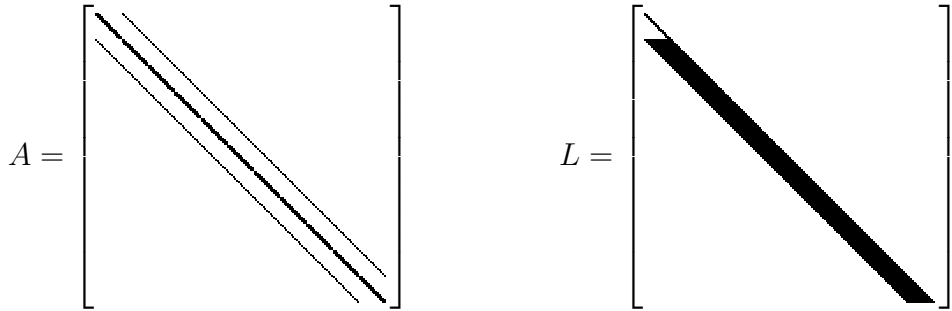}}}}

\savebox{\mat}{\raisebox{-0.166\columnwidth}{\includegraphics[width=0.35\columnwidth]{matlab/plots/mesh2d_11_15_pattern.eps}}}

\savebox{\fact}{\raisebox{-0.166\columnwidth}{\includegraphics[width=0.35\columnwidth]{matlab/plots/mesh2d_11_15_factor.eps}}}

\begin{centering}
$A=\left[\mata\right]$\hfill{}$L=\left[\faca\right]$
\par\end{centering}
\caption{\label{fig:11 by 15 with factor}The matrix of a $15$-by-$11$ mesh
(left) and its Cholesky factor (right).}
\end{figure}

\section{Preordering for Sparsity}

Sparse factorization codes exploit zero elements in the reduced matrices
and the factors. It turns out that by reordering the rows and columns
of a matrix we can reduce fill in the reduced matrices and the factors.
The most striking example for this phenomenon is the arrow matrix,
\[
A=\begin{bmatrix}n+1 & -1 & -1 & -1 & \cdots & -1\\
-1 & 1\\
-1 &  & 1\\
-1 &  &  & 1\\
\vdots &  &  &  & \ddots\\
-1 &  &  &  &  & 1
\end{bmatrix}\;.
\]
When we factor this matrix, it fills completely after the elimination
of the first column, since the outer product $L_{2\colon n,1}L_{2\colon n,1}^{T}$
is full (has no zeros). But if we reverse the order of the rows and
columns, we have
\begin{eqnarray*}
PAP^{T} & = & \left[\begin{smallmatrix} &  &  &  &  & 1\\
 &  &  &  & 1\\
 &  &  & 1\\
 &  & 1\\
 & \cdot\\
1
\end{smallmatrix}\right]\left[\begin{smallmatrix}n+1 & -1 & -1 & -1 & \cdots & -1\\
-1 & 1\\
-1 &  & 1\\
-1 &  &  & 1\\
\vdots &  &  &  & \ddots\\
-1 &  &  &  &  & 1
\end{smallmatrix}\right]\left[\begin{smallmatrix} &  &  &  &  & 1\\
 &  &  &  & 1\\
 &  &  & 1\\
 &  & 1\\
 & \cdot\\
1
\end{smallmatrix}\right]\\
 & = & \left[\begin{smallmatrix}1 &  &  &  &  & -1\\
 & 1 &  &  &  & -1\\
 &  & 1 &  &  & -1\\
 &  &  & 1 &  & -1\\
 &  &  &  & \ddots & \vdots\\
-1 & -1 & -1 & -1 & -1 & n+1
\end{smallmatrix}\right]\;.
\end{eqnarray*}
The Cholesky factor of $PAP^{T}$ is as sparse as $PAP^{T}$
\[
\textrm{chol}(PAP^{T})=\begin{bmatrix}1\\
 & 1\\
 &  & 1\\
 &  &  & 1\\
 &  &  &  & \ddots\\
-1 & -1 & -1 & -1 & \cdots & 1
\end{bmatrix}\;.
\]

Solving a linear system with a Cholesky factorization of a permuted
coefficient matrix is trivial. We apply the same permutation to the
right-hand-side $b$ before the two triangular solves, and apply the
inverse permutation after the triangular solves.

\begin{figure}
\def\mata{\vcenter{\hbox{\includegraphics[height=0.33\columnwidth]{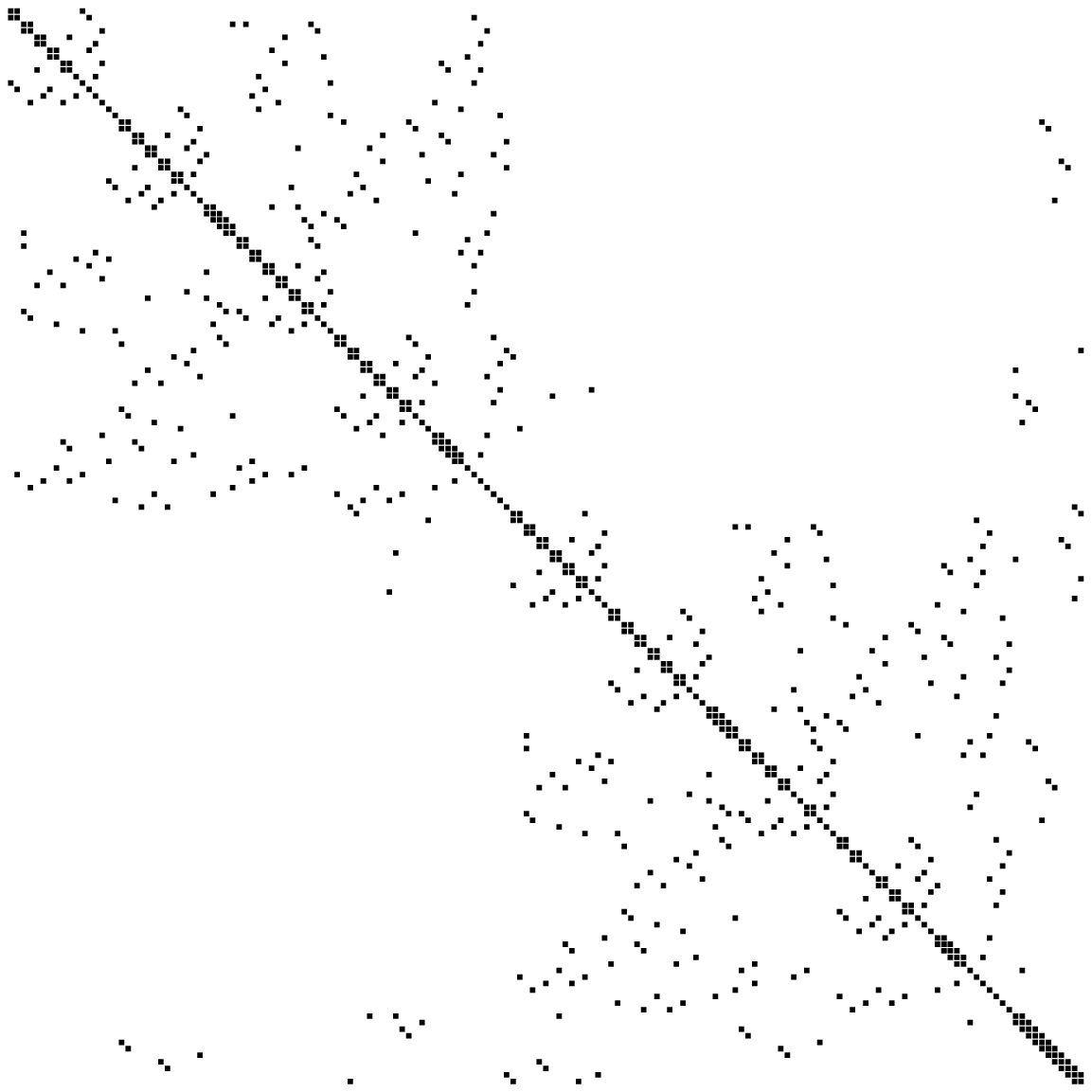}}}}

\def\faca{\vcenter{\hbox{\includegraphics[height=0.33\columnwidth]{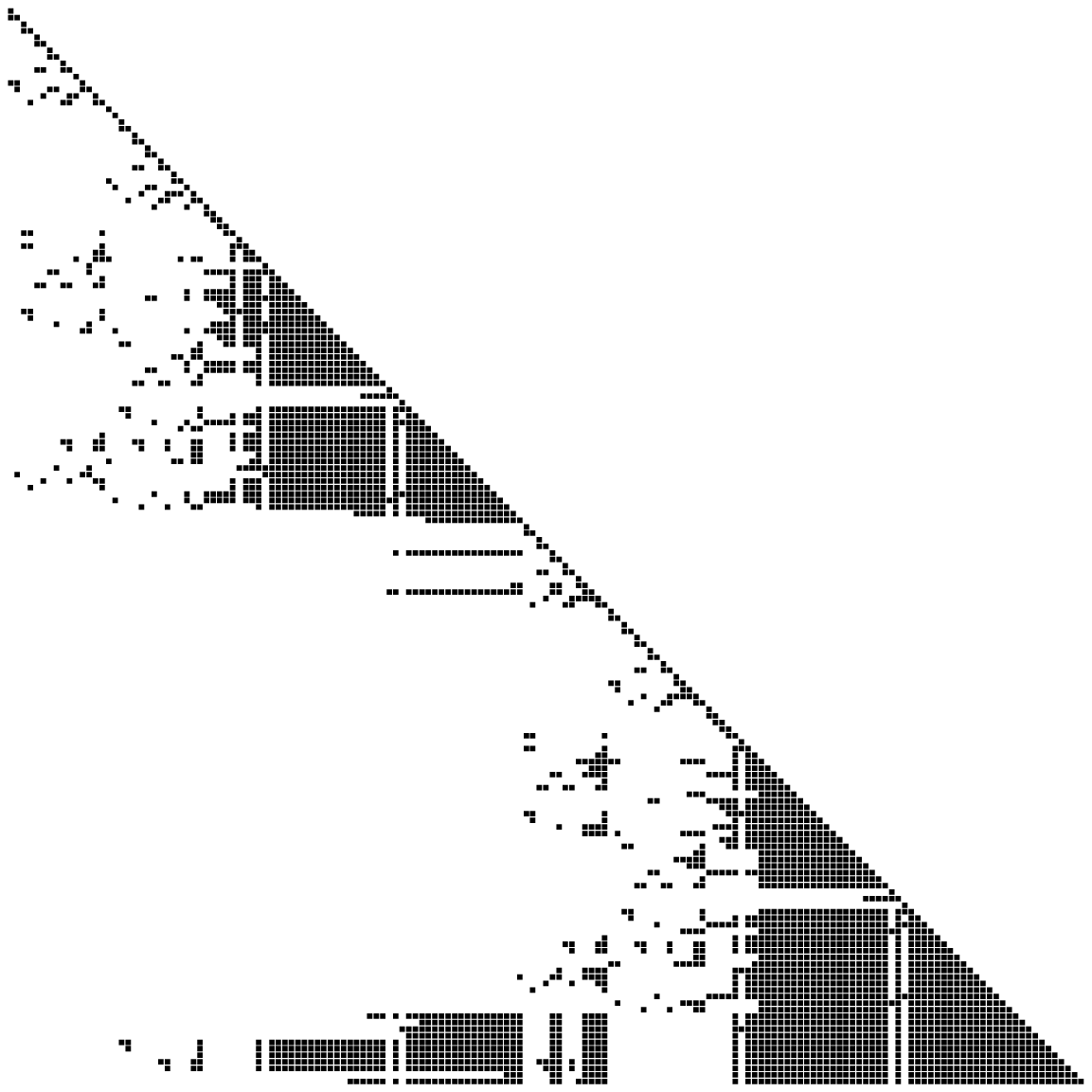}}}}\savebox{\mat}{\raisebox{-0.166\columnwidth}{\includegraphics[width=0.35\columnwidth]{matlab/plots/mesh2d_11_15_nesteddis.eps}}}

\savebox{\fact}{\raisebox{-0.166\columnwidth}{\includegraphics[width=0.35\columnwidth]{matlab/plots/mesh2d_11_15_ndfactor.eps}}}

\begin{centering}
$\tilde{A}=\left[\mata\right]$\hfill{}$\tilde{L}=\left[\faca\right]$
\par\end{centering}
\caption{\label{fig:11 by 15 nested dissection}A symmetric nested-dissection
permutation $\tilde{A}$ of the matrix shown in Figure~\ref{fig:11 by 15 with factor}
and the Cholesky factor $\tilde{L}$ of $\tilde{A}$.}
\end{figure}

For matrices that correspond to meshes, there is no permutation that
leads to a no-fill factorization, but there are permutations that
reduce fill significantly. Figure~\ref{fig:11 by 15 nested dissection}
shows a theoretically-effictive fill-reducing permutation of the matrix
of an $11$-by-$15$ mesh and its Cholesky factor. This symmetric
reordering of the rows and columns is called a \emph{nested dissection
ordering}. It reorders the rows and columns by finding a small set
of vertices in the mesh whose removal breaks the mesh into two separate
components of roughly the same size. This set is called a \emph{vertex
separator} and the corresponding rows and columns are ordered last.
The rows and columns corresponding to each connected component of
the mesh minus the separator are ordered using the same strategy recursively.
Figure~\ref{fig:3 by 7 nested dissection} illustrates this construction.

\begin{figure}
\def\mesh{\vcenter{\hbox{\includegraphics[width=0.35\columnwidth]{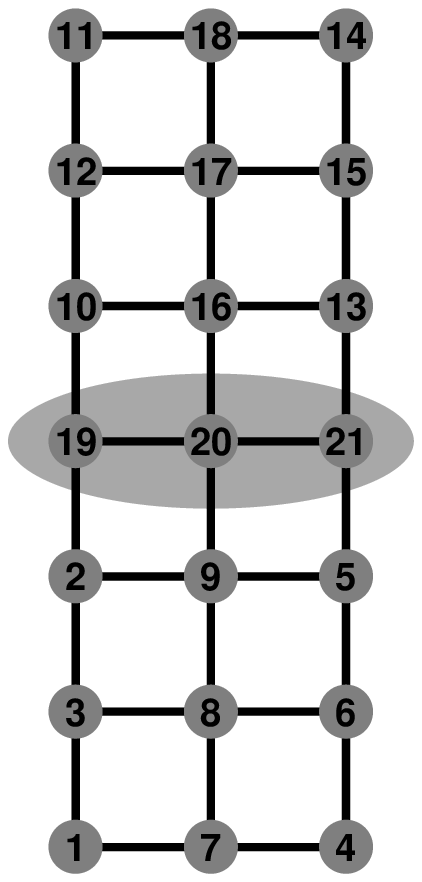}}}}

\def\mat{\vcenter{\hbox{\includegraphics[height=0.35\columnwidth]{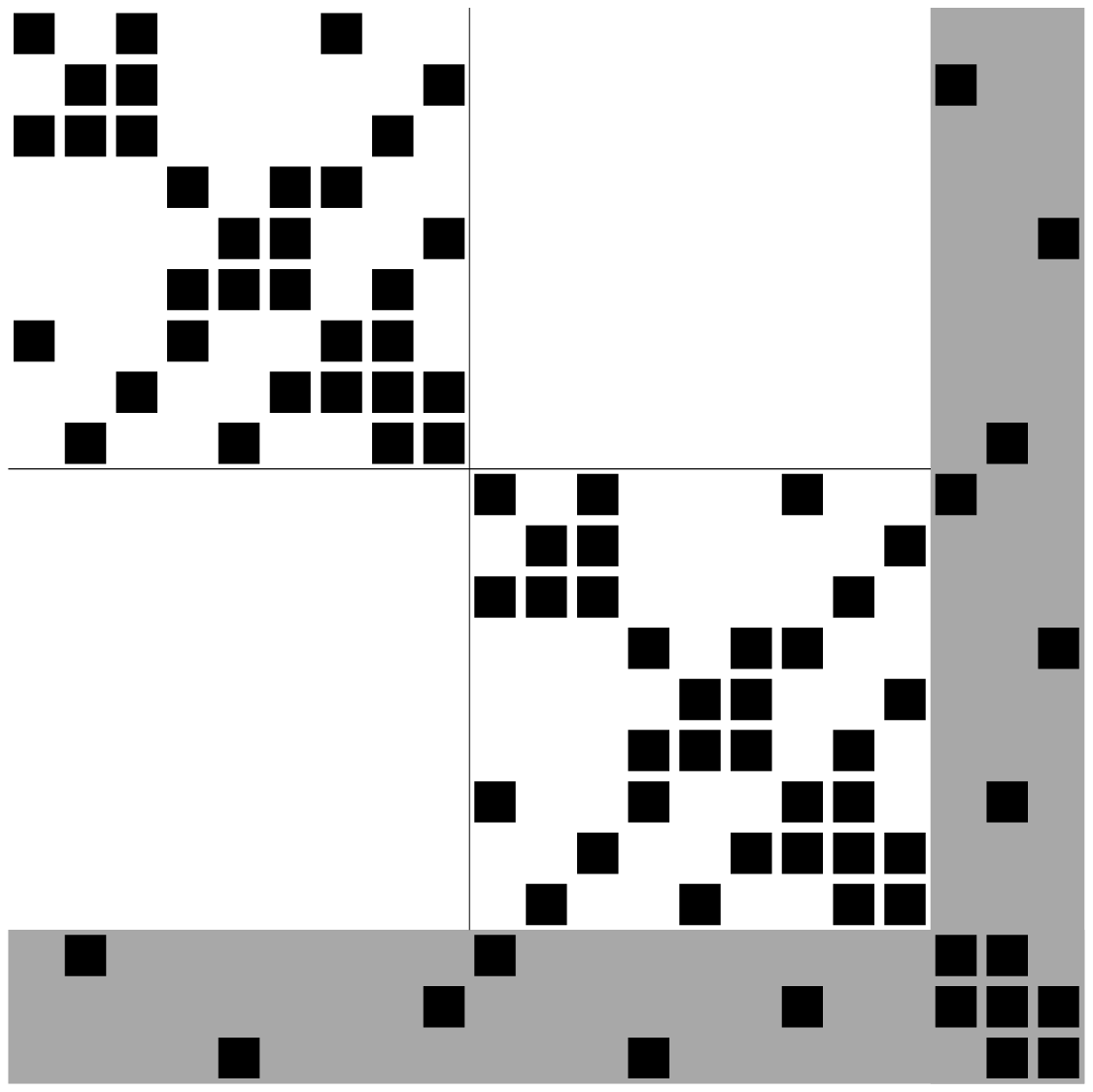}}}}

\begin{centering}
$\mesh$\hfill{}$\tilde{A}=\begin{bmatrix}\mat\end{bmatrix}$
\par\end{centering}
\caption{\label{fig:3 by 7 nested dissection}A nested-dissection ordering
of a mesh (left) and the nested-dissection reordered matrix of the
mesh (right). The shaded ellipse in the mesh shows the top-level vertex
separators. The shading in the matrix shows the rows and columns that
correspond to the separator; the thin lines separate the submatrices
that correspond to the connected components of the mesh minus the
separator..}
\end{figure}

The factorization of the nested-dissection-ordered matrix of a a $\sqrt{n}$-by-$\sqrt{n}$
mesh performs $\Theta(n^{3/2})$ arithmetic operations.\footnote{explain asymptotic notation.}
(We don't give here the detailed analysis of fill and work under nested-dissection
orderings.) For large meshes, this is a significant improvement over
the $\Theta(n^{2})$ operations that the factorization performs without
reordering. For a $\sqrt[3]{n}$ -by-$\sqrt[3]{n}$-by-$\sqrt[3]{n}$
mesh, the factorization performs $\Theta(n^{2})$ operations. This
is again an improvement, but the $\Theta(n^{2})$ solution cost is
too high for many applications. 

\begin{figure}
\begin{centering}
\includegraphics[width=0.45\columnwidth]{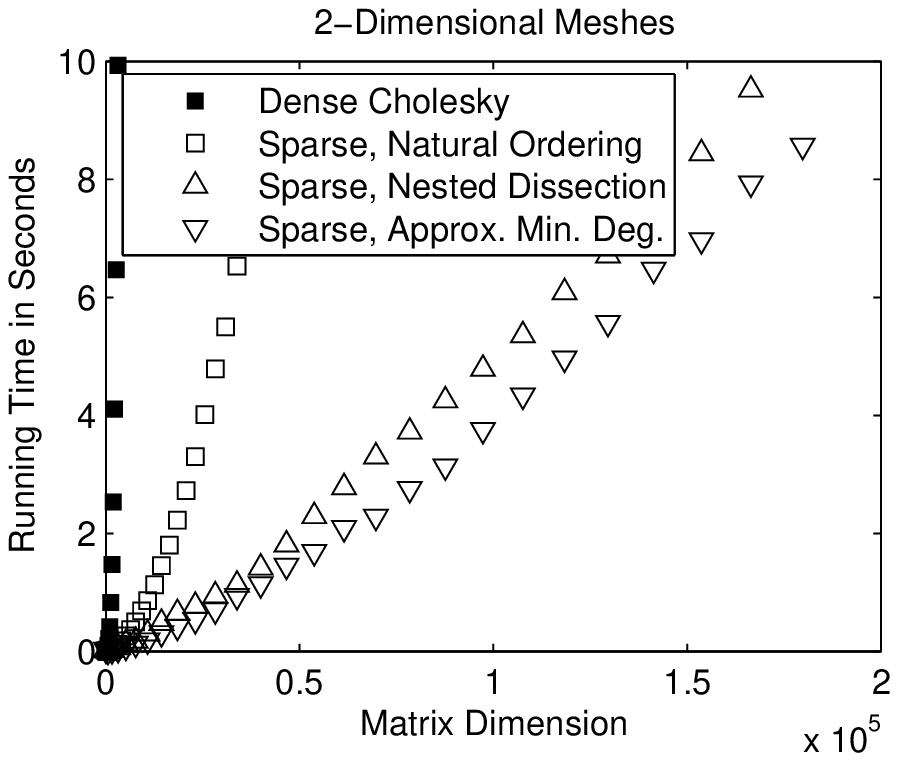}\hfill{}\includegraphics[width=0.45\columnwidth]{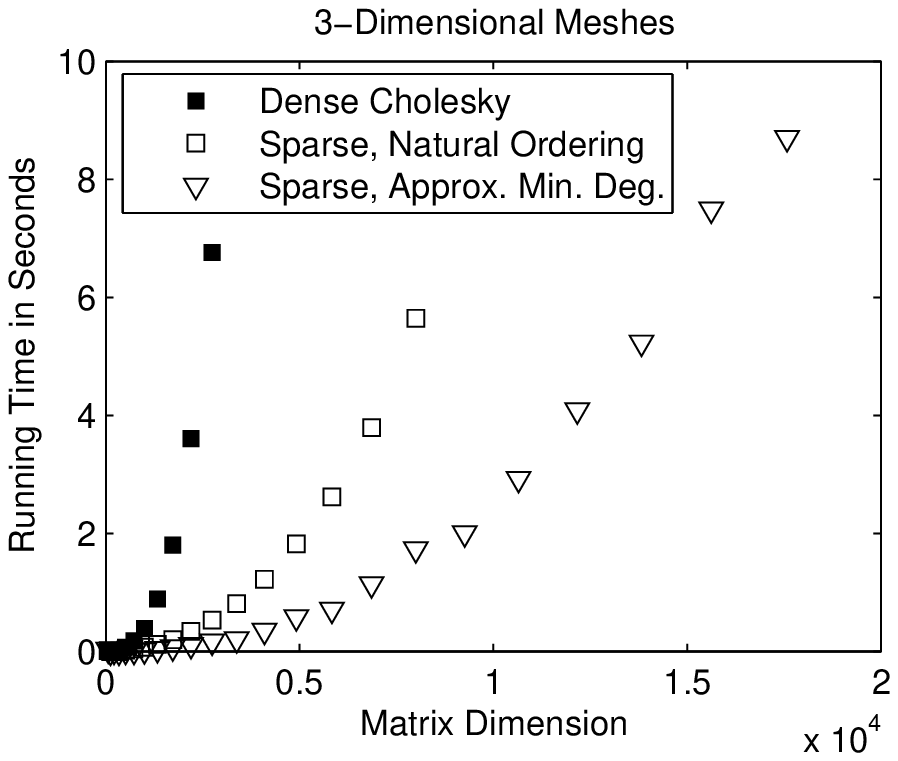}
\par\end{centering}
\caption{\label{fig: meshes sparse orderings}The effect of reordering the
symmetrically rows and columns on the running time of sparse Cholesky
factorizations.}
\end{figure}

Figure~\ref{fig: meshes sparse orderings} compares the running times
of sparse Choesky applied to the natural ordering of the meshes and
to fill-reducing orderings. The nested-dissection and minimum-degree
orderings speed up the factorization considerably.

Unfortunately, nested-dissection-type orderings are usually the best
we can do in terms of fill and work in sparse factorization of meshes,
both structured and unstructured ones. A class of ordering heuristics
called \emph{minimum-degree} heuristics can beat nested dissection
on small and medium-size meshes. On larger meshes, various hybrids
of nested-dissection and minimum-degree are often more effective than
pure variants of either method, but these hybrids have the same asymptotic
behavior as pure nested-dissection orderings.

To asymptotically beat sparse Cholesky factorizations, we something
more sophisticated than a factorization of the coefficient matrix.

\section{Krylov-Subspace Iterative Methods}

A radically different way to solve $Ax=b$ is to to find the vector
$x^{(k)}$ that minimizes $\|Ax^{(k)}-b\|$ in a subspace of $\mathbb{R}^{n}$,
the subspace that is spanned by $b,Ab,A^{2}b,\ldots,A^{k-1}b$. This
subspace is called the \emph{Krylov} subspace of $A$ and $b$, and
such solvers are called \emph{Krylov-subspace solvers}. This idea
is appealing because of two reasons, the first obvious and the other
less so. First, we can create a basis for this subspace by repeatedly
multiplying vectors by $A$. Since our matrices, as well as many of
the matrices that arise in practice, are so sparse, multiplying $A$
by a vector is cheap. Second, it turns out that when $A$ is symmetric,
finding the minimizer $x^{(k)}$ is also cheap. These considerations,
taken together, imply that the work performed by such solvers is proportional
to $k$. If $x^{(k)}$ converges quickly to $x$, the total solution
cost is low. But for most real-world problems, as well as for our
meshes, convergence is slow.

The \emph{conjugate gradients} method (\noun{cg}) and the \emph{minimal-residual}
method (\noun{minres}) are two Krylov-subspace methods that are appropriate
for symmetric positive-definite matrices. \noun{Minres} is more theoretically
appealing, because it minimizes the residual $Ax^{(k)}-b$ in the
$2$-norm. It works even if $A$ is indefinite. Its main drawback
is that it suffers from a certain numerically instability when implemented
in floating point. The instability is almost always mild, not catastrophic,
but it can slow convergence and it can prevent the method from converging
to very small residuals. Conjugate gradients is more reliable numerically,
but it minimizes the residual in the $A$-norm, not in the $2$-norm.

\begin{figure}
\begin{centering}
\includegraphics[width=0.45\columnwidth]{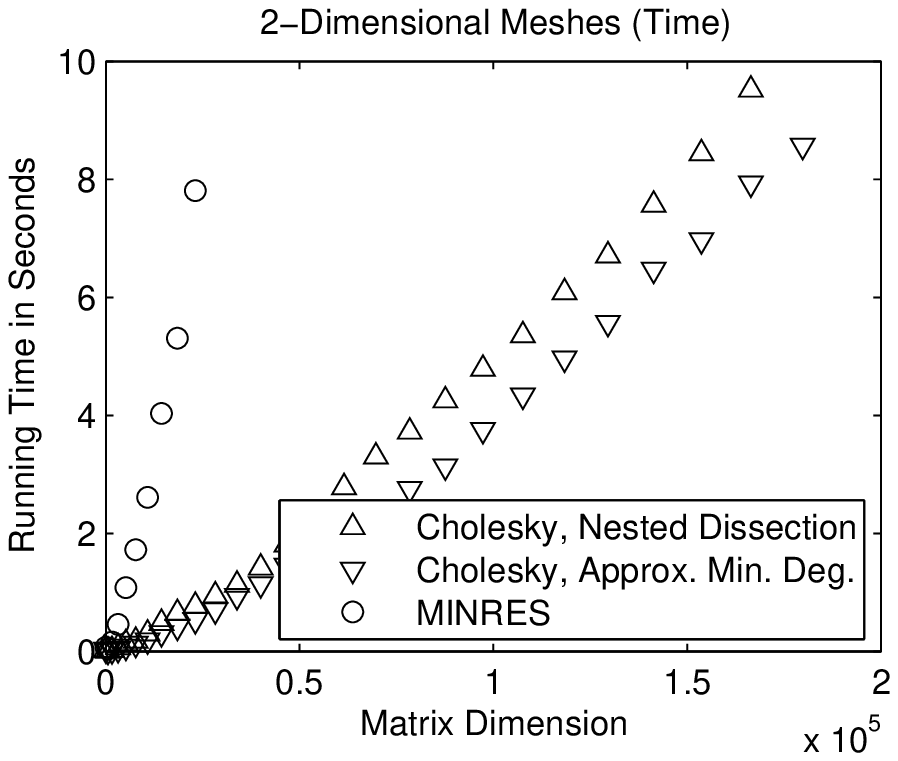}\hfill{}\includegraphics[width=0.45\columnwidth]{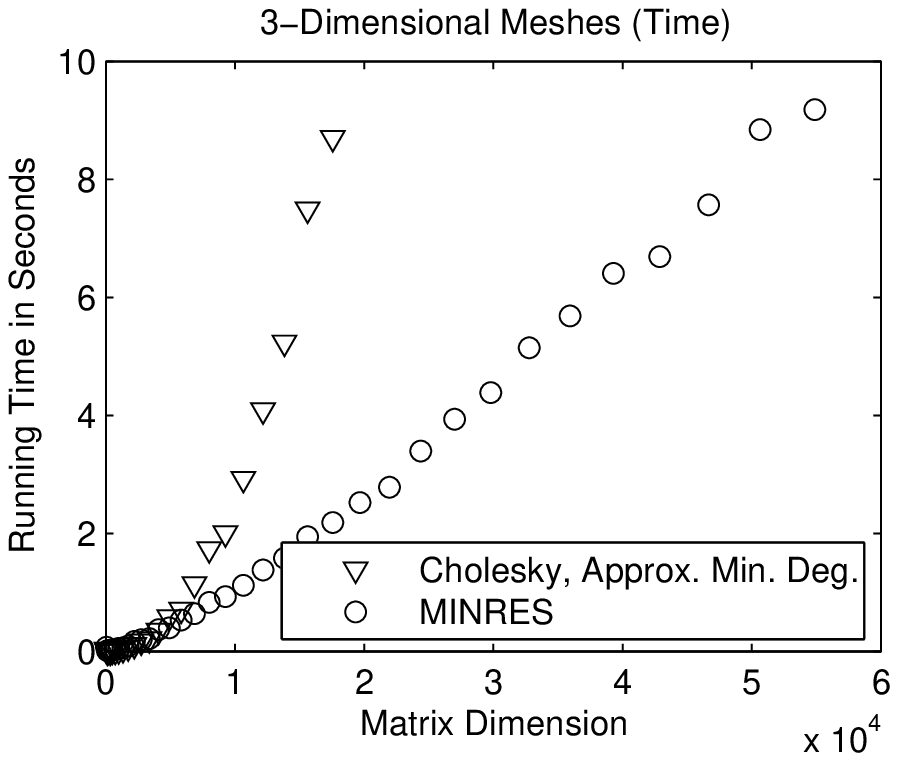}\\
\includegraphics[width=0.45\columnwidth]{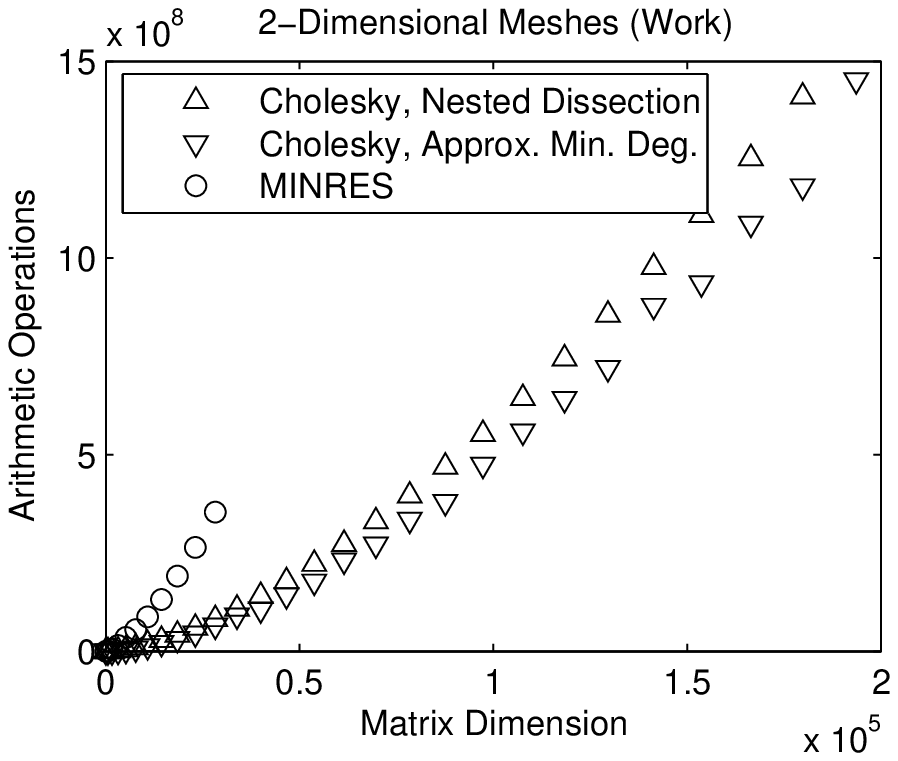}\hfill{}\includegraphics[width=0.45\columnwidth]{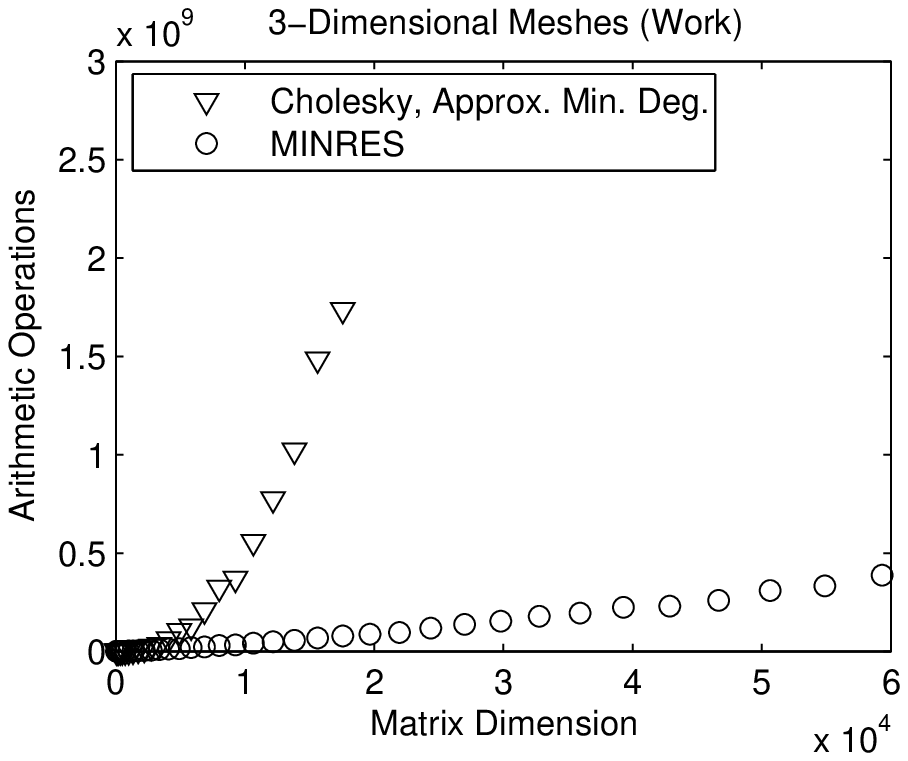}
\par\end{centering}
\caption{\label{fig: meshes direct vs minres}A comparison of a sparse direct
Cholesky solver with an iterative \noun{minres} solver.}
\end{figure}

Figure~\ref{fig: meshes direct vs minres} compares the performance
of a sparse direct solver (with a fill-reducing ordering) with the
performance of a \noun{minres} solver. The performance of an iterative
solver depends on the prescribed convergence threshold. Here we instructed
the solver to halt once the relative norm of the residual
\[
\frac{\|Ax^{(k)}-b\|}{\|b\|}
\]
 is $10^{-6}$ or smaller. On 2-dimensional meshes, the direct solver
is faster. On 3-dimensional meshes, the iterative solver is faster.
The fundamental reason for the difference is the size of the smallest
balanced vertex separators in 2- and 3-dimensional meshes. In a 2-dimensional
mesh with a bounded aspect ratio, the smallest balanced vertex separator
has $\Theta(n^{1/2})$ vertices; in a 3-dimensional mesh, the size
of the smallest separator is $\Theta(n^{2/3})$. The larger separators
in 3-dimensional meshes cause the direct solver to be expensive, but
they speed up the iterative solver. This is shown in Figure~\ref{fig: meshes minres iterations}:
\noun{minres} converges after fewer iterations on 2-dimensional meshes
than on 3-dimensional meshes. The full explanation of this behavior
is beyond the scope of this chapter, but this phenomenon shows how
different direct and iterative solvers are. Something that hurts a
direct solver may help an iterative solver. 

\begin{figure}
\begin{centering}
\includegraphics[width=0.45\columnwidth]{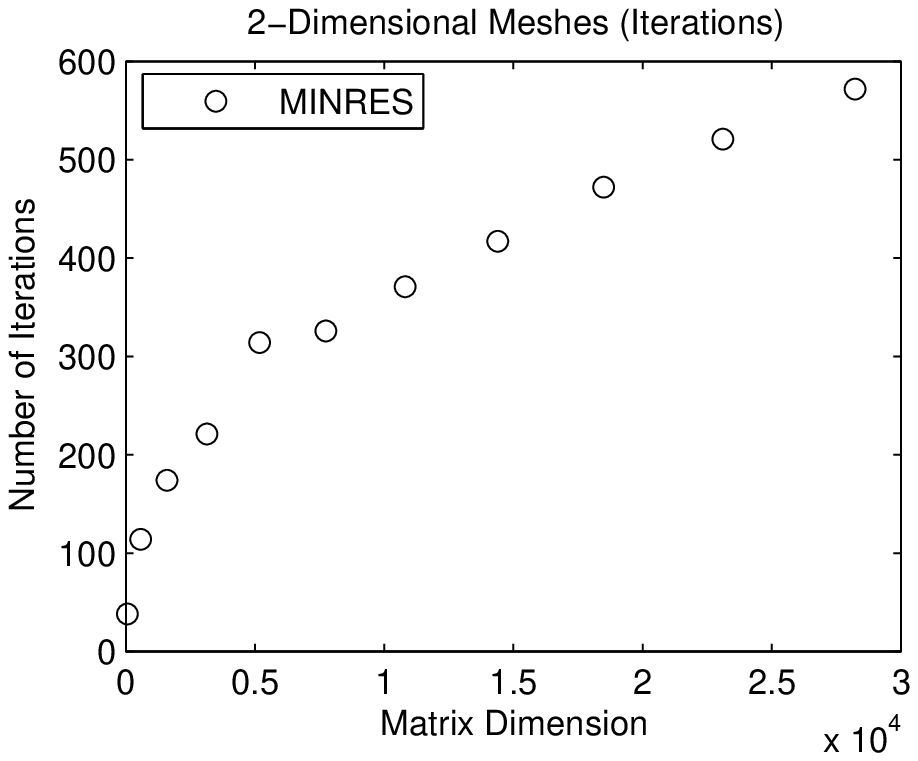}\hfill{}\includegraphics[width=0.45\columnwidth]{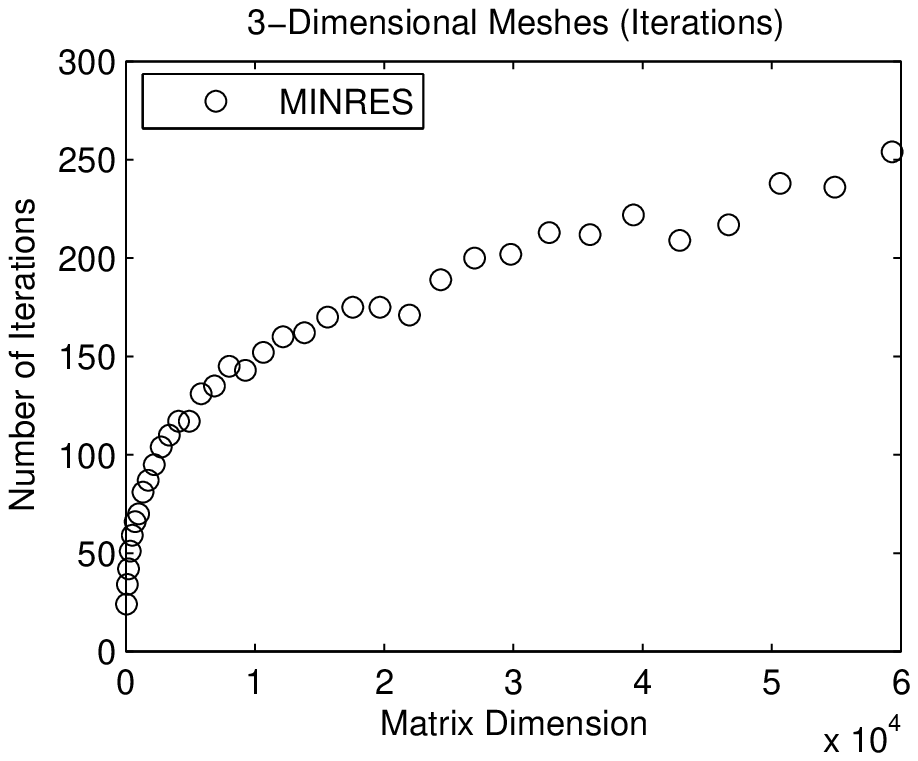}
\par\end{centering}
\caption{\label{fig: meshes minres iterations}The number of iterations required
to reduce the relative norm of the residual to $10^{6}$ or less.}
\end{figure}

As long as $A$ is symmetric and positive definite, the performance
of a sparse Cholesky solver does not depend on the values of the nonzero
elements of $A$, only on $A$'s nonzero pattern. The performance
of an iterative solver, on the other hand, depends on the values of
the nonzero elements. The behavior of \noun{minres} displayed in Figures~\ref{fig: meshes direct vs minres}
and~\ref{fig: meshes minres iterations} depend on our way of assigning
values to nonzero elements of $A$.

Comparing the running times of sparse and iterative solvers does not
reveal all the differences between them. The graphs that present the
number of arithmetic operations, in Figure~\ref{fig: meshes direct vs minres},
reveal another aspect. The differences in the amount of work that
the solvers perform are much greater than the differences in running
times. This indicates that the direct solver performs arithmetic at
a much higher rate than the iterative solver. The exact arithmetic-rate
ratio is implementation dependent, but qualitatively, direct solvers
usually run at higher computational rates than iterative solvers.
Most of the difference is due to the fact that it is easier to exploit
complex computer architectures, in particular cache memories, in direct
solvers.

Another important difference between direct and iterative solvers
is memory usage. A solver like \noun{minres} only allocates a few
$n$-vectors. A direct solver needs more memory, because, as we have
seen, the factors fill. The Cholesky factor of a 3-dimensional mesh
contains $\Omega(n^{4/3})$ nonzeros, so the memory-usage ratio between
a direct and iterative solver worsens as the mesh grows.

\section{Preconditioning}

How can we improve the speed of Krylov-subspace iterative solvers?
We cannot significantly reduce the cost of each iteration, because
for matrices as sparse as ours, the cost of each iteration is only
$\Theta(n)$. There is little hope for finding a better approximation
$x^{(k)}$ in the span of $b,Ab,\ldots,A^{k-1}b$, because \noun{minres}
already finds the approximation that minimizes the norm of the residual.

To understand what \emph{can} be improved, we examine simpler iterative
methods. Given an approximate solution $x^{(k)}$ to the system $Ax=b$,
we define the error $e^{(k)}=x-x^{(k)}$ and the residual $r^{(k)}=b-Ax^{(k)}$.
The error $e^{(k)}$ is the required \emph{correction} to $x^{(k)}$;
if we add $e^{(k)}$ to $x^{(k)}$, we obtain $x$. We obviously do
not know what $e^{(k)}$ is, but we know what $Ae^{(k)}$ is: it is
$r^{(k)}$, because
\begin{eqnarray*}
Ae^{(k)} & = & A\left(x-x^{(k)}\right)\\
 & = & Ax-Ax^{(k)}\\
 & = & b-Ax^{(k)}\\
 & = & r^{(k)}\;.
\end{eqnarray*}
Computing the correction requires solving a linear system of equation
whose coefficient matrix is $A$, which is what we are trying to do
in the first place. The situation seems cyclic, but it is not. The
key idea is to realize that even an inexact correction can bring us
closer to $x$. Suppose that we had another matrix $B$ that is close
to $A$ and such that linear systems of the form $Bz=r$ are much
earsier to solve than systems whose coefficient matrix is $A$. Then
we could solve $Bz^{(k)}=r^{(k)}$ for a correction $z^{(k)}$ and
set $x^{(k+1)}=x^{(k)}+z^{(k)}=x^{(k)}+B^{-1}(b-Ax^{(k)})$. Since
$Ae^{(k)}=b-Ax^{(k)}$, we can express $e^{(k)}$ directly in terms
of $e^{(0)}$, 
\begin{eqnarray*}
e^{(k+1)} & = & x-x^{(k+1)}\\
 & = & x-x^{(k)}+B^{-1}\left(b-Ax^{(k)}\right)\\
 & = & e^{(k)}+B^{-1}Ae^{(k)}\\
 & = & \left(I-B^{-1}A\right)e^{(k)}\\
 & = & \left(I-B^{-1}A\right)^{k}e^{(0)}\;.
\end{eqnarray*}
Therefore, if $\left(I-B^{-1}A\right)^{k}$ converges quickly to a
zero matrix, then $x^{(k)}$ quickly towards $x$. This is the precise
sense in which $B$ must be close to $A$: $\left(I-B^{-1}A\right)^{k}$
should converge quickly to zero. This happens if the spectral radius
of $\left(I-B^{-1}A\right)$ is small (the spectral radius of a matrix
is the maximal absolute value of its eigenvalues).

The convergence of more sophisticaded iterative methods, like \noun{minres}
and \noun{cg}, can also be accelerated using this technique, called
\emph{preconditioning}. In every iteration, we solve a correction
equation $Bz^{(k)}=r^{(k)}$. The coefficient matrix $B$ of the correction
equation is called a \emph{preconditioner}. The convergence of preconditioned
\noun{minres} and preconditioned \noun{cg} also depends on spectral
(eigenvalue) properties associated with $A$ and $B$. For these methods,
it is not the spectral radius that determines the convergence rate,
but other spectral properties. But the principle is the same: $B$
should be close to $A$ in some spectral metric.

In this book, we mainly focus on one family of methods to construct
and analyze preconditioners. There are many other ways. The constructions
that we present are called \emph{support preconditioners}, and our
analysis method is called \emph{support theory}.

We introduce support preconditioning here using one particular family
of support preconditioners for our meshes. These preconditioners,
which are only effective for regular meshes whose edges all have the
same weights (off-diagonals in $A$ all have the same value), are
called \emph{Joshi} preconditioners. We construct a Joshi preconditioner
$B$ for a mesh matrix $A$ by dropping certain edges from the mesh,
and then constructing a coefficient matrix $B$ for the sparsified
mesh using the procedure given in Section~\ref{sec:model meshes}.
Figures~\ref{fig: joshi 2d} and~\ref{fig: 3-by-4-by-3 joshi mesh}
present examples of such sparsified meshes. 

\begin{figure}
\begin{centering}
\includegraphics[width=0.45\columnwidth]{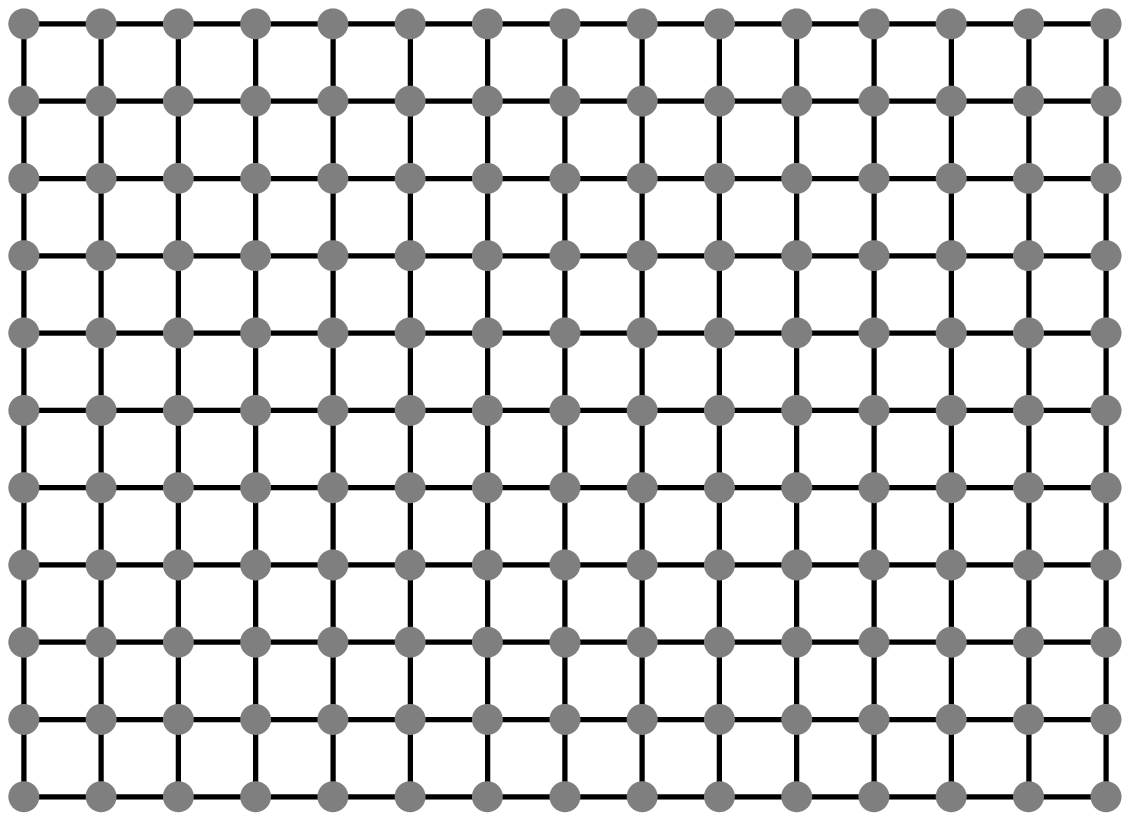}\hfill{}\includegraphics[width=0.45\columnwidth]{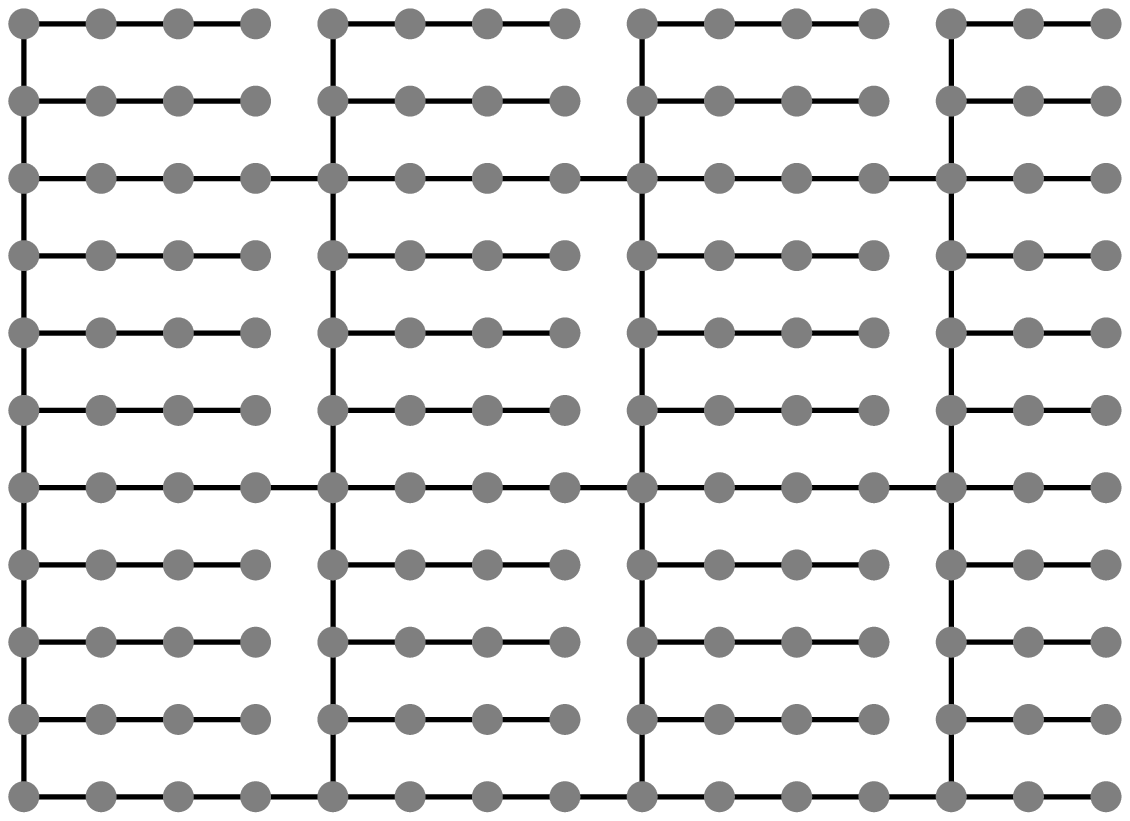}\\
~\\
\includegraphics[width=0.45\columnwidth]{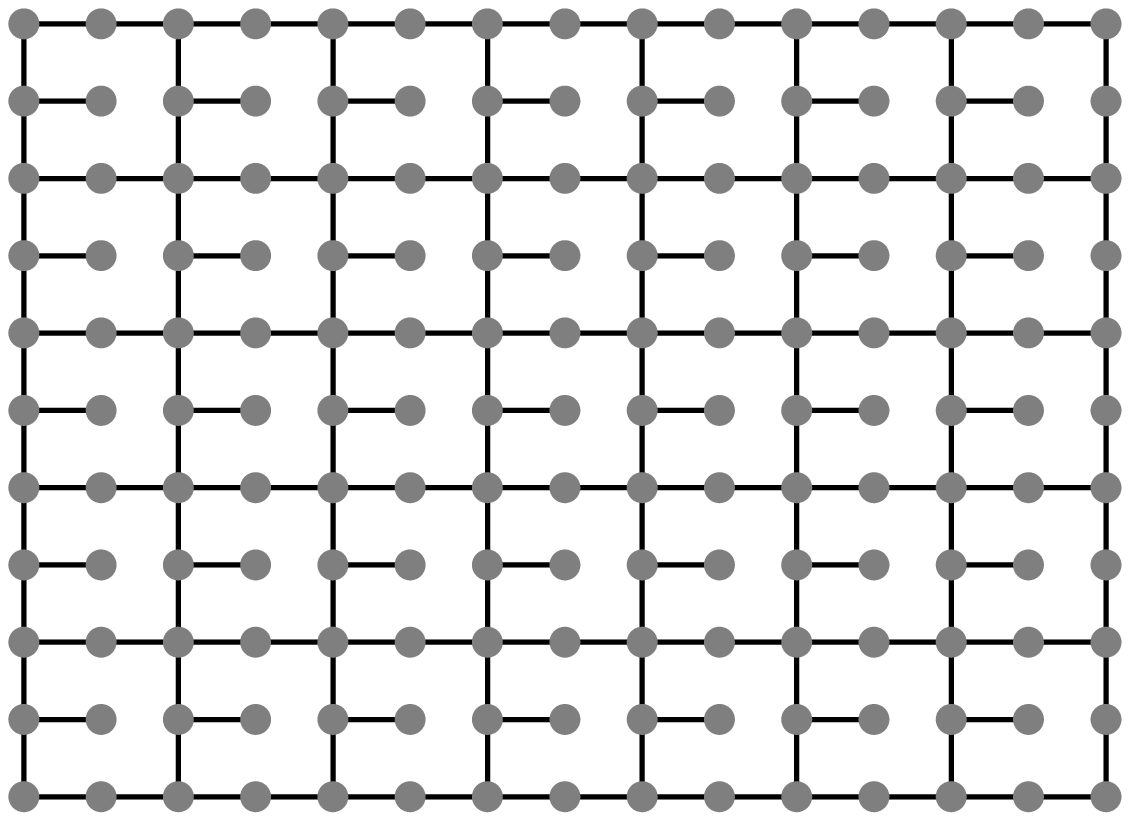}\hfill{}\includegraphics[width=0.45\columnwidth]{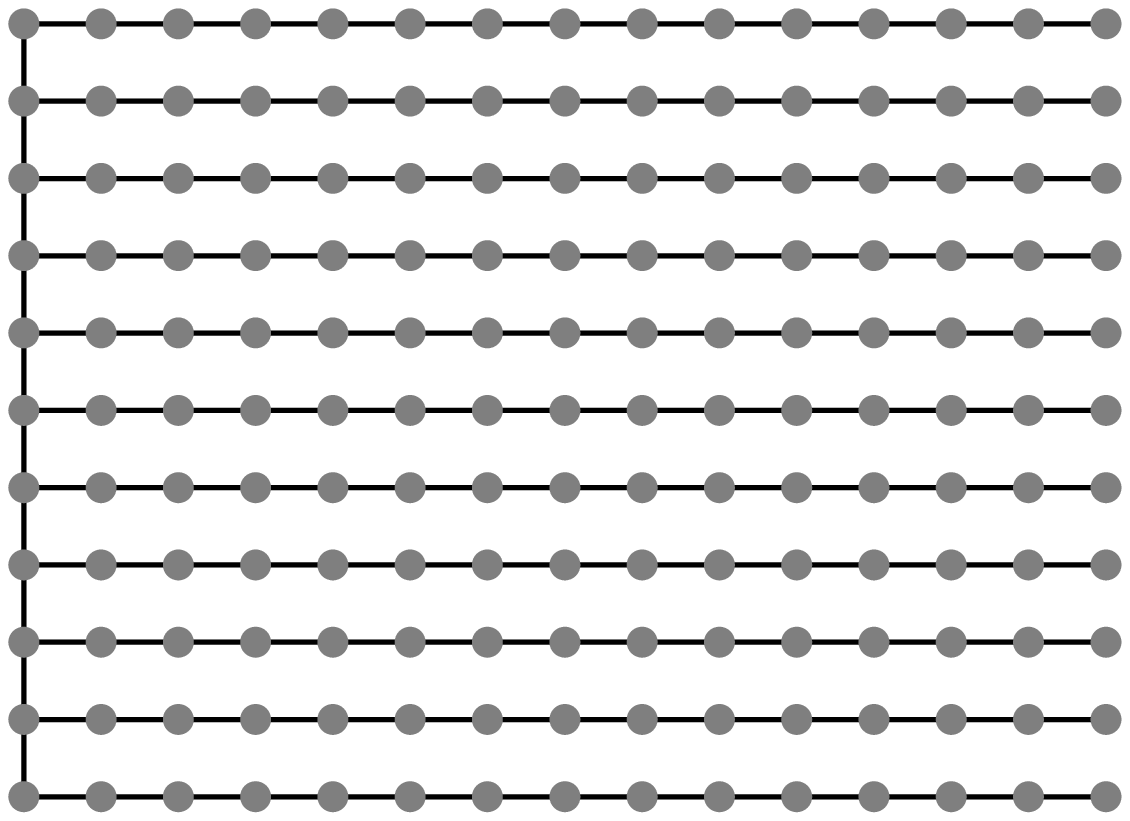}
\par\end{centering}
\caption{\label{fig: joshi 2d}A $15$-by-$11$ two-dimensional mesh (top left)
with Joshi subgraphs for $k=4$ (top right), $k=2$ (bottom left),
and $k=15$ (bottom right). The subgraph for $k=2$ is the densest
nontrivial subgraph and the subgraph for $k=15$ is the sparsest possible;
it is a spanning tree of the mesh.}
\end{figure}
\begin{figure}
\begin{centering}
\includegraphics[width=0.8\columnwidth]{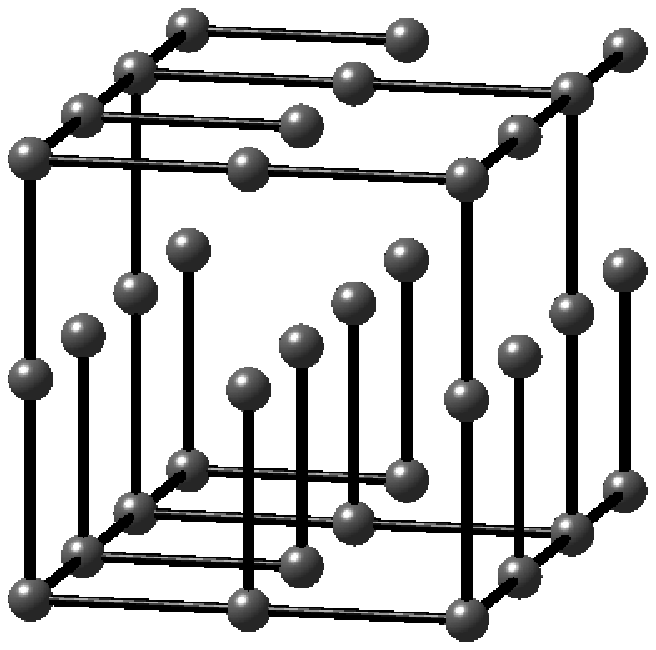}
\par\end{centering}
\caption{\label{fig: 3-by-4-by-3 joshi mesh}A $3$-by-$4$-by-$3$ three-dimensional
Joshi mesh.}
\end{figure}

To solve the correction equations during the iterations, we compute
the sparse Cholesky factorization of $PBP^{T}$, where $P$ is a fill-reducing
permutation $P$ for $B$.

The construction of Joshi preconditioners depends on a parameter $k$.
For small $k$, the mesh of the preconditioner is similar to the mesh
of $A$. In particular, for $k=1$ we obtain $B=A$. As we shall see,
Joshi preconditioners with a small $k$ lead to fast convergence rates.
But Joshi preconditioners with a small $k$ do not have small balanced
vertex separators, so their Cholesky factorization suffers from significant
fill. Not as much as in the Cholesky factor of $A$, but still significant.
This fill slows down the preconditioner-construction phase of the
linear solver, in which we construct and factor $B$, and it also
slows down each iteration. As $k$ gets larger, the similarity between
the meshes diminishes. Convergence rates are worse than for small
$k$, but the factorization of $B$ and the solution of the correction
equation in each iteration become cheaper. In particular, when $k$
is at least as large as the side of the mesh, the mesh of $B$ is
a tree, which has single-vertex balanced separators; factoring it
costs only $\Theta(n)$ operations, and its factor contains only $\Theta(n)$
nonzeros.

\begin{figure}
\begin{centering}
\includegraphics[width=0.45\columnwidth]{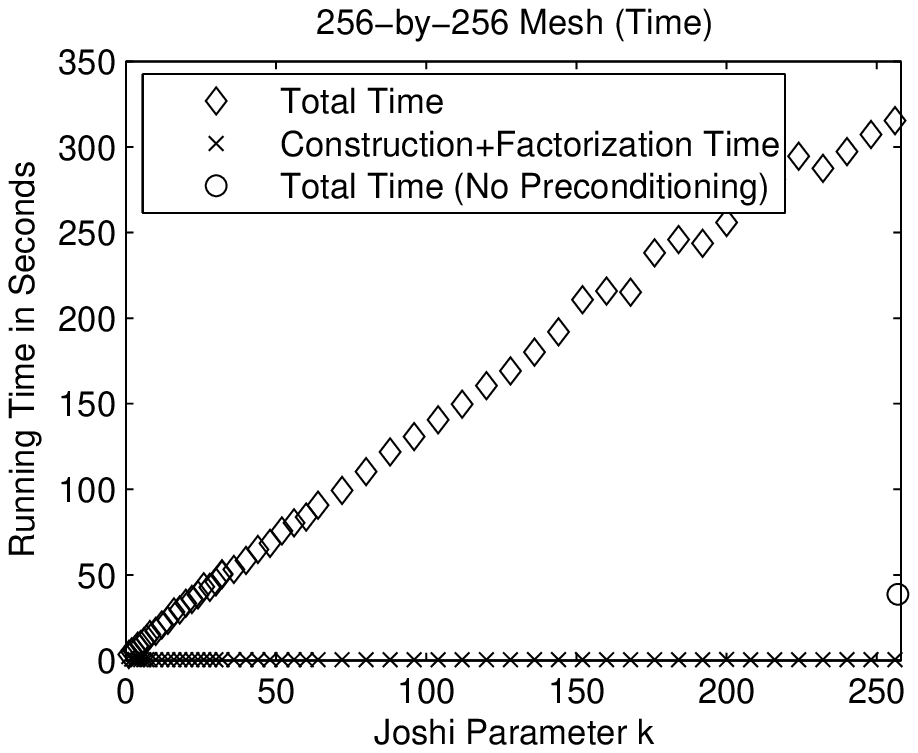}\hfill{}\includegraphics[width=0.45\columnwidth]{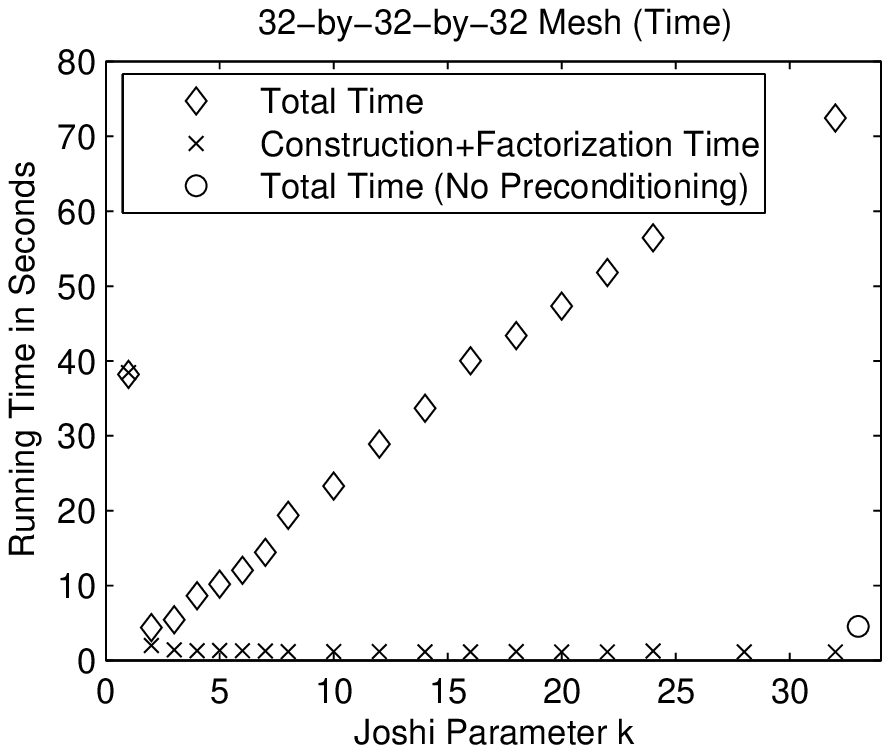}\\
\includegraphics[width=0.45\columnwidth]{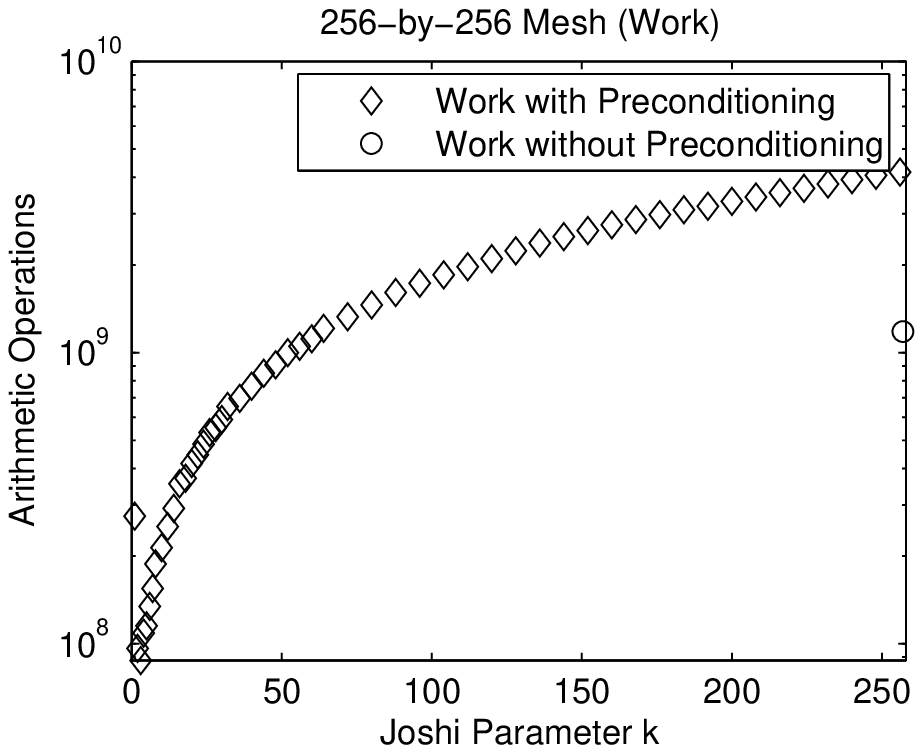}\hfill{}\includegraphics[width=0.45\columnwidth]{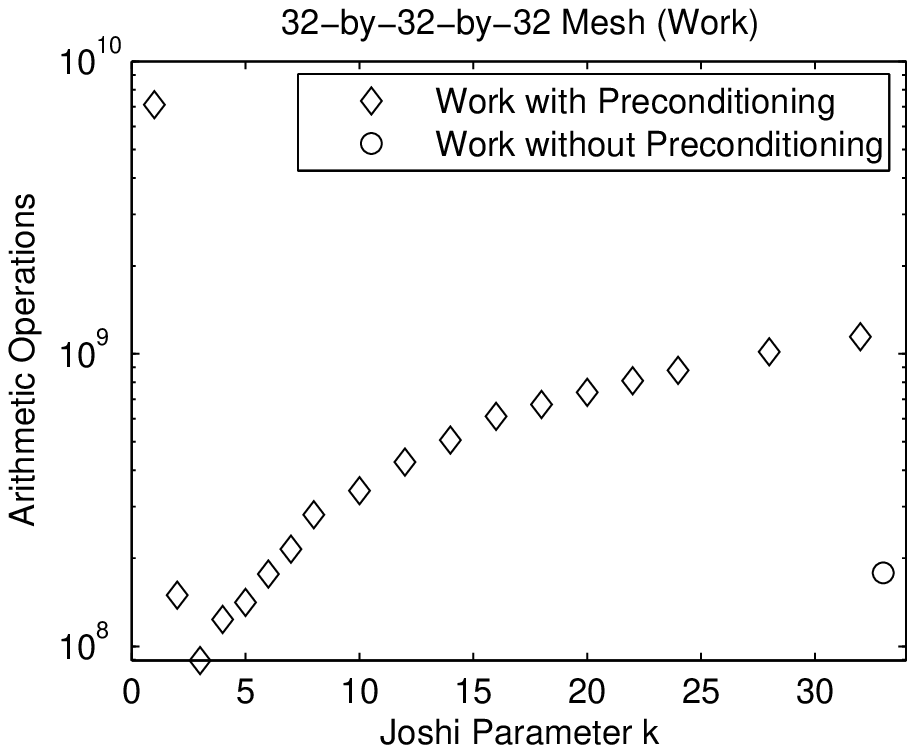}\\
\includegraphics[width=0.45\columnwidth]{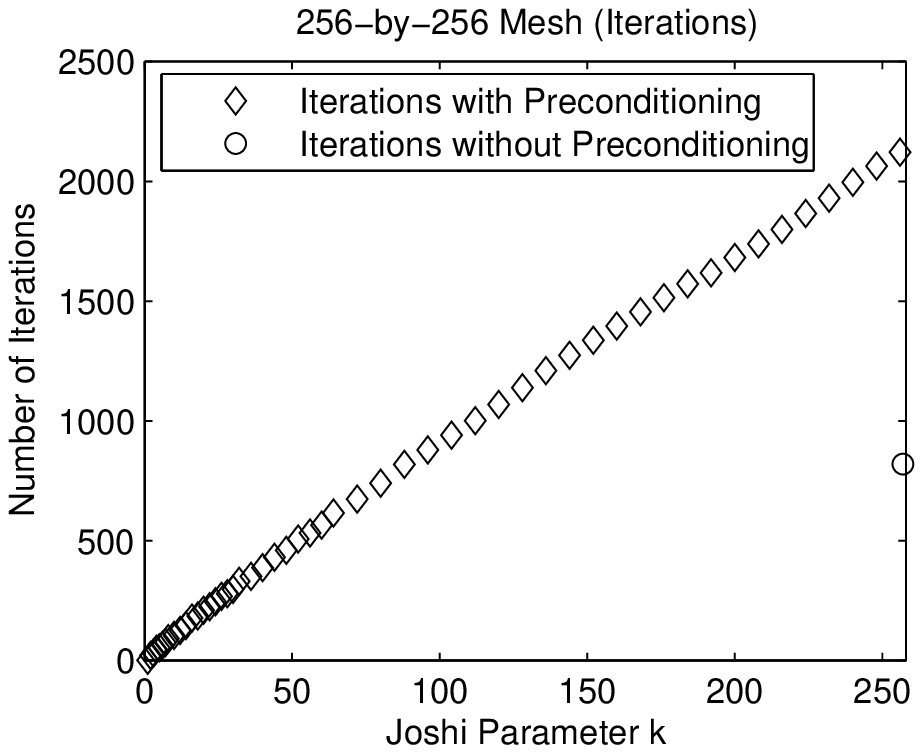}\hfill{}\includegraphics[width=0.45\columnwidth]{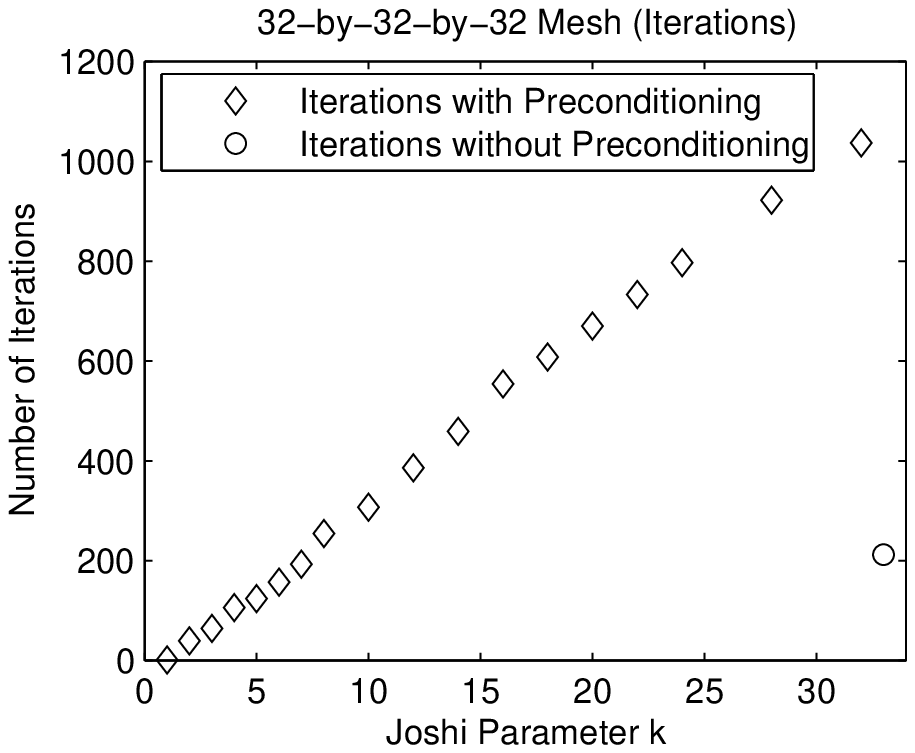}
\par\end{centering}
\caption{\label{fig: meshes joshi}The behavior of preconditioned \noun{minres}
solvers with Joshi preconditioners. For $k=1$, the performance is
essentially the performance of the direct solver. The preconditioners
were reordered using a minimum degree ordering and factored.}
\end{figure}

Figure~\ref{fig: meshes joshi} presents the performance of preconditioned
\noun{minres} with Joshi preconditioners on two meshes. On the three
dimensional mesh, the preconditioned iterative solver is much faster
than the direct solver and about as fast as the unpreconditioned \noun{minres}
solver. The preconditioned solver performs fewer arithmetic operations
than both the direct solver and the unpreconditioned solver. In terms
of arithmetic operations, the Joshi preconditioner for $k=3$ is the
most efficient. For small $k$, the preconditioned solvers perform
fewer iterations than the unpreconditioned solver. The number of iteration
rises with $k$. Perhaps the most interesting aspect of these graphs
is the fact that the cost of the factorization of $B$ drops dramatically
from $k=1$, where $B=A$, to $k=2$. Even slightly sparsifying the
mesh reduces the cost of the factorization dramatically. The behavior
on the 2-dimensional mesh is similar, except that in terms of time,
the cost of the solver rises monotonically with $k$. In particular,
the direct solver is fastest. But in terms of arithmetic operations,
the preconditioned solver with $k=2$ is the most efficient.

\section{The Main Questions}

We have now set the stage to pose the main questions in support theory
and preconditioning. The main theme in support preconditioning is
the construction of preconditioners using graph algorithms. Given
a matrix $A$, we model it as a graph $G_{A}$. From the $G_{A}$
we build another graph $G_{B}$, which approximates $G_{A}$ in some
graph-related metric. From $G_{B}$ we build the preconditioner $B$.
The preconditioner is usually factored as preparation for solving
correction equations $Bz=r$, but sometimes the correction equation
is also solved iteratively. This raises three questions
\begin{enumerate}
\item How do we model matrices as graphs?
\item Given two matrices $A$ and $B$ and their associated graphs $G_{A}$
and $G_{B}$, how do we use the graphs to estimate or bound the convergence
of iterative linear solvers in which $B$ preconditions $A$?
\item Given the graph $G_{A}$ of a matrix $A$, how do we construct a graph
$G_{B}$ in a metric that ensures fast convergence rates?
\end{enumerate}
Support theory addresses all three questions: the modeling question,
the analysis qustion, and the graph-approximation question.

The correspondence between graphs and matrices is usually an isomorphism,
to allow the analysis to yield useful convergence-rate bounds. It
may seem useless to transform one problem, the matrix-approximation
probem, into a completely equivalent one, a graph-approximation question.
It turns out, however, that many effective constructions only make
sense when the problem is described in terms of graphs. They would
not have been invented as matrix constructions. The Joshi preconditioners
are a good example. They are highly structured when viewed as graph
constructions, but would make little sense as direct sparse-matrix
constructions. The analysis, on the other hand, is easier, more natural,
and more general when carried out in terms of matrices. And so we
transform the orignal matrix to a graph to allow for sophisticaed
constructions, and then we transform the graphs back to matrices to
analyze the convergence rate.

The preconditioning framework that we presented is the main theme
of support theory, but there are other themes. The combinatorial and
algebraic tools that were invented in order to develop preconditioners
have also been used to analyze earlier preconditioners, to bound the
extreme eigenvalues of matrices, to analyze the null spaces of matrices,
and more.

%% file: chapter2.tex
 \setcounter{chapter}{1}

\chapter{\label{ch:iterative}Iterative Krylov-Subspace Solvers}

Preconditioned Krylov-subspace iterations are a key ingredient in
many modern linear solvers, including in solvers that employ support
preconditioners. This chapter presents Krylov-subspace iterations
for symmetric semidefinite matrices. In particular, we analyze the
convergence behavior of these solvers. Understanding what determines
the convergence rate is key to designing effective preconditioners.

\section{The Minimal Residual Method}

The minimal-residual method (\noun{minres}) is an iterative algorithm
that finds in each iteration the vector $x^{(t)}$ that minimizes
the residual $\|Ax^{(t)}-b\|_{2}$ in a subspace $\mathcal{K}_{t}$
of $\mathbb{R}^{n}$. These subspaces, called the Krylov subspaces,
are nested, $\mathcal{K}_{t}\subseteq\mathcal{K}_{t+1}$, and the
dimension of the subspace usually grows by one in every iteration,
so the accuracy of the approximate solution $x^{(t)}$ tends to improve
from one iteration to the next. The construction of the spaces $\mathcal{K}_{t}$
is designed to allow an efficient computation of the approximate solution
in every iteration.
\begin{defn}
The \emph{Krylov subspace} $\mathcal{K}_{t}$ is the subspace that
is spanned by the columns of the matrix $K_{t}=\begin{bmatrix}b & Ab & A^{2}b & \cdots & A^{t-1}\end{bmatrix}$.
That is, 
\[
\mathcal{K}_{t}=\left\{ K_{t}y:y\in\mathbb{R}^{t}\right\} \;.
\]
\end{defn}

Once we have a basis for $\mathcal{K}_{t}$, we can express the minimization
of $\|Ax-b\|_{2}$ in $\mathcal{K}_{t}$ as an unconstrained linear
least-squares problem, since
\[
\min_{x\in\mathcal{K}_{t}}\|Ax-b\|_{2}=\min_{y\in\mathbb{R}^{t}}\|AK_{t}y-b\|_{2}\;.
\]

There are three fundamental tools in the solution of least squares
problems with full-rank $m$-by-$n$ coefficient matrices with $m\geq n$.
The first tool is unitary transformations. A unitary matrix $Q$ preserves
the $2$-norm of vectors, $\|Qx\|_{2}=\|x\|_{2}$ for any $x$ (a
square matrix with orthonormal columns, or equivalently, a matrix
such that $QQ^{*}=I$). This allows us to transform least squares
problem into equivalent forms
\begin{equation}
\min_{x}\|Ax-b\|_{2}=\min_{y}\|QAx-Qb\|_{2}\;.\label{eq:Krylov least squares formulation}
\end{equation}
An insight about matrices that contain only zeros in rows $n+1$ through
$m$ is the second tool. Consider the least squares problem 
\[
\min_{x}\left\Vert \begin{bmatrix}R_{1}\\
0
\end{bmatrix}x-\begin{bmatrix}b_{1}\\
b_{2}
\end{bmatrix}\right\Vert _{2}\;.
\]
If the coefficient matrix has full rank, then the square block $R_{1}$
must be invertible. The solution to the problem is the vector $x$
that minimizes the Euclidean distance between $\left[\begin{smallmatrix}R_{1}\\
0
\end{smallmatrix}\right]x$ and $\left[\begin{smallmatrix}b_{1}\\
b_{2}
\end{smallmatrix}\right]$ is minimal. But it is clear what is this vector: $x=R_{1}^{-1}b_{1}$.
This vector yields $\left[\begin{smallmatrix}R_{1}\\
0
\end{smallmatrix}\right]x=\begin{bmatrix}b_{1}\\
0
\end{bmatrix}$, and we clearly cannot any closer. The third tool combines the first
two into an algorithm. If we factor $A$ into a product $QR$ of a
unitary matrix $Q$ and an upper trapezoidal matrix $R=\begin{bmatrix}R_{1}^{*} & 0\end{bmatrix}^{*}$,
which is always possible, then we can multiply the system by $Q^{*}$
to obtain
\begin{eqnarray*}
\min_{x}\left\Vert Ax-b\right\Vert _{2} & = & \min_{x}\left\Vert QRx-b\right\Vert _{2}\\
 & = & \min_{x}\left\Vert Rx-Q^{*}b\right\Vert _{2}\\
 & = & \min_{x}\left\Vert \begin{bmatrix}R_{1}\\
0
\end{bmatrix}x-Q*b\right\Vert _{2}\;.
\end{eqnarray*}
We can now compute the minimizer $x$ by substitution. Furthermore,
the norm of the residual is given by $\|(Q^{*}b)_{n+1\colon m}\|_{2}$.

In principle, we could solve (\ref{eq:Krylov least squares formulation})
for $y$ using a $QR$ factorization of $AK_{t}$. Once we find $y$,
we can compute $x=K_{t}y$. There are, however, two serious defects
in this approach. First, it is inefficient, because $AK_{t}$ is dense.
Second, as $t$ grows, $A^{t}b$ tends to the subspace spanned by
the dominant eigenvectors of $A$ (the eigenvectors associated with
the eigenvalues with maximal absolute value). This causes $K_{t}$
to become ill conditioned; for some vectors $x\in\mathcal{K}_{t}$
with $\|x\|_{2}=1$, the vector $y$ such that $x=K_{t}y$ has huge
elements. This phenomenon will happen even if we normalize the columns
of $K_{t}$ to have unit norm, and it causes instabilities when the
computation is carried out using floating-point arithmetic. 

We need a better basis for $\mathcal{K}_{t}$ for stability, and we
need to exploit the special properties of $AK_{t}$ for efficiency.
\noun{Minres} does both. It uses a stable basis that can be computed
efficiently, and it it combines the three basic tools in a clever
way to achieve efficiency.

An orthonormal basis $Q_{t}$ for $\mathcal{K}_{t}$ would work better,
because for $x=Q_{t}z$ we would have $\|z\|_{2}=\|x\|_{2}$. There
are many orthonormal bases for $\mathcal{K}_{t}$. We choose a particular
basis that we can compute incrementally, one basis column in each
iteration. The basis that we use consists of the columns of the $Q$
factor from the $QR$ factorization of $K_{t}=Q_{t}R_{t}$, where
$Q_{t}$ is orthonormal and $R_{t}$ is upper triangular. We compute
$Q_{t}$ using Gram-Schmidt orthogonalization. The Gram-Schmidt process
is not always numerically stable when carried out in floating-point
arithmetic, but it is the only efficient way to compute $Q_{t}$ one
column at a time.

We now show how to compute, given $Q_{t}=\begin{bmatrix}q_{1} & \cdots & q_{t}\end{bmatrix}$,
the next basis column $q_{t+1}$. The key to the efficient computation
of $q_{t+1}$ is the special relationship of the matrices $Q_{t}$
with $A$. We assume by induction that $K_{t}=Q_{t}R_{t}$ and that
$r_{t,t}\neq0$. If $r_{t,t}=0$, then it is not hard to show that
the exact solution $x$ is in $\mathcal{K}_{t}$, so we would have
found it in one of the previous iterations. If we reached iteration
$t$, then $x\not\in\mathcal{K}_{t}$, so $r_{t,t}\neq0$. Expanding
the last column of $K_{t}=Q_{t}R_{t}$, we get
\[
A^{t-1}b=r_{1,t}q_{1}+r_{2,t}q_{2}+\cdots+r_{t-1,t}q_{t-1}+r_{t,t}q_{t}\;.
\]
We now isolate $q_{t}$ to obtain
\[
q_{t}=r_{t,t}^{-1}\left(A^{t-1}b-r_{1,t}q_{1}-r_{2,t}q_{2}-\cdots-r_{t-1,t}q_{t-1}\right)\;,
\]
so
\[
Aq_{t}=r_{t,t}^{-1}\left(A^{t}b-r_{1,t}Aq_{1}-r_{2,t}Aq_{2}-\cdots-r_{t-1,t}Aq_{t-1}\right)\;.
\]
Because $q_{t}\in\mathcal{K}_{t}$, clearly $Aq_{t}\in\mathcal{K}_{t+1}$.
Furthermore, if we reached iteration $t+1$, then $Aq_{t}\not\in\mathcal{K}_{t}$,
because if $Aq_{t}\in\mathcal{K}_{t}$ then $x\in\mathcal{K}_{t}$.
Therefore, we can orthogonalize $Aq_{t}$ with respect to $q_{1},\dots,q_{t}$
and normalize to obtain $q_{t+1}$ 
\begin{eqnarray*}
\tilde{q}_{t+1} & = & Aq_{t}-\left(q_{t}^{*}Aq_{t}\right)q_{t}-\cdots-\left(q_{1}^{*}Aq_{t}\right)q_{1}\\
q_{t+1} & = & \tilde{q}_{t+1}/\left\Vert \tilde{q}_{t+1}\right\Vert _{2}\;.
\end{eqnarray*}

These expressions allows us not only to compute $q_{t}$, but also
to express $Aq_{t}$ as a linear combination of $q_{1},\ldots q_{t},q_{t+1}$,
\begin{eqnarray*}
Aq_{t} & = & \left\Vert \tilde{q}_{t+1}\right\Vert _{2}q_{t+1}+\left(q_{t}^{*}Aq_{t}\right)q_{t}+\cdots+\left(q_{1}^{*}Aq_{t}\right)q_{1}\\
 & = & h_{t+1,t}q_{t+1}+h_{t,t}q_{t}+\cdots+h_{1,t}q_{1}\;,
\end{eqnarray*}
where $h_{t+1,t}=\|\tilde{q}_{t+1}\|_{2}$, and where $h_{j,t}=q_{j}^{*}Aq_{t}$
for $j\leq t$. The same argument also holds for $Aq_{1},\ldots,Aq_{t-1}$,
so in matrix terms,
\[
AQ_{t}=Q_{t+1}\tilde{H_{t}}\;.
\]
The matrix $\tilde{H}_{t}\in\mathbb{R}^{(t+1)\times t}$ is upper
Hessenberg: in column $j$, rows $j+2$ to $t+1$ are zero. If we
multiply both sides of this equation from the left by $Q_{t}^{*}$,
we obtain
\begin{eqnarray*}
Q_{t}^{*}AQ_{t} & = & Q_{t}^{*}Q_{t+1}\tilde{H_{t}}\\
 & = & \left(\begin{array}{ccc}
1\\
 & \ddots\\
 &  & 1\\
0 & \cdots & 0
\end{array}\right)\tilde{H_{t}}\\
 & = & \left(\tilde{H_{t}}\right)_{1\colon t,1\colon t}\;.
\end{eqnarray*}
We denote the first $t$ rows of $\tilde{H_{t}}$ by $H_{t}$ to obtain
\[
Q_{t}^{*}AQ_{t}=H_{t}\;.
\]
Because $A$ is symmetric, $H_{t}$ must be symmetric, and hence tridiagonal.
The symmetry of $H_{t}$ implies that in column $j$ of $H_{t}$ and
$\tilde{H}_{t}$, rows $1$ through $j-2$ are also zero. This is
the key to the efficiency of \noun{minres} and similar algorithms,
because it implies that for $j\leq t-2$, $h_{j,t}=q_{j}^{*}Aq_{t}=0$;
we do not need to compute these inner products and we do not need
to subtract $\left(q_{j}^{*}Aq_{t}\right)q_{j}$ from $Aq_{t}$ in
the orthogonalization process. As we shall shortly see, we do not
even need to store $q_{j}$ for $j\leq t-2$.

Now that we have an easy-to-compute orthonormal basis for $\mathcal{K}_{t}$,
we return to our least-squares problem
\begin{eqnarray*}
\min_{x\in\mathcal{K}_{t}}\|Ax-b\|_{2} & = & \min_{y\in\mathbb{R}^{t}}\|AK_{t}y-b\|_{2}\\
 & = & \min_{z\in\mathbb{R}^{t}}\|AQ_{t}z-b\|_{2}\\
 & = & \min_{z\in\mathbb{R}^{t}}\|Q_{t+1}\tilde{H_{t}}z-b\|_{2}\;.
\end{eqnarray*}
Our strategy now is to compute the minimizer $z$ from the expression
in the last line. Once we compute $z$, the minimizer $x$ in $\mathcal{K}_{t}$
is $x=Q_{t}z$. To compute the minimizer, we use the equality
\begin{eqnarray*}
\arg\min_{z\in\mathbb{R}^{t}}\|Q_{t+1}\tilde{H}_{t}z-b\|_{2} & = & \arg\min_{z\in\mathbb{R}^{t}}\|\tilde{H}_{t}z-Q_{t+1}^{*}b\|_{2}\\
 & = & \arg\min_{z\in\mathbb{R}^{t}}\|\tilde{H}_{t}z-\|b\|_{2}e_{1}\|_{2}\;,
\end{eqnarray*}
where $e_{1}$ is the first unit vector. The equality $Q_{t+1}^{*}b=\|b\|_{2}e_{1}$
holds because $Q_{t+1}$ is the orthogonal factor in the $QR$ factorization
of $K_{t+1}$, whose first column is $b$. To solve this least-squares
problem, we will use a $QR$ factorization $\tilde{H}_{t}=V_{t}U_{t}$
of $\tilde{H}_{t}$: the minimizer $z$ is then defined by $U_{t}z=\|b\|_{2}V_{t}^{*}e_{1}$.
Because $\tilde{H}_{t}\in\mathbb{R}^{t\times t}$ is triadiagonal
we can compute its $QR$ factorization with a sequence of of $t-1$
Givens rotations, where the $i$th rotation transforms rows $i$ and
$i-1$ of $\tilde{H}_{t}$ and of $\|b\|_{2}e_{1}$. A Given rotation
is a unitary matrix of the form 
\[
\begin{bmatrix}1\\
 & \ddots\\
 &  & 1\\
 &  &  & \cos\theta & \sin\theta\\
 &  &  & -\sin\theta & \cos\theta\\
 &  &  &  &  & 1\\
 &  &  &  &  &  & \ddots\\
 &  &  &  &  &  &  & 1
\end{bmatrix}\;.
\]
We choose $\theta$ so as to anihilate the subdiagonal element in
column $i$ of $\tilde{H}_{t}$.

One difficulty that arises is that $z=z_{t}$ changes completely in
every iteration, so to form $x=Q_{t}z$, we need to either store all
the columns of $Q_{t}$ to produce the approximate solution $x$,
or to recompute $Q_{t}$ again once we obtain a $z$ that ensures
a small-enough residual. Fortunately, there is a cheaper alternative
that only requires storing a constant number of vectors. Let $M_{t}=Q_{t}U_{t}^{-1}$,
and let $w=U_{t}z$. Instead of solving for $z$, we only compute
$w=\|b\|_{2}V_{t}^{*}e_{1}$. Because the $i$the rotation only transforms
rows $i-1$ and $i$ of $\|b\|_{2}e_{1}$, we can compute $w$ one
entry at a time. Since $x=Q_{t}z=Q_{t}U_{t}^{-1}U_{t}z=M_{t}w$, we
can accumulate $x$ using the columns of $M_{t}$. We compute the
columns of $M_{t}$ one at a time using the triangular linear system
$U_{t}M_{t}=Q_{t}$; in iteration $t$, we compute the last column
of $M_{t}$ from $U_{t}$ and the last column of $Q_{t}$.

\begin{figure}
\noindent \begin{raggedright}
\noun{minres}($A,b$)\\
~~$q_{1}=b/\|b\|_{2}$\hfill{}$\triangleright$ the first column
of $Q_{t}$\\
~~$w_{1}=\|b\|_{2}$\hfill{}$\triangleright$ the first element
of the vector $w$\\
~~for $t=1,2,\ldots$ until convergence\\
~~~~compute $Aq_{t}$\\
~~~~$\tilde{q}_{t+1}=Aq_{t}-\left(q_{t}^{*}Aq_{t}\right)q_{t}-\left(q_{t-1}^{*}Aq_{t}\right)q_{t-1}$\\
~~~~$H_{t+1,t}=\|\tilde{q}_{t+1}\|_{2}$\\
~~~~$H_{t,t}=q_{t}^{*}Aq_{t}$\\
~~~~$H_{t-1,t}=q_{t-1}^{*}Aq_{t}$\\
~~~~$q_{t+1}=\tilde{q}_{t+1}/\left\Vert \tilde{q}_{t+1}\right\Vert _{2}$\\
~~~~$\triangleright$ Apply rotation $t-2$ to $H_{\colon,t}$\\
~~~~$R_{t-2,t}=s_{t-2}H_{t-1,t}$ \\
~~~~if $t>2$ then $U_{t-1,t}=c_{t-2}H_{t-1,t}$ else $U_{t-1,t}=H_{t-1,t}$\\
~~~~$\triangleright$ Apply rotation $t-1$ to $H_{\colon,t}$\\
~~~~$U_{t-1,t}=c_{t-1}U_{t-1,t}+s_{t-1}H_{t,t}$ \\
~~~~if $t>1$ then $U_{t,t}=-s_{t-1}U_{t-1,t}+c_{t-1}H_{t,t}$
else $U_{t,t}=H_{t,t}$\\
~~~~compute $\left[\begin{smallmatrix}c_{t} & -s_{t}\\
s_{t} & c_{t}
\end{smallmatrix}\right]$, the Givens rotation such\\
~~~~\hfill{}that $\left[\begin{smallmatrix}c_{t} & -s_{t}\\
s_{t} & c_{t}
\end{smallmatrix}\right]\left[\begin{smallmatrix}U_{t,t}\\
H_{t+1,t}
\end{smallmatrix}\right]=\left[\begin{smallmatrix}\textrm{anything}\\
0
\end{smallmatrix}\right]$\\
~~~~$\triangleright$ Apply rotation $t$ to $H_{\colon,t}$\\
~~~~$U_{t,t}=c_{t}U_{t,t}+s_{t}H_{t+1,t}$\\
~~~~$\triangleright$ Apply rotation $t$ to form $w=U_{t}z=V_{t}^{*}\|b\|e_{1}$\\
~~~~$w_{t+1}=-s_{t}w_{t}$\\
~~~~$w_{t}=c_{t}w_{t}$\\
~~~~$m_{t}=r_{t,t}^{-1}\left(q_{t}-U_{t-1,t}m_{t-1}-U_{t-2,t}m_{t-2}\right)$\hfill{}$\triangleright$
next column of $M$\\
~~~~$x^{(t)}=x^{(t-1)}+w_{t}m_{t}$\\
~~end for
\par\end{raggedright}
\begin{lyxcode}
\end{lyxcode}
\caption{\noun{Minres}. To keep the pseudo-code simple, we use the convention
that vectors and matrix/vector elements with nonpositive indices are
zeros. By this convension, $x^{(0)}=q_{0}=m_{0}=m_{-1}=0$, and so
on. In an actual code, this convension can be implemented either using
explicit zero vectors or using conditionals.}
\end{figure}

\section{The Conjugate Gradients Algorithm}

\noun{Minres} is a variant of an older and more well-known algorithm
called the \emph{Conjugate Gradients} method (\noun{cg}). The Conjugate
Gradients method is only guaranteed to work when $A$ is symmetric
positive definite, whereas \noun{minres} only requires $A$ to be
symmetric. Conjugate Gradients also minimizes the norm of the residuals
over the same Krylov subspaces, but not the $2$-norm but the $A^{-1}$-norm,
\[
\|Ax^{(t)}-b\|_{A^{-1}}=\sqrt{\left(Ax^{(t)}-b\right)^{*}A^{-1}\left(Ax^{(t)}-b\right)}\;.
\]
This is equivalent to minimizing the $A$-norm of the error,
\[
\|x-x^{(t)}\|_{A}=\sqrt{\left(x-x^{(t)}\right)^{*}A\left(x-x^{(t)}\right)}\;.
\]

Minimizing the $A$-norm of the residual may seem like an odd idea,
because of the dependence on $A$ in the measurement of the residual
and the error. But there are several good reasons not to worry. First,
we use the $2$-norm of the residual as a stoping criterion, to stop
the iterations only when the $2$-norm is small enough, not when the
$A^{-1}$-norm is small. Second, even when the stopping criterion
is based on the $2$-norm of the residual, Conjugate Gradients usually
converges only slightly slower than \noun{minres}. Still, the minimization
of the $2$-norm of the residual in \noun{minres} is more elegant.

Why do people use Conjugate Gradients if \noun{minres} is more theoretically
appealing? The main reason that that the matrices $M_{t}=Q_{t}R_{t}^{-1}$
and $R_{t}$ that are used in \noun{minres} to form $x^{(t)}$ can
be ill conditioned. This means that in floating-point arithmetic,
the computed $x^{(t)}$ are not always accurate minimizers. In the
Conjugate Gradients method, the columns of the basis matrix (the equivalent
of $M_{t}$) are $A$-conjugate, not arbitrary. This reduces the inaccuracies
in the computation of the approximate solution in each iterations,
so the method is numerically more stable than \noun{minres}.

We shall not derive the details of the Conjugate Gradient method here.
The derivation is similar to that of \noun{minres} (there are also
other ways to derive \noun{cg}), and is presented in many textbooks.

\section{Convergence-Rate Bounds}

The appeal of Krylov-subspace iterations stems to some extent from
the fact that their convergence is easy to understand and to bound.

The crucial step in the analysis is the expression of the residual
$b-Ax^{(t)}$ as a application of a univariate polynomial $\tilde{p}$
to $A$ and a multiplication of the resulting matrix $\tilde{p}(A)$
by $b$. Since $x^{(t)}\in\mathcal{K}_{t}$, $x^{(t)}=K_{t}y$ for
some $y$. That is, $x^{(t)}=y_{1}b+y_{2}Ab+\cdots y_{t}A^{t-1}b$.
Therefore, 
\[
b-Ax^{(t)}=b-y_{1}Ab-y_{2}A^{2}b-\cdots-y_{t}A^{t}b\;.
\]
If we denote $p(z)=1-y_{1}z-y_{2}z^{2}-\cdots-y_{t}z^{t}=1-z\tilde{p}(z)$,
we obtain $b-Ax^{(t)}=p(A)b$. 
\begin{defn}
Let $x^{(t)}\in\mathcal{K}_{t}$ be an approximate solution of $Ax=b$
and let $r^{(t)}=b-Ax^{(t)}$ be the corresponding residual. The polynomial
$\tilde{p}_{t}$ such that $x^{(t)}=\tilde{p}_{t}(A)b$ is called
the \emph{solution polynomial} of the iteration and the polynomial
$p_{t}(z)=1-z\tilde{p}(z)$ is called the \emph{residual polynomial}
of the iteration. 
\end{defn}

Figure~\ref{fig:minres polynomials} shows several \noun{minres}
residual polynomials.

\begin{figure}
\begin{tabular}{cc}
\includegraphics[width=0.45\textwidth]{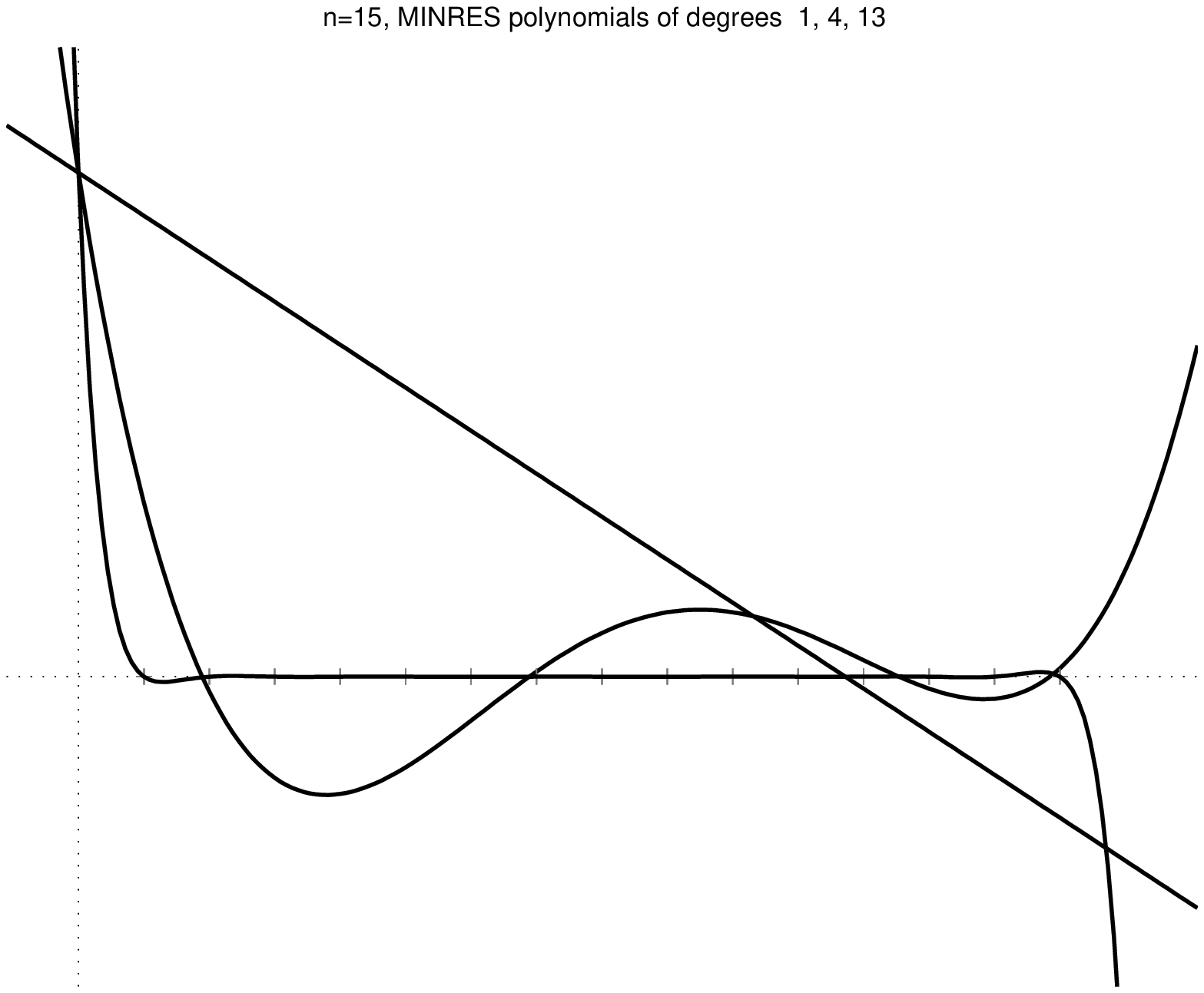} & \includegraphics[width=0.45\textwidth]{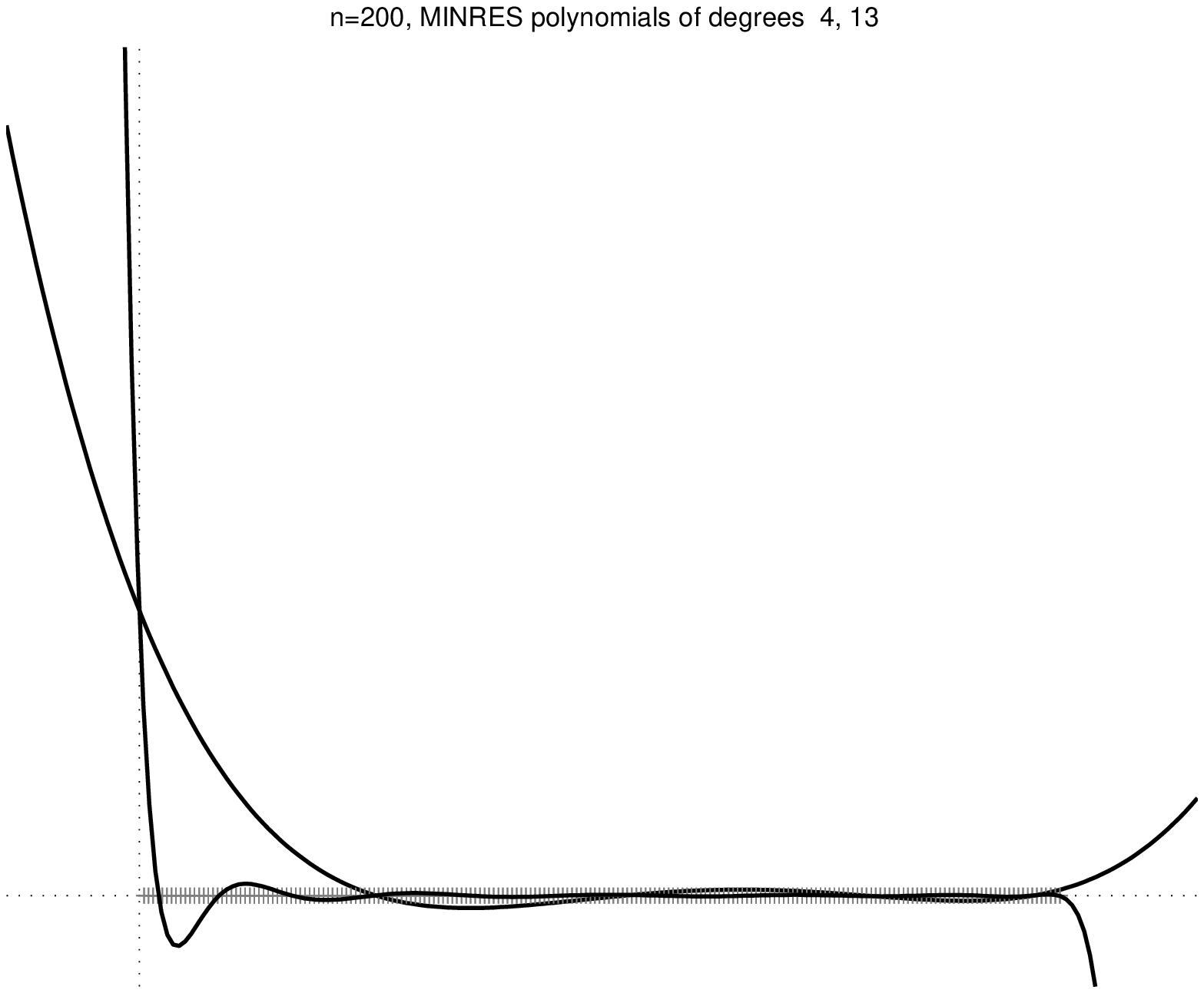}\tabularnewline
\includegraphics[width=0.45\textwidth]{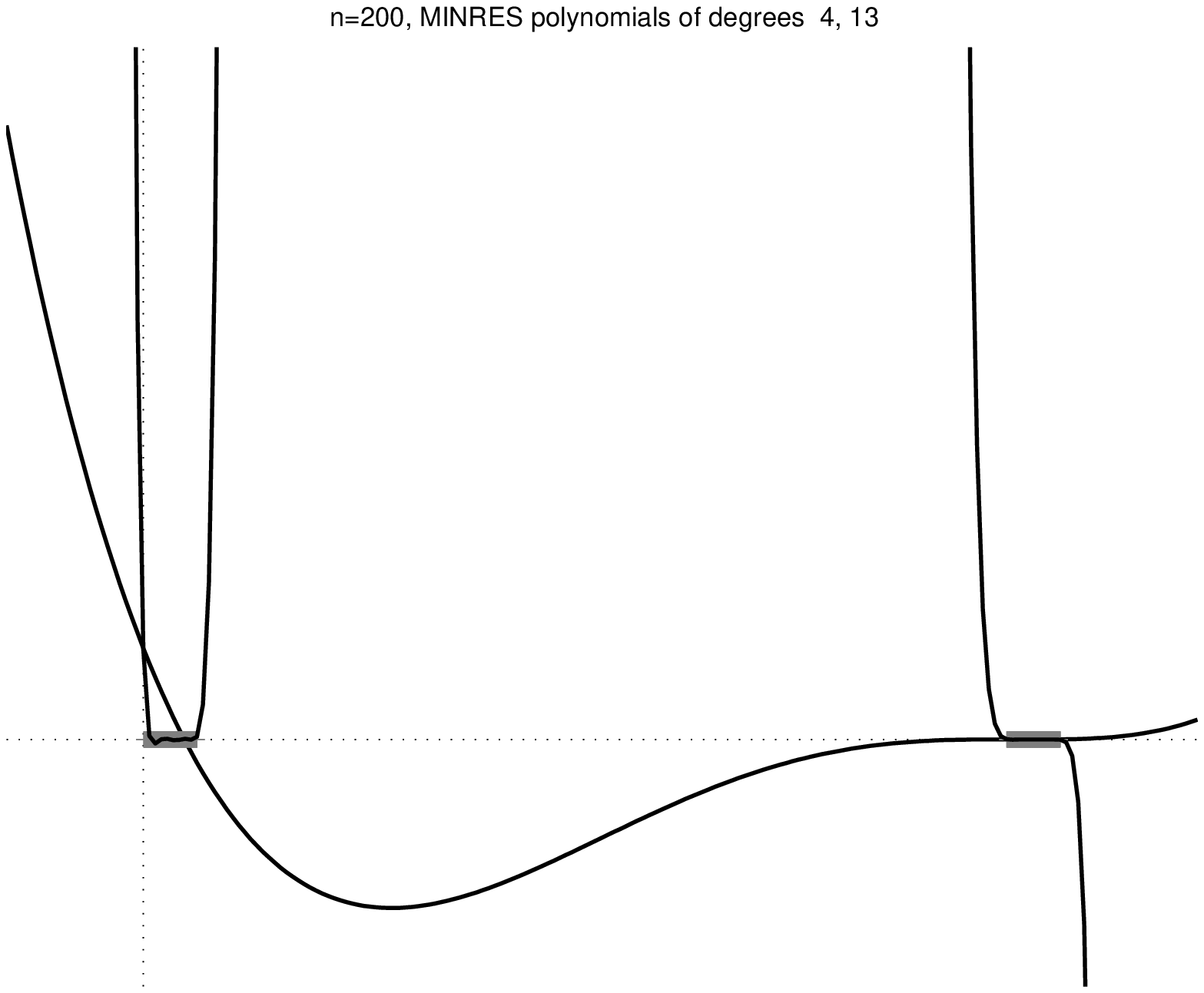} & \includegraphics[width=0.45\textwidth]{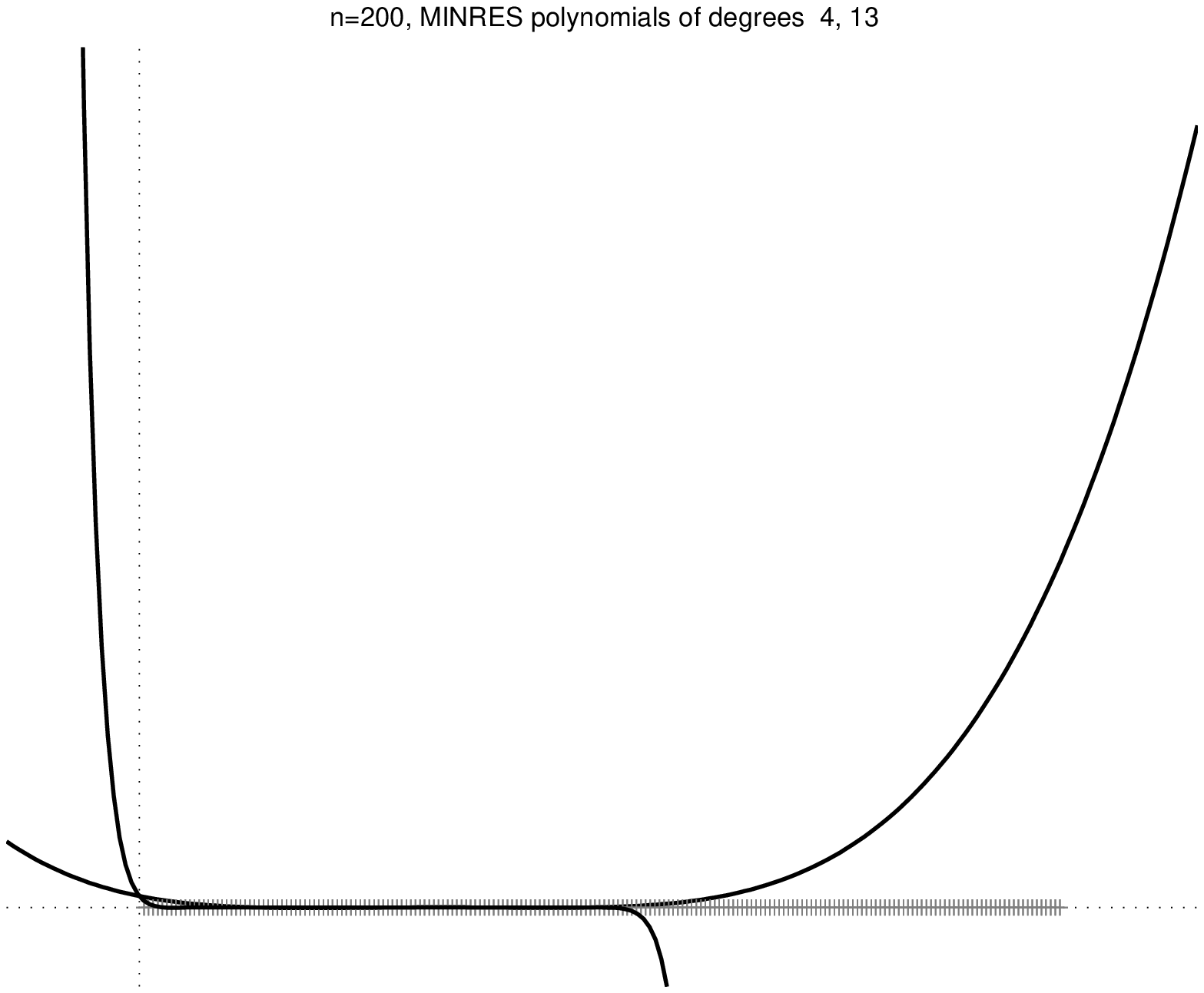}\tabularnewline
\end{tabular}

\caption{\noun{\label{fig:minres polynomials}Minres} residual polynomials
of four linear problems. The order $n$ of the matrices is shown in
each plot, as well as the degree of the polynomials. The eigenvalues
of the matrices are shown using gray tick marks on the $x$ axis.
In all but the bottom-right plot the solution vector is a random vector;
in the bottom-right plot, the solution vector (and hance also the
right-hand side) is a random combination of only $100$ eigenvectors
of $A$, those associated with the $100$ smallest eigenvalues.}
\end{figure}

We now express $p(A)$ in terms of the eigendecomposition of $A$.
Let $A=V\Lambda V^{*}$ be an eigendecomposition of $A$. Since $A$
is Hermitian, $\Lambda$ is real and $V$ is unitary. We have 
\[
p(A)b=p\left(V\Lambda V^{*}\right)b=Vp(\Lambda)V^{*}b\;,
\]
so 
\begin{equation}
\left\Vert b-Ax^{(t)}\right\Vert _{2}=\left\Vert p(A)b\right\Vert _{2}=\left\Vert Vp(\Lambda)V^{*}b\right\Vert _{2}=\left\Vert p(\Lambda)V^{*}b\right\Vert _{2}\;.\label{eq:residual as p(Lambda)}
\end{equation}
We can obtain several bounds on the norm of the residual from this
expression. The most important one is
\begin{eqnarray*}
\left\Vert b-Ax^{(t)}\right\Vert _{2} & = & \left\Vert p(\Lambda)V^{*}b\right\Vert _{2}\\
 & \leq & \left\Vert p(\Lambda)\right\Vert _{2}\left\Vert V^{*}b\right\Vert _{2}=\left\Vert p(\Lambda)\right\Vert _{2}\left\Vert b\right\Vert _{2}\\
 & = & \max_{i=1}^{n}\left\{ \left|p\left(\lambda_{i}\right)\right|\right\} \left\Vert b\right\Vert _{2}\;.
\end{eqnarray*}
The last equality follows from the facts that $p(\Lambda)$ is a diagonal
matrix and that the $2$-norm of a diagonal matrix is the largest
absolute value of an element in it. This proves the following result.
\begin{thm}
\label{thm:polynomial-roots bound for MINRES}The relative $2$-norm
of the residual in the $t$th iteration of \noun{minres}, 
\[
\left\Vert b-Ax^{(t)}\right\Vert _{2}/\left\Vert b\right\Vert _{2}\;,
\]
is bounded by $\max_{i=1}^{n}\left\{ \left|p\left(\lambda_{i}\right)\right|\right\} $
for any univariate polynomial $p$ of degree $t$ such that $p(0)=1$,
where the $\lambda_{i}$'s are the eigenvalues of $A$.
\end{thm}

We can strenghen this result by noting that if $b$ is a linear combination
of only some of the eigenvectors of $A$, then only the action of
$p$ on corresponding eigenvalues matters (not on all the eigenvalues).
An example of this behavior is shown in the bottom-right plot of Figure~\ref{fig:minres polynomials}.
More formally, from (\ref{eq:residual as p(Lambda)}) we obtain
\[
\left\Vert p(\Lambda)V^{*}b\right\Vert _{2}=\left\Vert \begin{bmatrix}p(\lambda_{1})(v_{1}^{*}b)\\
p(\lambda_{2})(v_{2}^{*}b)\\
\vdots\\
p(\lambda_{n})(v_{n}^{*}b)
\end{bmatrix}\right\Vert _{2}\;.
\]
In general, if $b$ is orthogonal or nearly orthogonal to an eigenvector
$v_{j}$ of $A$, then the product $p(\lambda_{j})(v_{j}^{*}b)$ can
be small even if $p(\lambda_{j})$ is quite large. But since right-hand
sides $b$ with this property are rare in practice, this stronger
bound is not useful for us.

Theorem~\ref{thm:polynomial-roots bound for MINRES} states that
if there are low-degree polynomials that are low on the eigenvalues
of $A$ and assume the value $1$ at $0$, the \noun{minres} converges
quickly. Let us examine a few examples. A degree-$n$ polynomial $p$
can satisfy $p(0)=1$ and $p(\lambda_{i})=0$ simulteneously for any
set of $n$ nonzero eigenvalues $\lambda_{1},\ldots,\lambda_{n}$.
Therefore, in the absense of rounding errors \noun{minres} must converge
after $n$ iterations to the exact solution. We could also derive
this exact-convergence result from the fact that $\mathcal{K}_{n}=\mathbb{R}^{n}$,
but the argument that we just gave characterizes the \noun{minres}
polynomials at or near convergence: their roots are at or near the
eigenvalues of $A$. If $A$ has repeated eigenvalues, then it has
fewer than $n$ distinct eigenvalues, so we expect exact convergence
after fewer than $n$ iterations. Even if $A$ does not have repeated
eigenvalues, but it does have only a few tight clusters of eigenvalues,
then \noun{minres} will converge quickly, because a polynomial with
one root near every cluster and a bounded derivative at the roots
will assume low values at all the roots. On the other hand, a residual
polynomial cannot have small values very close to $0$, because it
must assume the value $1$ at $0$. These examples lead us to the
most important observation about Krylov-subspace solvers:
\begin{quote}
Symmetric Krylov-subspace iterative methods for solving linear systems
of equations $Ax=b$ converge quickly if the eigenvalues of $A$ form
a few tight clusters and if $A$ does not have eigenvalues very close
to $0$.
\end{quote}
Scaling both $A$ and $b$ can cluster the eigenvalues or move them
away from zero, but has no effect at all on convergence. Scaling up
$A$ and $b$ moves the eigenvalues away from zero, but distributes
them on a larger interval; scaling $A$ and $b$ down clusters the
eigenvalues around zero, but this brings them closer to zero. Krylov-subspace
iterations are invariant to scaling.

This observation leads to two questions, one analytic and one constructive:
(1) Exactly how quickly does the iteration converges given some characterization
of the spectrum of $A$? (2) How can we alter the specturm of $A$
in order to accelerate convergence? We shall start with the second
question.

\section{Preconditioning}

Suppose that we have a matrix $B$ that approximates $A$ (in a sense
that will become clear shortly), and whose inverse is easier to apply
than the inverse of $A$. That is, linear systems of the form $Bz=r$
are much easier to solve for $z$ than linear systems $Ax=b$. Perhpas
the sparse Cholesky factorization of $B$ is cheaper to compute than
$A$'s, and perhaps there is another inexpensive way to apply $B^{-1}$
to $r$. If $B$ approximates $A$ in the sense that $B^{-1}A$ is
close to the identity, then an algorithm like \noun{minres }will converge
quickly when applied to the linear system
\begin{equation}
\left(B^{-1}A\right)x=B^{-1}b\;,\label{eq:left preconditioning}
\end{equation}
because the eigenvalues of the coefficient matrix $B^{-1}A$ are clustered
around $1$. We will initialize the algorithm by computing the right-hand
side $B^{-1}$, and in every iteration we will multiply $q_{t}$ by
$A$ and then apply $B^{-1}$ to the product. This technique is called
\emph{preconditioning}.

The particular form of preconditioning that we used in (\ref{eq:left preconditioning})
is called \emph{left preconditioning}, because the inverse of the
preconditioner $B$ is applied to both sides of $Ax=b$ from the left.
Left preconditioning is not appropriate to algorithms line \noun{minres}
and Conjugate Gradients that exploit the symmetry of the coefficient
matrix, because in general $B^{-1}A$ is not symmetric. We could replace
\noun{minres} by a Krylov-subspace iterative solver that is applicable
to unsymmetric matrices, but this would force us to give up either
the efficiency or the optimality of \noun{minres} and Conjugate Gradients.

Fortunately, if $B$ is symmetric positive definite, then there are
forms of preconditioning that are appropriate for symmetric Krylov-subspace
solvers. One form of symmetric preconditioning solves
\begin{equation}
\left(B^{-1/2}AB^{-1/2}\right)\left(B^{1/2}x\right)=B^{-1/2}b\label{eq:symmetric root preconditioning}
\end{equation}
for $x$. A clever transformation of Conjugate Gradients yields an
iterative algorithm in which the matrix that generates the Krylov
subspace is $\left(B^{-1/2}AB^{-1/2}\right)$, but which only applies
$A$ and $B^{-1}$ in every iteration. In other words, $B^{-1/2}$
is never used, not even as a linear operator. We shall not show this
transformation here.

We will show how to use a simpler but equally effective form of preconditioning.
Let $B=LL^{T}$ be the Cholesky factorization of $A$. We shall solve
\begin{equation}
\left(L^{-1}AL^{-T}\right)\left(L^{T}x\right)=L^{-1}b\label{eq:Cholesky preconditioning}
\end{equation}
for $y=L^{T}x$ using \noun{minres}, and then we will solve $L^{T}x=y$
by substitution. To form the right-hand side $L^{-1}b$, we solve
for it by substitution as well. The coefficient matrix $L^{-1}AL^{-T}$
is clearly symmetric, so we can indeed apply \noun{minres} to it.
To apply the coefficient matrix to $q_{t}$ in every iteration, we
apply $L^{-T}$ by substitution, apply $A$, and apply $L^{-1}$ by
substitution.

\begin{figure}
\begin{tabular}{cc}
\includegraphics[width=0.45\textwidth]{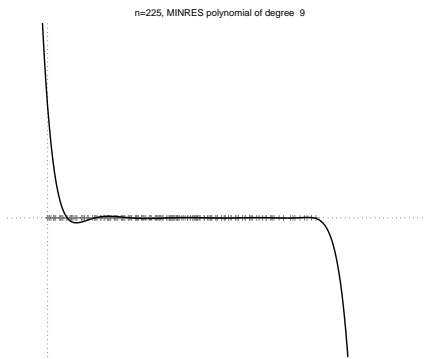} & \tabularnewline
\end{tabular}

\caption{\noun{\label{fig:minres polynomials w/preconditioning}Minres} residual
polynomials for a $15$-by-$15$ two-dimensional mesh (both graphs).
The graph on the left shows polynomial from the application of \noun{minres}
directly to the original linear problem $Ax=b$, and the graph on
the right shows a polynomial from the application of \noun{minres}
to a preconditioned problem $\left(L^{-1}AL^{-T}\right)\left(L^{T}x\right)=\left(L^{-1}b\right)$.
The scaling of the axes in the two graphs are the same.}
\end{figure}

Figure~\ref{fig:minres polynomials w/preconditioning} shows the
spectrum and one \noun{minres }polynomial for two matrices: a two-dimensional
mesh and the same mesh preconditioned as in (\ref{eq:Cholesky preconditioning})
with a Joshi preconditioner $B$. We can see that this particular
preconditioner causes three changes in the spectrum. The most important
change is the small eigenvalues of $L^{-1}AL^{-T}$ are much larger
than the small eigenvalues of $A$. This allows the \noun{minres }polynomial
to assume much smaller vaues on the spectrum. The \noun{minres} polynomials
must assume the value $1$ at $0$, so they tend to have large values
on the neighborhood of $0$ as well. Therefore, a specturm with larger
smallest eigenvalues leads to faster convergence. Indeed, the preconditioned
problem decreased the size of the residual by a factor $10^{-14}$
in $46$ iterations, whereas the unpreconditioned problem took $110$
iterations to achieve a residual with a similar norm. Two other changes
that the preconditioner caused are a large gap in the middle of the
spectrum, which is a good thing, and a larger largest eigenvalue,
which is not. But these two changes are probably less important than
the large increase in the smallest eigenvalues.

The different forms of preconditioning differ in the algorithmic details
of the solver, but they all have the same spectrum.
\begin{thm}
\label{thm:equivalence of preconditioners}Let $A$ be a symmetric
matrix and let $B=LL^{T}$ be a symmetric positive-definite matrix.
A scalar $\lambda$ is either an eigenvalue of all the following eigenvalue
problems or of none of them:
\begin{eqnarray*}
B^{-1}Ax & = & \lambda x\\
B^{-1/2}AB^{-1/2}y & = & \lambda y\\
L^{-1}AL^{-T}z & = & \lambda z\\
Aw & = & \lambda Bw
\end{eqnarray*}
\end{thm}

\begin{proof}
The following relations prove the equivalence of the specta:
\begin{eqnarray*}
x & = & B^{-1/2}y\\
y & = & B^{1/2}x\\
x & = & L^{-T}z\\
z & = & L^{T}x\\
x & = & w
\end{eqnarray*}
\end{proof}
We can strengthen this theorem to also include certain semidefinite
preconditioners.
\begin{thm}
Let $A$ be a symmetric matrix and let $B=LL^{T}$ be a symmetric
positive-semidefinite matrix such that $\textrm{null}(B)=\textrm{null}(A)$.
Denote by $X^{+}$ the pseudo-inverse of a matrix $X$. A scalar $\lambda$
is either an eigenvalue of all the following eigenvalue problems or
of none of them:
\begin{eqnarray*}
B^{+}Ax & = & \lambda x\\
\left(B^{+}\right)^{1/2}A\left(B^{+}\right)^{1/2}y & = & \lambda y\\
\left(L^{+}\right)A\left(L^{+}\right)^{T}z & = & \lambda z\\
Aw & = & \lambda Bw
\end{eqnarray*}
\end{thm}

\begin{proof}
We note that $\textrm{null}(L^{+})=\textrm{null}((L^{+})^{T})=\textrm{null}(B^{+})=\textrm{null}(B)=\textrm{null}(A)$.
Therefore, $\lambda=0$ is an eigenvalue of all the above problems.
If $\lambda\neq0$ is an eigenvalue of one of the above problems,
then the corresponding eigenvector is not in $\textrm{null}(A)$.
This implies that the relations defined in the proof of Theorem~\ref{thm:equivalence of preconditioners},
with inverses replaced by pseudo-inverses, define relations between
nonzero vectors. Therefore, $\lambda$ is an eigenvalue of all the
eigenvalue problems.
\end{proof}
Even though the different forms of preconditioning are equivalent
in terms of the spectra of the coefficient matrices, they are different
algorithmically. If symmetry is not an issue (e.g., if $A$ itself
is unsymmetric), the form $B^{-1}A$ is the most general. When $A$
is symmetric, we usually require that $B$ is symmetric positive-definite
(or semi-definite with the same null space as $A$). In this case,
the form $B^{-1/2}AB^{-1/2}$, when coupled with the transformation
that allows multiplications only by $B^{-1}$ (and not by $B^{-1/2}$),
is more widely applicable than the form $L^{-1}AL^{-T}$, because
the latter requires a Cholesky factorization of $B$, whereas in the
former any method of applying $B^{-1}$ can be used.

One issue that arizes with any form of preconditioning is the definition
of the residual. If we apply \noun{minres} to $\left(L^{-1}AL^{-T}\right)\left(L^{T}x\right)=L^{-1}b$,
say, it minimizes the $2$-norm of the \emph{preconditioned residual}
\[
L^{-1}b-\left(L^{-1}AL^{-T}\right)\left(L^{T}x^{(t)}\right)=L^{-1}b-L^{-1}Ax^{(t)}=L^{-1}\left(b-Ax^{(t)}\right)\;.
\]
Thus, the true residual $b-Ax^{(t)}$ in preconditioned \noun{minres}
may not be minimal. This is roughly the same issue as with the norms
used in Conjugate Gradients: we minimize the residual in a norm that
is related to $A$.

\section{Chebyshev Polynomials and Convergence Bounds}

The link between Krylov-subspace iterations and polynomials suggests
another idea. Given some information on the spectrum of $A$, we can
try to analytically define a sequence $\tilde{p}_{t}$ of solution
polynomials such that for any $b$ the vector $x^{(t)}\equiv\tilde{p}_{t}(A)b\in\mathcal{K}_{t}$
is a good approximation to $x$, the exact solution of $Ax=b$. More
specifically, we can try to define $\tilde{p}_{t}$ such that the
residuals $b-Ax^{(t)}$ are small. We have seen that the residual
can be expressed as $p_{t}(A)b$, where $p_{t}(z)=1-z\tilde{p}_{t}(z)$
is the residual polynomial. Therefore, if $\tilde{p}_{t}$ is such
that $p_{t}$ assumes low values on the eigenvalues of $A$ and satisfies
$p_{t}(0)=1$, then $x^{(t)}$ is a good approximate solution. This
idea can be used both to construct iterative solvers and to prove
bounds on the convergence of methods like \noun{minres} and Conjugate
Gradients.

One obvious problem is that we do not know what the eigenvalues of
$A$ are. Finding the eigenvalues is more difficult than solving $Ax=b$.
However, in some cases we can use the structure of $A$ (even with
preconditioning) to derive bounds on the smallest and largest eigenvalues
of a positive-definite martrix $A$, denoted $\lambda_{\min}$ and
$\lambda_{\max}$. 

Suppose that we somehow obtained bounds on the extereme eigenvalues
of $A$, 
\[
0<\rho_{\min}\leq\lambda_{\min}\leq\lambda_{i}\leq\lambda_{\max}\leq\rho_{\max}\;.
\]
We shall not discuss here how we might obtain $\rho_{\min}$ and $\rho_{\max}$;
this is the topic of much of the rest of the book. It turns out that
we can build a sequence $p_{t}$ of polynomials such that
\begin{enumerate}
\item $p_{t}(0)=1$, and
\item $\max_{z\in[\rho_{\min},\rho_{\max}]}\left|p_{t}(z)\right|$ is as
small as possible for a degree $t$ polynomial with value $1$ at
$0$.
\end{enumerate}
The polynomials that solve this optimization problem are derived from
\emph{Chebyshev} polynomials, which can be defined using the recurrence
\begin{eqnarray*}
c_{0}(z) & = & 1\\
c_{1}(z) & = & z\\
c_{t}(z) & = & 2zc_{t-1}(z)-c_{t-2}(z)\;.
\end{eqnarray*}
The polynomials that reduce the residual are
\[
p_{t}(z)=\frac{1}{c_{t}\left(\frac{\rho_{\max}+\rho_{\min}}{\rho_{\max}-\rho_{\min}}\right)}c_{t}\left(\frac{\rho_{\max}+\rho_{\min}-2z}{\rho_{\max}-\rho_{\min}}\right)\;.
\]

\subsection*{An Iterative Linear Solver based on Chebyshev Polynomials}

Our first application of Chebyshev polynomials is an iterative Krylov-subspace
solver based on them. We will refer to this solver as the \emph{Krylov-Chebyshev}
solver\footnote{The algorithm that we describe below is related to a more well-known
Chebychev linear solver that is used with matrix splittings.}. The polynomials $p_{t}$ implicitly define polynomials $\tilde{p}_{t}$
that we can use to construct approximate solutions. The residual for
an approximate solution $x^{(t)}$ is $r^{(t)}=b-Ax^{(t)}$. If we
define $x^{(t)}=\tilde{p}_{t}(A)b$, we have $r^{(t)}=p_{t}(A)b=b-A\tilde{p}_{t}(A)b$.

We now derive recurrences for $x^{(t)}$ and $r^{(t)}$. To keep the
notation simple, we define $\rho_{+}=\rho_{\max}+\rho_{\min}$ and
$\rho_{-}=\rho_{\max}+\rho_{\min}$. For the two base cases we have
\begin{eqnarray*}
r^{(0)}=p_{0}(A)b & = & c_{0}^{-1}\left(\frac{\rho_{+}}{\rho_{-}}\right)c_{0}\left(\frac{\rho_{+}}{\rho_{-}}I-\frac{2}{\rho_{-}}A\right)b=Ib=b\\
r^{(1)}=p_{1}(A)b & = & c_{1}^{-1}\left(\frac{\rho_{+}}{\rho_{-}}\right)c_{1}\left(\frac{\rho_{+}}{\rho_{-}}I-\frac{2}{\rho_{-}}A\right)b\\
 & = & \left(\frac{\rho_{+}}{\rho_{-}}\right)^{-1}\left(\frac{\rho_{+}}{\rho_{-}}I-\frac{2}{\rho_{-}}A\right)b=b-\frac{2}{\rho_{+}}Ab\;.
\end{eqnarray*}
This implies that 
\begin{eqnarray*}
x^{(0)} & = & 0\\
x^{(1)} & = & \frac{2}{\rho_{+}}b\;.
\end{eqnarray*}
For $t\geq2$ we have
\begin{eqnarray*}
p_{t}(A)b & = & c_{t}^{-1}\left(\frac{\rho_{+}}{\rho_{-}}\right)c_{t}\left(\frac{\rho_{+}}{\rho_{-}}I-\frac{2}{\rho_{-}}A\right)b\\
 & = & c_{t}^{-1}\left(\frac{\rho_{+}}{\rho_{-}}\right)\left(2\left(\frac{\rho_{+}}{\rho_{-}}I-\frac{2}{\rho_{-}}A\right)c_{t-1}\left(\frac{\rho_{+}}{\rho_{-}}I-\frac{2}{\rho_{-}}A\right)-c_{t-2}\left(\frac{\rho_{+}}{\rho_{-}}I-\frac{2}{\rho_{-}}A\right)\right)b\\
\\
 & = & c_{t}^{-1}\left(\frac{\rho_{+}}{\rho_{-}}\right)\left(2\left(\frac{\rho_{+}}{\rho_{-}}I-\frac{2}{\rho_{-}}A\right)c_{t-1}\left(\frac{\rho_{+}}{\rho_{-}}\right)p_{t-1}\left(A\right)-c_{t-2}\left(\frac{\rho_{+}}{\rho_{-}}\right)p_{t-2}\left(A\right)\right)b\\
 & = & c_{t}^{-1}\left(\frac{\rho_{+}}{\rho_{-}}\right)\left(2\left(\frac{\rho_{+}}{\rho_{-}}I-\frac{2}{\rho_{-}}A\right)c_{t-1}\left(\frac{\rho_{+}}{\rho_{-}}\right)r^{(t-1)}-c_{t-2}\left(\frac{\rho_{+}}{\rho_{-}}\right)r^{(t-2)}\right)\;.
\end{eqnarray*}
To compute $r^{(t)}$ from this recurrence, we need $r^{(t-1)}$ and
$r^{(t-2)}$, three elements of the sequence $c_{t}(\rho_{+}/\rho_{-})$,
and one multiplication of a vector by $A$. Therefore, we can compute
$r^{(t)}$ and $c_{t}(\rho_{+}/\rho_{-})$ concurrently in a loop.
From the recurrence for $r^{(t)}$ we can derive a recurrence for
$x^{(t)}$,
\begin{eqnarray*}
r^{(t)} & = & p_{t}(A)b\\
 & = & c_{t}^{-1}\left(\frac{\rho_{+}}{\rho_{-}}\right)\left(\frac{2\rho_{+}}{\rho_{-}}c_{t-1}\left(\frac{\rho_{+}}{\rho_{-}}\right)r^{(t-1)}-\frac{4}{\rho_{-}}Ac_{t-1}\left(\frac{\rho_{+}}{\rho_{-}}\right)r^{(t-1)}-c_{t-2}\left(\frac{\rho_{+}}{\rho_{-}}\right)r^{(t-2)}\right)\\
 & = & c_{t}^{-1}\left(\frac{\rho_{+}}{\rho_{-}}\right)\biggl(\frac{2\rho_{+}}{\rho_{-}}c_{t-1}\left(\frac{\rho_{+}}{\rho_{-}}\right)\left(b-Ax^{(t-1)}\right)\\
 &  & \quad\quad\quad\quad\quad\quad-\frac{4}{\rho_{-}}Ac_{t-1}\left(\frac{\rho_{+}}{\rho_{-}}\right)r^{(t-1)}\\
 &  & \quad\quad\quad\quad\quad\quad-c_{t-2}\left(\frac{\rho_{+}}{\rho_{-}}\right)\left(b-Ax^{(t-2)}\right)\biggr)\;.
\end{eqnarray*}
so
\begin{eqnarray*}
x^{(t)} & = & \tilde{p}_{t}(A)b\\
 & = & c_{t}^{-1}\left(\frac{\rho_{+}}{\rho_{-}}\right)\biggl(\frac{2\rho_{+}}{\rho_{-}}c_{t-1}\left(\frac{\rho_{+}}{\rho_{-}}\right)x^{(t-1)}\\
 &  & \quad\quad\quad\quad\quad\quad+\frac{4}{\rho_{-}}c_{t-1}\left(\frac{\rho_{+}}{\rho_{-}}\right)r^{(t-1)}\\
 &  & \quad\quad\quad\quad\quad\quad-c_{t-2}\left(\frac{\rho_{+}}{\rho_{-}}\right)x^{(t-2)}\biggr)\;.
\end{eqnarray*}

\begin{figure}
\begin{centering}
\begin{tabular}{cc}
\includegraphics[width=0.45\textwidth]{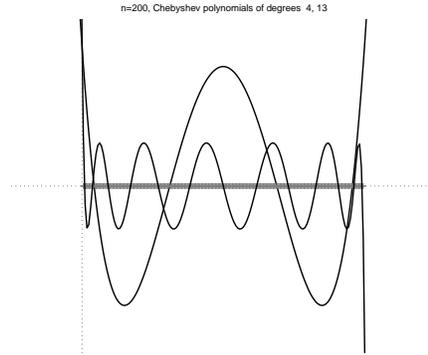} & \tabularnewline
\end{tabular}
\par\end{centering}
\caption{\noun{\label{fig:chebyshev polynomials}}Chebyshev residual polynomials
for two $200$-by-$200$ matrices with different spectra, where $\rho_{\min}=\lambda_{\min}$
is set at the minimal eigenvalue and $\rho_{\max}=\lambda_{\max}$
is set at the largest eigenvalue. The two matrices have the same extreme
eigenvalues but otherwise their spectra is very different. The Chebyshev
polynomials depend only on $\rho_{\max}$ and $\rho_{\min}$. Compare
to the \noun{minres }polynomials for the similar spectra in Figure~\ref{fig:minres polynomials}.}
\end{figure}

This algorithm converges more slowly than \noun{minres }and Conjugate
Gradients. \noun{Minres} is guaranteed to minimize the residual over
all $x^{(t)}$ in $\mathcal{K}_{t}$. The solution that we obtain
from the Krylov-Chebyshev recurrences is in $\mathcal{K}_{t}$, so
it cannot yield a smaller residual than the \noun{minres }solution.
Its main algorithmic advantage over \noun{minres} and Conjugate Gradients
is that the polynomials $\tilde{p}_{t}$ that it produces depend only
on $t$ and on $\rho_{\min}$ and $\rho_{\max}$, but they do not
depend on $b$. Therefore, $\tilde{p}_{t}(A)$ is a fixed linear operator,
so it can be used as a preconditioner $B^{-1}=\tilde{p}_{t}(A)$.
In contrast, the polynomials that \noun{minres} and Conjugate Gradients
generate depend on $b$, so we cannot use these solvers as preconditioners
unless we set very strict convergence bounds, in which case an application
of these solvers is numerically indistiguishable from an application
of $A^{-1}$.

Since this algorithm does not exploit the symmetry of $A$, it can
be used with left preconditioning, not only with symmetric forms of
preconditioning.

\subsection*{Chebyshev-Based Convergence Bounds}

We can also use the Chebyshev iteration to bound the convergence of
\noun{minres }and Conjugate Gradients. The key is the following theorem,
which we state without a proof.
\begin{thm}
\label{thm:Chebyshev polynomial bounds}Let $p(z)$ be the Chebyshev
polynomials defined above with respect to the interval $[\lambda_{\min},\lambda_{\max}]$.
Denote by $\kappa$ the ratio $\kappa=\lambda_{\max}/\lambda_{\min}$.
For any $z\in[\lambda_{\min},\lambda_{\max}]$ we have
\[
\left|p_{t}(z)\right|\leq2\left(\left(\frac{\sqrt{\kappa}+1}{\sqrt{\kappa}-1}\right)^{t}+\left(\frac{\sqrt{\kappa}+1}{\sqrt{\kappa}-1}\right)^{-t}\right)^{-1}\leq2\left(\frac{\sqrt{\kappa}-1}{\sqrt{\kappa}+1}\right)^{t}\;.
\]
\end{thm}

In this theorem, $0<\lambda_{\min}\leq\lambda_{\max}$ are arbitrary
positive numbers that denote then end of the inteval that defines
$p_{t}$. But when these numbers are the extreme eigenvalues of a
symmetric positive definite matrix, the ratio $\kappa=\lambda_{\max}/\lambda_{\min}$
plays an important-enough role in numerical linear algebra to deserve
a name.
\begin{defn}
Let $A$ be a symmetric semidefinite matrix, and let $\lambda_{\min}$
and $\lambda_{\max}$ be its extreme nonzero eigenvalue. The ratio
\[
\kappa=\frac{\lambda_{\max}}{\lambda_{\max}}=\left\Vert A\right\Vert _{2}\left\Vert A^{+}\right\Vert _{2}
\]
is called the \emph{spectral condition number} of $A$. The definition
\[
\kappa=\left\Vert A\right\Vert \left\Vert A^{+}\right\Vert 
\]
generalizes the condition number to any matrix and to any norm.
\end{defn}

We can use Theorem~\ref{thm:Chebyshev polynomial bounds} to bound
the residuals in \noun{minres}. 
\begin{thm}
Consider the application of \noun{minres} to the linear system $Ax=b$.
Let $r^{(t)}$ be the \noun{minres} residual at iteration $t$. Then

\[
\frac{\left\Vert r^{(t)}\right\Vert _{2}}{\left\Vert b\right\Vert _{2}}\leq2\left(\frac{\sqrt{\kappa}-1}{\sqrt{\kappa}+1}\right)^{t}\;.
\]
\end{thm}

\begin{proof}
The residual $r^{(t)}$ is the minimal residual for any $x^{(t)}\in\mathcal{K}_{t}$.
Let $\hat{r}^{(t)}$ be the Krylov-Chebyshev residual for $A$ and
$b$ and let $p_{t}$ be the Krylov-Chebyshev residual polynomial.
Let $A=V\Lambda V^{*}$ be an eigendecomposition of $A$. The theorem
follows from the following inequalities.
\begin{eqnarray*}
\left\Vert r^{(t)}\right\Vert _{2} & \leq & \left\Vert \hat{r}^{(t)}\right\Vert _{2}=\left\Vert p(A)b\right\Vert _{2}=\left\Vert Vp(\Lambda)V^{*}b\right\Vert _{2}=\left\Vert p(\Lambda)V^{*}b\right\Vert _{2}\\
 & \leq & \left\Vert p(\Lambda)V^{*}\right\Vert _{2}\left\Vert b\right\Vert _{2}=\left\Vert p(\Lambda)\right\Vert _{2}\left\Vert b\right\Vert _{2}\\
 & = & \max_{i}\left\{ \left|p\left(\lambda_{i}\right)\right|\right\} \left\Vert b\right\Vert _{2}\\
 & \leq & \max_{z\in[\lambda_{\min},\lambda_{\max}]}\left\{ \left|p\left(z\right)\right|\right\} \left\Vert b\right\Vert _{2}\\
 & \leq & 2\left(\frac{\sqrt{\kappa}-1}{\sqrt{\kappa}+1}\right)^{t}\left\Vert b\right\Vert _{2}\;.
\end{eqnarray*}
\end{proof}
Similar results can be stated for the error and residual of Conjugate
Gradients in the $A$ and $A^{-1}$ norms, respectively.

For small $\kappa$, we can expect convergence to a fixed tolerance,
say $\left\Vert r^{(t)}\right\Vert _{2}\leq10^{-12}\left\Vert r^{(t)}\right\Vert _{2}$
within a constant number of iterations. As $\kappa$ grows, 
\[
\frac{\sqrt{\kappa}-1}{\sqrt{\kappa}+1}\to1-\frac{2}{\sqrt{\kappa}}\;,
\]
so we are guaranteed convergence to a fixed tolerance within $O(\sqrt{\kappa})$
iterations.

%% file: chapter-direct.tex
 \setcounter{chapter}{2}

\chapter{Computing the Cholesky Factorization of Sparse Matrices}

In many support preconditioners, the preconditioner $B$ is factored
before the iterations begin. The Cholesky factorization of $B$ allows
us to efficiently solve the correction equations $Bz=r$. This chapter
explains the principles behind the factorization of sparse symmetric
positive definite matrices.

\section{The Cholesky Factorization}

We first show that the Cholesky factorization $A=LL^{T}$ of a symmetric
positive-definite (\noun{spd}) matrix $A$ always exists. 

A matrix $A$ is \emph{positive definite} if $x^{T}Ax>0$ for all
$0\neq x\in\mathbb{R}^{n}$. The same definition extends to complex
Hermitian matrices. The matrix is positive \emph{semidefinite} if
$x^{T}Ax\geq0$ for all $0\neq x\in\mathbb{R}^{n}$ and $x^{T}Ax=0$
for some $0\neq x\in\mathbb{R}^{n}$.

To prove the existence of the factorization, we use induction and
the construction shown in Chapter~XXX. If $A$ is $1$-by-$1$, then
$x^{T}Ax=A_{11}x_{1}^{2}>0$, so $A_{11}\geq0$, so it has a real
square root. We set $L_{11}=\sqrt{A_{11}}$ and we are done. We now
assume by induction that all \noun{spd} matrices of dimension $n-1$
or smaller have a Cholesky factorization. We now partition $A$ into
a $2$-by-$2$ block matrix and use the partitioning to construct
a factorization,
\[
A=\begin{bmatrix}A_{11} & A_{21}^{T}\\
A_{21} & A_{22}
\end{bmatrix}\;.
\]
Because $A$ is \noun{spd}, $A_{11}$ must also be \noun{spd}. If
it is not positive definite, then a vector $y\neq0$ such that $y^{T}A_{11}y\leq0$
can be extended with zeros to a vector $x$ such that $x^{T}Ax\leq0$.
Therfore, $A_{11}$ has a Choleskyh factor $L_{11}$ by induction.
The Cholesky factor of a nonsingular matrix must be nonsingular, so
we can define $L_{21}=A_{21}L_{11}^{-T}$ ($L_{11}^{-T}$ denotes
the inverse of $L_{11}^{T}$). The key step in the proof is to show
that the \emph{Schur complement} $A_{22}-L_{21}L_{21}^{T}=A_{22}-A_{21}A_{11}^{-1}A_{21}^{T}$
is also positive definite. Suppose for contradition that it is not,
and let $z\neq0$ be such that $z^{T}\left(A_{22}-L_{21}L_{21}^{T}\right)z\leq0$.
We define 
\[
x=\begin{bmatrix}-A_{11}^{-1}A_{21}^{T}z\\
z
\end{bmatrix}=\begin{bmatrix}w\\
z
\end{bmatrix}\;.
\]
Now we have 
\begin{eqnarray*}
x^{T}Ax & = & w^{T}A_{11}w+w^{T}A_{21}^{T}z+z^{T}A_{21}w+z^{T}A_{22}z\\
 & = & \left(A_{11}^{-1}A_{21}^{T}z\right)^{T}A_{11}\left(A_{11}^{-1}A_{21}^{T}z\right)-2z^{T}A_{21}\left(A_{11}^{-1}A_{21}^{T}z\right)w+z^{T}A_{22}z\\
 & = & z^{T}A_{21}A_{11}^{-1}A_{21}^{T}z-2z^{T}A_{21}A_{11}^{-1}A_{21}^{T}z+z^{T}A_{22}z\\
 & = & z^{T}\left(A_{22}-L_{21}L_{21}^{T}\right)z\\
 & \leq & 0\;,
\end{eqnarray*}
which is impossible, so our supposition was wrong. Because $A_{22}-L_{21}L_{21}^{T}$
is \noun{spd}, it has a Cholesky factor $L_{22}$ by induction. The
three blocks $L_{11}$, $L_{21}$, and $L_{22}$ form a Cholesky factor
for $A$, since

\begin{eqnarray*}
A_{11} & = & L_{11}L_{11}^{T}\\
A_{21} & = & L_{21}L_{11}^{T}\\
A_{22} & = & L_{21}L_{21}^{T}+L_{22}L_{22}^{T}\;.
\end{eqnarray*}

Symmetric positive semidefinite (\noun{spsd}) matrices also have a
Cholesky factorization, but in floating-point arithmetic, it is difficult
to compute a Cholesky factor that is both backward stable and has
the same rank as $A$. To see that a factorization exists, we modify
the construction as follows. If $A$ is $1$-by-$1$, then if it is
singular than it is exactly zero, in which case we can set $L=A$.
Otherwise, we partition $A$, selecting $A_{11}$to be $1$-by-$1$.
If $A_{11}\neq0$, then it is invertible and we can continue with
the same proof, except that we show that the Schur complement is semidefinite,
not definte. If $A_{11}=0$, then it is not hard to show that $A_{21}$
must be zero as well. This implies that the equation $A_{21}=L_{21}L_{11}^{T}$
is zero on both sides for any choice of $L_{21}$. We set $L_{21}=0$
and continue.

The difficulty in a floating-point implementation lies in deciding
whether a computed $A_{11}$ would be zero in exact arithmetic. In
general, even if the next diagonal element $A_{11}=0$ in exact arithmetic,
in floating-point it might be computed as a small nonzero value. Should
we round it (and $A_{21}$) to zero? Assuming that it is a nonzero
when in fact is should be treated is a zero leads to an unstable factorization.
The small magnitude of the computed $L_{11}$ can cause the elements
of $L_{21}$ to be large, which leads to large inaccuracies in the
Schur complement. On the other hand, rounding small values to zero
always leads to a backward stable factorization, since the rounding
is equivalent to a small perturbation in $A$. But rounding a column
to zero when the value in exact arithmetic is not zero causes the
rank of $L$ to be smaller than the rank of $A$. This can later cause
trouble, since some vectors $b$ that are in the range of $A$ are
not in the range of $L$. In such a case, there is no $x$ such that
$LL^{T}x=b$ even if $Ax=b$ is consistent.

\section{Work and Fill in Sparse Cholesky}

When $A$ is sparse, operations on zeros can be skipped. For example,
suppose that 
\[
A=\begin{bmatrix}2 & 0 & \cdots\\
0 & 3\\
\vdots &  & \ddots
\end{bmatrix}=\begin{bmatrix}\sqrt{2} & 0 & \cdots\\
0 & \sqrt{3}\\
\vdots &  & \ddots
\end{bmatrix}\begin{bmatrix}\sqrt{2} & 0 & \cdots\\
0 & \sqrt{3}\\
\vdots &  & \ddots
\end{bmatrix}^{T}\;.
\]
There is no need to divide $0$ by $\sqrt{2}$ to obtain $L_{2,1}$
(the element in row $2$ and column $1$; from here on, expressions
like $L_{2,1}$ denote matrix elements, not submatrices). Similarly,
there is no need to subtract $0\times0$ from the $3$, the second
diagonal element. If we represent zeros implicitly rather than explicitly,
we can avoid computations that have no effect, and we can save storage.
How many arithmetic operations do we need to perform in this case?
\begin{defn}
We define $\eta(A)$ to be the number of nonzero elements in a matrix
$A$. We define $\phi_{\mathtt{alg}}(A)$ to be the number of arithmetic
operations that some algorithm $\mathtt{alg}$ performs on an input
matrix $A$, excluding multiplications by zeros, divisions of zeros,
additions and subtractions of zeros, and taking square roots of zeros.
When the algorithm is clear from the context, we drop the subscript
in the $\phi$ notation.
\end{defn}

\begin{thm}
\label{thm:ops in sparse cholesky}Let $\mathtt{SparseChol}$ be a
sparse Cholesky factorization algorithm that does not multiply by
zeros, does not divide zeros, does not add or subtract zeros, and
does not compute square roots of zeros. Then for a symmetric positive-definite
matrix $A$ with a Cholesky factor $L$ we have
\begin{eqnarray*}
\phi_{\mathtt{SparseChol}}(A) & = & \sum_{j=1}^{n}\left(1+\eta\left(L_{j+1\colon n}\right)+\frac{\eta\left(L_{j+1\colon n}\right)\left(\eta\left(L_{j+1\colon n}\right)+1\right)}{2}\right)\\
 & = & \sum_{j=1}^{n}O\left(\eta^{2}\left(L_{j+1\colon n}\right)\right)\;.
\end{eqnarray*}
\end{thm}

\begin{proof}
Every arithmetic operation in the Cholesky factorization involves
an element or two from a column of $L$: in square roots and divisions
the output is an element of $L$, and in multiply-subtract the two
inputs that are multiplied are elements from one column of $L$. We
count the total number of operations by summing over columns of $L$. 

Let us count the operations in which elements of $L_{j\colon n,j}$
are involved. First, one square root. Second, divisions of $\eta\left(L_{j+1\colon n}\right)$
by that square root. We now need to count the operations in which
elements from this column update the Schur complement. To count these
operations, we assume that the partitioning of $A$ is into a $2$-by-$2$
block matrix, in which the first diagonal block consists of $j$ rows
and columns and the second of $n-j$. The computation of the Schur
complement is
\[
A_{j+1\colon n,j+1\colon n}-L_{j+1\colon n,1\colon j}L_{j+1\colon n,1\colon j}^{T}=A_{j+1\colon n,j+1\colon n}-\sum_{k=1}^{j}L_{j+1\colon n,k}L_{j+1\colon n,k}^{T}\;.
\]
This is the only Schur-complement computation in which $L_{j\colon n,j}$
is involved under this partitioning of $A$. It was not yet used,
because it has just been computed. It will not be used again, because
the recursive factorization of the Schur complement is self contained.
The column contributes one outer product $L_{j+1\colon n,j}L_{j+1\colon n,j}^{T}$.
This outer product contains $\eta^{2}\left(L_{j+1\colon n}\right)$
nonzero elements, but it is symmetric, so only its lower triangle
needs to be computed. For each nonzero element in this outer product,
two arithmetic operations are performed: a multiplication of two elements
of $L_{j+1\colon n}$ and a subtraction of the product from another
number. This yields the total operation count.
\end{proof}
Thus, the number of arithmetic operations is asymptotically proportional
to the sum of squares of the nonzero counts in columns of $L$. The
total number of nonzeros in $L$ is, of course, simply the sum of
the nonzero counts,
\[
\eta(L)=\sum_{j=1}^{n}\eta\left(L_{j+1\colon n}\right)\;.
\]
This relationship between the arithmetic operation count and the nonzero
counts shows two things. First, sparser factors usually (but not always)
require less work to compute. Second, a factor with balanced nonzero
counts requires less work to compute than a factor with some relatively
dense columns, even if the two factors have the same dimension and
the same total number of nonzeros.

The nonzero structure of $A$ and $L$ does not always reveal everything
about the sparsity \emph{during} the factorization. Consider the following
matrix and its Cholesky factor,
\[
A=\begin{bmatrix}4 &  & 2 & 2\\
 & 4 & 2 & -2\\
2 & 2 & 6\\
2 & -2 &  & 6
\end{bmatrix}=\begin{bmatrix}2\\
 & 2\\
1 & 1 & 2\\
1 & -1 &  & 2
\end{bmatrix}\begin{bmatrix}2\\
 & 2\\
1 & 1 & 2\\
1 & -1 &  & 2
\end{bmatrix}^{T}\;.
\]
The element in position $4,3$ is zero in $A$ and in $L$, but it
might fill in one of the Schur complements. If we partition $A$ into
two $2$-by-$2$ blocks, this element never fills, since
\begin{eqnarray*}
A_{3\colon4,3\colon4}-L_{3\colon4,1\colon2}L_{3\colon4,1\colon2}^{T} & = & \begin{bmatrix}6\\
 & 6
\end{bmatrix}-\begin{bmatrix}1 & 1\\
1 & -1
\end{bmatrix}\begin{bmatrix}1 & 1\\
1 & -1
\end{bmatrix}\\
 & = & \begin{bmatrix}6\\
 & 6
\end{bmatrix}-\begin{bmatrix}2 & 0\\
0 & 2
\end{bmatrix}\\
 & = & \begin{bmatrix}4\\
 & 4
\end{bmatrix}\;.
\end{eqnarray*}
However, if we first partition $A$ into a $1$-by-$1$ block and
a $3$-by-$3$ block, then the $4,3$ element fills in the Schur complement,
\begin{eqnarray*}
A_{2\colon4,2\colon4}-L_{2\colon4,1}L_{2\colon4,1}^{T} & = & \begin{bmatrix}4 & 2 & -2\\
2 & 6\\
-2 &  & 6
\end{bmatrix}-\begin{bmatrix}0\\
1\\
1
\end{bmatrix}\begin{bmatrix}0 & 1 & 1\end{bmatrix}\\
 & = & \begin{bmatrix}4 & 2 & -2\\
2 & 6\\
-2 &  & 6
\end{bmatrix}-\begin{bmatrix}0 & 0 & 0\\
0 & 1 & 1\\
0 & 1 & 1
\end{bmatrix}\\
 & = & \begin{bmatrix}4 & 2 & -2\\
2 & 5 & -1\\
-2 & -1 & 5
\end{bmatrix}\;.
\end{eqnarray*}
When we continue the factorization, the $4,3$ element in the Schur
complement must cancel exactly by a subsequent subtraction, because
we know it is zero in $L$ (The Cholesky factor is unique). This example
shows that an element can fill in some of the Schur complements, even
if it zero in $L$. Clearly, even an position that is not zero in
$A$ can become zero in $L$ due to similar cancelation. For some
analyses, it helps to define fill in a way that accounts for this
possibility.
\begin{defn}
\label{def:fill}A \emph{fill} in a sparse Cholesky factorization
is a row-column pair $i,j$ such that $A_{i,j}=0$ and which fills
in at least one Schur complement. That is, $A_{i,j}=0$ and $A_{i,j}-L_{i,1\colon k}L_{1\colon k,j}^{T}\neq0$
for some $k<j$.
\end{defn}

\section{An Efficient Implemention of Sparse Cholesky}

To fully exploit sparsity, we need to store the matrix and the factor
in a data structure in which effect-free operations on zeros incur
no computational cost at all. Testing values to determine whether
they are zero before performing an arithmetic operation is a bad idea:
the test takes time even if the value is zero (on most processor,
a test like this takes more time than an arithmetic operation). The
data structure should allow the algorithm to implicitly skip zeros.
Such a data structure increases the cost of arithmetic on nonzeros.
Our goal is to show that in spite of the overhead of manipulating
the sparse data structure, the total number of operations in the factorization
can be kept proportional to the number of arithmetic operations.

Another importan goal in the design of the data structure is memory
efficiency. We would like the total amount of memory required for
the sparse-matrix data structure to be proportional to the number
of nonzeros that it stores.

Before we can present a data structure that achieves these goals,
we need to reorder the computations that the Cholesky factorization
perform. The reordered variant that we present is called \emph{column-oriented}
Cholesky. In the framework of our recursive formulations, this variant
is based on repartitioning the matrix after the elimination of each
column. We begin with a partitioning of $A$ into the first row and
column and the rest of the matrix,
\[
A=\begin{bmatrix}A_{1,1} & A_{2\colon n,1}^{T}\\
A_{2\colon n,1} & A_{2\colon n,2\colon n}
\end{bmatrix}\;.
\]
We compute $L_{1,1}=\sqrt{A_{1,1}}$ and divide $A_{2\colon n,1}$
to by the root to obtain $L_{2\colon n,1}$. Now comes the key step.
Instead of computing all the Schur complement $A_{2\colon n,2\colon n}-L_{2\colon n,1}L_{2\colon n,1}^{T}$,
we compute only its first column, $A_{2\colon n,2}-L_{2\colon n,1}L_{2,1}^{T}$.
The first column of the Schur complement allows us to compute the
second column of $L$. At this point we have computed the first two
columns of $L$ and nothing else. We now view the partitioning of
$A$ as 
\[
A=\begin{bmatrix}A_{1:2,1:2} & A_{3\colon n,1:2}^{T}\\
A_{3\colon n,1:2} & A_{3\colon n,3\colon n}
\end{bmatrix}=\begin{bmatrix}L_{1:2,1:2}\\
L_{3\colon n,1:2} & L_{3\colon n,3\colon n}
\end{bmatrix}\begin{bmatrix}L_{1:2,1:2}\\
L_{3\colon n,1:2} & L_{3\colon n,3\colon n}
\end{bmatrix}^{T}\;.
\]
We have already computed $L_{1:2,1:2}$ and $L_{3\colon n,1:2}$.
We still need to compute $L_{3\colon n,3\colon n}$. We do so in the
same way: computing the first column of the Schur complement $A_{3\colon n,3}-L_{3\colon n,1:2}L_{3,1\colon2}^{T}$
and eliminating it. The algorithm, ignoring sparsity issues, is shown
in Figure~XXX.

\begin{figure}
\begin{raggedright}
for $j=1\colon n$
\par\end{raggedright}
\begin{raggedright}
$\quad$$S_{j\colon n}=A_{j\colon n,j}-L_{j\colon n,1:j-1}L_{j,1\colon j-1}^{T}$
\par\end{raggedright}
\begin{raggedright}
$\quad$$L_{j,j}=\sqrt{S_{j}}$
\par\end{raggedright}
\begin{raggedright}
$\quad$$L_{j+1\colon n,j}=S_{j+1\colon n}/L_{j,j}$
\par\end{raggedright}
\begin{raggedright}
end
\par\end{raggedright}
\begin{lyxcode}
\end{lyxcode}
\caption{\label{fig:column Cholesky}Column-oriented Cholesky. The vector $S$
is a temporary vector that stores a column of the Schur complement.
By definition, operations involving a range $i\colon j$ of rows or
columns for $i<1$ or $j>n$ are not performed at all. (This allows
us to drop the conditional that skips the computation of the Schur-complement
column for $j=1$ and the computation of the nonexisting subdiagonal
part of the last column of $L$; in an actual algorithm, these conditionals
would be present, or the processing of $j=1$ and $j=n$ would be
performed outside the loop.)}
\end{figure}

We now present a data structure for the efficient implementation of
sparse column-oriented Cholesky. Our only objective in presenting
this data structure is show that sparse Cholesky can be implemented
in time proportional to arithmetic operations and in space proportional
to the number of nonzeros in $A$ and $L$. The data structure that
we present is not the best possible. Using more sophisticated data
structures, the number of operations in the factorization and the
space required can be reduced further, but not asymptotically.

We store $A$ using an array of $n$ linked lists. Each linked list
stores the diagonal and subdiagonal nonzeros in one column of $A$.
Each structure in the linked-list stores the value of one nonzero
and its row index. There is no need to store the elements of $A$
above its main diagonal, because (1) $A$ is symmetric, and (2) column-oriented
Cholesky never references them. There is also no need to store the
column indices in the linked-list structures, because each list stores
elements from only one column. The linked lists are ordered from low
row indices to high row indices.

We store the already-computed columns of $L$ redundantly. One part
of the data structure, which we refer to as \texttt{columns}, stores
the columns of $L$ in exactly the same way we store $A$. The contents
of the two other parts of the data structure, called \texttt{cursors}
and \texttt{rows}, depend on the number of columns that have already
been computed. Immediately before step $j$ begins, these parts contains
the following data. \texttt{Cursors} is an array of $n$ pointers.
The first $j-1$ pointers, if set, point to linked-list elements of
\texttt{columns}. $\mathtt{Cursors}_{i}$ points to the first element
with row index larger than $j$ in $\mathtt{columns}_{i}$, if there
is such an element. If not, it is not set (that is, it contains a
special invalid value). The other $n-j+1$ elements of \texttt{cursors}
are not yet used. Like \texttt{columns}, \texttt{rows} is an array
of linked list. The $i$th list stores the elements of $L_{i,1\colon j-1}$,
but in reverse order, from high to low column indices. Each element
in such a list contains a nonzero value and its column index.

The column $S$ of the Schur complement is represented by a data structure
\texttt{s} called a \emph{sparse accumulator}. This data structure
consists of an array \texttt{s.values} of $n$ real or complex numbers,
an array \texttt{s.rowind} of $n$ row indices, and array \texttt{s.exists}
of booleans, and an integer \texttt{s.nnz} that specifies how many
rows indices are actually stored in \texttt{s.rowind}.

Here is how we use these data structures to compute the factorization.
Step $j$ begins by copying $A_{j\colon n,j}$ to \texttt{s}. To do
so, we start by setting \texttt{s.nnz} to zero. We then traverse the
list that represents $A_{j\colon n,j}$. For a list element that represents
$A_{i,j}$, we increment \texttt{s.nnz}, store $i$ in $\mathtt{s.rowind_{s.nnz}}$
store $A_{ij}$ in $\mathtt{s.values}_{i}$, and set $\mathtt{s.exists}_{i}$
to a true value.

Our next task is to subtract $L_{j\colon n,1:j-1}L_{j,1\colon j-1}^{T}$
from $S$. We traverse $\mathtt{rows}_{j}$ to find the nonzero elements
in $L_{j,1\colon j-1}^{T}$. For each nonzero $L_{j,k}$ that we find,
we subtract $L_{j\colon n,k}L_{k,j}^{T}=L_{j,k}L_{j\colon n,k}$ from
$S$. To subtract, we traverse $\mathtt{columns}_{k}$ starting from
$\mathtt{cursors}_{k}$. Let $L_{i,k}$ be a nonzero found during
this traversal. To subtract $L_{j,k}L_{i,k}$ from $S_{i}$, we first
check whether $\mathtt{s.exists}_{i}$ is true. If it is, we simply
subtract $L_{j,k}L_{i,k}$ from $\mathtt{s.values}_{i}$. If not,
then $S_{i}$ was zero until now (in step $j$). We increment \texttt{s.nnz},
store $i$ in $\mathtt{s.rowind_{s.nnz}}$ store $0-L_{j,k}L_{i,k}$
in $\mathtt{s.values}_{i}$, and set $\mathtt{s.exists}_{i}$ to a
true. During the traversal of $\mathtt{columns}_{k}$, we many need
to advance $\mathtt{cursors}_{k}$ to prepare it for subsequent steps
of the algorithm. If the first element that we find in the travelral
has row index $j$, we advance $\mathtt{cursors}_{k}$ to the next
element in $\mathtt{columns}_{k}$. Otherwise, we do not modify $\mathtt{columns}_{k}$.

Finally, we compute $L_{j\colon n,j}$ and insert it into the nonzero
data structure that represents $L$. We replace $\mathtt{s.values}_{j}$
by its square root. For each rows index $i\neq j$ stored in one of
the first $\mathtt{s.nnz}$ entries of $\mathtt{s.rowind}$, we divide
$\mathtt{s.values}_{i}$ by $\mathtt{s.values}_{j}$. We now sort
elements $1$ though $\mathtt{s.nnz}$ of $\mathtt{s.rowind}$, to
ensure that the elements $\mathtt{columns}_{j}$ are linked in ascending
row order. We traverse the row indices stored in $\mathtt{s.rowind}$.
For each index $i$ such that $\mathtt{s.values}_{i}\neq0$, we allocate
a new element for $\mathtt{columns}_{j}$, link it to the previous
element that we created, and store in it $i$ and $\mathtt{s.values}_{i}$.
We also set $\mathtt{s.exists}_{i}$ to false, to prepare it for the
next iteration. 

We now analyze the number of operations and the storage requirements
of this algorithm.
\begin{thm}
\label{thm:complexity of sparse chol algorithms}The total number
of operations that the sparse-cholesky algorithm described above performs
is $\Theta(\phi(A))$. The amount of storage that the algorithm uses,
including the representation of its input and output, is $\Theta(n+\eta(A)+\eta(A))$.
\end{thm}

\begin{proof}
Let's start with bounding work. Step $j$ starts with copying column
$j$ of $A$ into the sparse accumulator, at a total cost of $\Theta(1+\eta(A_{j\colon n,j}))$.
No aritmetic is performed, but we can charge these operations to subsequent
arithmetic operations. If one of these values is modified in the accumulator,
we charge the copying to the subsequent multiply-subtract. If not,
we charge it to either the square root of the diagonal element or
to the division of subdiagonal elements. We cannot charge all the
copying operations to roots and divisions, becuase some of the copied
elements might get canceled before they are divided by $L_{j,j}$. 

The traversal of $\mathtt{rows}_{j}$ and the subtractions of scaled
columns of $L$ from the accumulator are easy to account for. The
processing of an element of $\mathtt{rows}_{j}$ is charged to the
modification of the diagonal, $S_{j}=S_{j}-L_{j,k}^{2}$. The traversal
of the suffix of $\mathtt{columns}_{k}$ performs $2\eta(L_{j\colon n,k})$
arithmetic operations and$\Theta(\eta(L_{j\colon n,k}))$ non-arithmetic
operations, so all operations are accounted for. 

After the column of the Schur complement is computed, the algorithm
computes a square root, scales $\eta(L_{j\colon n,j})-1$ elements,
sorts the indices of $L_{j\colon n,j}$, and creates a linked-list
to hold the elements of $L_{j\colon n,j}$. The total number of operations
in these computations is clearly $O(\eta^{2}(L_{j+1\colon n,j}))$,
even if we use a quadratic sorting algorithm, so by using Theorem~\ref{thm:ops in sparse cholesky},
we conclude that the operation-count bound holds.

Bounding space is easier. Our data structure includes a few arrays
of length $n$ and linked lists. Every linked-list element represents
a nonzero in $A$ or in $L$, and every nonzero is represented by
at most two linked-list elements. Therefore, the total storage requirements
is $\Theta(n+\eta(A)+\eta(A))$.
\end{proof}
The theorem shows that the only way to asymptotically reduce the total
number of operations in a sparse factorization is to reduce the number
of arithmetic operations. Similarly, to asymptotically reduce the
total storage we must reduce the fill. There are many ways to optimize
and improve both the data structure and the algorithm that we described,
but these optimizations can reduce the number of operations and the
storage requirements only by a constant multiplicative constant.

Improvements in this algorithm range from simple data-structure optimizations,
through sophisticated preprocessing steps, to radical changes in the
representation of the Schur complement. The most common data-structure
optimization, which is used by many sparse factorization codes, is
the use of \emph{compressed-column storage}. In this data structure,
all the elements in all the $\mathtt{columns}$ list are stored in
two contiguous arrays, one for the actual numerical values and another
for the row indices. A integer third array of size $n+1$ points to
the beginning of each column in these arrays (and to the location
just past column $n$). Preprocessing steps can upper bound the number
of nonzeros in each column of $L$ (this is necessary for exact preallocation
of the compressed-column arrays) and the identity of potential nonzeros.
The prediction of fill in $L$ can can eliminate the conditional in
the inner loop that updates the sparse accumulator; this can significantly
speed up the computation. The preprocessing steps can also construct
a compact data structured called the \emph{elimination tree} of $A$
that can be used for determining the nonzero structure of a row of
$L$ without maintaining the $\mathtt{rows}$ lists. This also speeds
up the computation significantly. The elimination tree has many other
uses. Finally, \emph{multifrontal} factorization algorithms maintain
the Schur complement in a completely different way, as a sum of sparse
symmetric update matrices.

The number of basic computational operations in an algorithm is not
an accurate predictor of its running time. Different algorithms have
different mixtures of operations and different memory-access patterns,
and these factors affect running times, not only the total number
of operations. We have seen evidence for this in Chapter~XXX, were
we have seen cases where direct solvers run at much higher computational
rates than iterative solvers. But to develop fundamental improvements
in algorithms, it helps to focus mainly on operation counts, and in
particular, on asymptotic operation counts. Once new algorithms are
discovered, we can and should optimize them both in terms of operation
counts and in terms of the ability to exploit cache memories, multiple
processors, and other architectural features of modern computers.
But our primary concern here is asymptotic operation counts.

\section{Characterizations of Fill}

If we want to reduce fill, we need to characterize fill. Definition~\ref{def:fill}
provides an characterization, but this characterization is not useful
for predicting fill before it occurs. One reason that predicting fill
using Definition~\ref{def:fill} is that it for cancellations, which
are difficult to predict. In this section we provide two other characterizations.
They are not exact, in the sense that they characterize a superset
of the fill. In many cases these characterizations are exact, but
not always. On the other hand, these charactersizations can be used
to predict fill before the factorization begins. Both of them are
given in terms of graphs that are related to $A$, not in terms of
matrices.
\begin{defn}
The \emph{pattern graph} of an $n$-by-$n$ symmetric matrix $A$
is an undirected graph $G_{A}=(\{1,2,\ldots n\},E)$ whose vertices
are the integers $1$ through $n$ and whose edges are pairs $(i,j)$
such that $A_{i,j}\neq0$. 
\end{defn}

\begin{defn}
Let $G$ be an undirected graph with vertices $1,2,\ldots,n$. A \emph{vertex
elimination step} on vertex $j$ of $G$ is the transformation of
$G$ into another undirected graph $\textrm{eliminate}(G,j)$. The
edge set of $\textrm{eliminate}(G,j)$ is the union of the edge set
of $G$ with a clique on the neighbors of $j$ in $G$ whose indices
are larger than $j$,
\begin{eqnarray*}
E(\textrm{eliminate}(G,j))=E(G)\cup\{(i,k) & | & i>j\\
 &  & k>j\\
 &  & (j,i)\in E(G)\\
 &  & (k,i)\in E(G)\}\;.
\end{eqnarray*}
The \emph{fill graph} $G_{A}^{+}$ of $A$ is the graph
\[
G_{A}^{+}=\textrm{eliminate}(\textrm{eliminate}(\cdots\textrm{eliminate}(G_{A},1)\cdots),n-1),n)\;.
\]
\end{defn}

The edges of the fill graph provide a bound on fill
\begin{lem}
\label{lemma: elimination game}Let $A$ be an $n$-by-$n$ symmetric
positive-definite matrix. Let $1\leq j<i\leq n$. If $A_{i,j}\neq0$
or $i,j$ is a fill element, then $(i,j)$ is an edge of $G_{A}^{+}$.
\end{lem}

\begin{proof}
The elimination of a vertex only adds edges to the graph. The fill
graph is produced by a chain of vertex eliminations, starting from
the graph of $A$. If $A_{i,j}\neq0$, then $(i,j)$ is an edge of
the graph of $A$, and therefore also an edge of its fill graph.

We now prove by induction on $j$ that a fill element $i,j$ is also
an edge $(i,j)$ in 
\[
\textrm{eliminate}(\textrm{eliminate}(\cdots\textrm{eliminate}(G_{A},1)\cdots),j-2),j-1)\;.
\]
The claim hold for $j=1$ because $L_{i,1}\neq0$ if and only if $A_{i,1}\neq0$.
Therefore, there are no fill edges for $j=1$, so the claims holds.
We now prove the claim for $j>1$. Assume that $i,j$ is a fill element,
and let $S^{(j)}$ be the Schur complement just prior to the $j$th
elimination step,
\[
S^{(j)}=A_{j\colon n,j\colon n}-L_{j\colon n,1\colon j-1}L_{j\colon n,1\colon j-1}^{T}=A_{j\colon n,j\colon n}-\sum_{k=1}^{j-1}L_{j\colon n,k}L_{j\colon n,k}^{T}\;.
\]
Since $i,j$ is a fill element, $S_{i,j}^{(j)}\neq0$ but $A_{i,j}=0$.
Therefore, $L_{i,k}L_{j,k}\neq0$ for some $k<j$. This implies that
$L_{i,k}\neq0$ and $L_{j,k}\neq0$. By induction, $(i,k)$ and $(j,k)$
are edges of 
\[
\textrm{eliminate}(\cdots\textrm{eliminate}(G_{A},1)\cdots),k-1)\;.
\]
Therefore, $(i,j)$ is an edge of 
\[
\textrm{eliminate}(\textrm{eliminate}(\cdots\textrm{eliminate}(G_{A},1)\cdots),k-1),k)\;.
\]
Since edges are never removed by vertex eliminations, $(i,j)$ is
also an edge of 
\[
\textrm{eliminate}(\textrm{eliminate}(\cdots\textrm{eliminate}(G_{A},1)\cdots),j-2),j-1)\;.
\]
This proves the claim and the entire lemma.
\end{proof}
The converse is not true. An element in a Schur complement can fill
and then be canceled out, but the edge in the fill graph remains.
Also, fill in a Schur complement can cancel exactly an original element
of $A$, but the fill graph still contains the corresponding edge.

The following lemma provides another characterization of fill.
\begin{lem}
\label{lemma: fill paths}The pair $(i,j)$ is an edge of the fill
graph of $A$ if and only if $G_{A}$ contains a simple path from
$i$ to $j$ whose internal vectices all have indices smaller than
$\min(i,j)$.
\end{lem}

\begin{proof}
Suppose that $G_{A}$ contains such a path. We prove that $(i,j)$
edge of the fill graph by induction on the length of the path. If
path contains only one edge then $(i,j)$ is an edge of $G_{A}$,
so it is also an edge of $G_{A}^{+}$. Suppose that the claim holds
for paths of length $\ell-1$ or shorter and that a path of length
$\ell$ connects $i$ and $j$. Let $k$ be the smallest-index vertex
in the path. The elimination of vertex $k$ adds an edge connecting
its neighbors in the path, because their indices are higher than $k$.
Now there is a path of length $\ell-1$ between $i$ and $j$; the
internal vertices still have indices smaller than $\min(i,j)$. By
induction, future elimination operations on the graph will create
the fill edge $(i,j)$.

To prove the other direction, suppose that $(i,j)$ is an edge of
$G_{A}^{+}$. If $(i,j)$ is an edge of $G_{A}$, a path exists. If
not, the edge $(i,j)$ is created by some elimination step. Let $k$
be the vertex whose elimination creates this edge. We must have $k<\min(i,j)$,
otherwise the $k$th elimination step does not create the edge. Furthermore,
when $k$ is eliminated, the edges $(k,i)$ and $(k,j)$ exist. If
they are original edge of $G_{A}$, we are done---we have found the
path. If not, we use a similar argument to show that there must be
suitable paths from $i$ to $k$ and from $k$ to $j$. The concatenation
of these paths is the sought after path.
\end{proof}

\section{Perfect and Almost-Perfect Elimination Orderings}

The Cholesky factorization of symmetric positive definite matrices
is numerically stable, and symmetrically permuting the rows and columns
of an \noun{spd} matrix yields another \noun{spd} matrix. Therefore,
we can try to symmetrically permute the rows and columns of a sparse
matrix to reduce fill and work in the factorization. We do not have
to worry that the permutation will numerically destabilize the factorization.
In terms of the graph of the matrix, a symmetric row and column pemutation
corresponds to relabeling the vertices of the graphs. In otherwords,
given an undirected graph we seek an elimination ordering for its
vertices.

Some graphs have elimination orderings that generate no fill, so that
$G_{A}^{+}=G_{A}$. Such orderings are called \emph{perfect-elimination
orderings}. The most common example are trees.
\begin{lem}
\label{lemma: perfect orderings for trees}If $G_{A}$ is a tree or
a forest, then $G_{A}^{+}=G_{A}$.
\end{lem}

\begin{proof}
On a tree or forest, any depth-first-search postorder is a perfect-elimination
ordering. 

By Lemma~\ref{lemma: fill paths}, all the fill edges occur within
connected components of $G_{A}$. Therefore, it is enough to show
that each tree in a forest has a perfect-elimination ordering. We
assume from here on that $G_{A}$ is a tree.

Let $v$ be an arbitrary vertex of $G_{A}$. Perform a depth-first
traversal of $G_{A}$ starting from the root $v$, and number the
vertices in postorder: give the next higher index to a vertex $v$
immediately after the traversal returns from the last child of $v$,
starting from index $1$. Such an ordering guarantees that in the
rooted tree rooted at $v$, a vertex $u$ has a higher index than
all the vertices in the subtree rooted at $u$.

Under such an ordering, the elimination of a vertex $u$ creates no
fill edges at all, because $u$ has at most one neighbor with a higher
index, its parent in the rooted tree.
\end{proof}
Most graphs do not have perfect-elimination orderings, but some orderings
are almost as good. The elimination of a vertex with only one higher-numbered
neighbor is called a perfect elimination step, because it produces
no fill edge. Eliminating a vertex with two higher-numbered neighbors
is not perfect, but almost: it produces one fill edge. If $G_{A}$
contains an isolated path
\[
v_{0}\leftrightarrow v_{1}\leftrightarrow v_{2}\leftrightarrow\cdots\leftrightarrow v_{\ell}\leftrightarrow v_{\ell+1}
\]
such that the degree of $v_{1},v_{2},\ldots,v_{\ell}$ in $G_{A}$
is $2$, then we can eliminate $v_{1},\ldots,v_{\ell}$ (in any order),
producing only $\ell$ fill edges and performing only $\Theta(\ell)$
operations. This observation is useful if we try to sparsify a graph
so that the sparsified graph has a good elimination ordering.

\section{Symbolic Elimination and the Elimination Tree}

We can use the fill-characterization technique that Lemma~\ref{lemma: elimination game}
describes to create an efficient algorithm for predicting fill. Predicting
fill, or even predicting just the nonzero counts in rows and columns
of the factor, can lead to more efficient factorization algorithms.
The improvements are not asymptotic, but they can be nonetheless significant.

The elimination of vertex $k$ adds to the graph the edges of a clique,
a complete subgraph induced by the higher-numbered neighbors of $k$.
Some of these edges may have already been part of the graph, but some
may be new. If we represent the edge set of the partially-eliminated
graph using a \emph{clique-cover} data structure, we efficiently simulate
the factorization process and enumerate the fill edges. An algorithm
that does this is called a \emph{symbolic elimination} algorithm.
It is symbolic in the sense that it simulates the sparse factorization
process, but without computing the numerical values of elements in
the Schur complement or the factor.

A clique cover represents the edges of an undirected graph using an
array of linked lists. Each list specifies a clique by specifying
the indices of a set of vertices: each link-list element specifies
one vertex index. The edge set of the graph is the union of the cliques.
We can create a clique-cover data structure by creating one two-vertex
clique for each edge in the graph.

A clique cover allows us to simulate elimination steps efficiently.
We also maintain an array of vertices. Each element in the vertex
array is a doubly-linked list of cliques that the vertex participates
in. To eliminate vertex $k$, we need to create a new clique containing
its higher-numbered neighbors. These neighbors are exactly the union
of higher-than-$k$ vertices in all the cliques that $k$ participates
in. We can find them by traversing the cliques that $k$ participates
in. For each clique $q$ that $k$ participates in, we traverse $q$'s
list of vertices and add the higher-than-$k$ vertices that we find
to the new clique. We add a vertex to the new clique at most once,
using a length-$n$ bit vector to keep track of which vertices have
already been added to the new clique. Before we move on to eliminate
vertex $k+1$, we clear the bits that we have set in the bit vector,
using the vertices in the new clique to indicate which bits must be
cleared. For each vertex $j$ in the new clique, we also add the new
clique to the list of cliques that $j$ participates in.

The vertices in the new clique, together with $k$, are exactly the
nonzero rows in column $k$ of a cancellation-free Cholesky factor
of $A$. Thus, the symbolic-elimination algorithm predicts the structure
of $L$ if there are no cancellations.

We can make the process more efficient by not only creating new cliques,
but merging cliques. Let $q$ be a clique that $k$ participates in,
and let $i>k$ and $j>k$ be vertices in $q$. The clique $q$ represents
the edge $(i,j)$. But the new clique that the elimination of $k$
creates also represents $(i,j)$, because both $i$ and $j$ belong
to the new clique. Therefore, we can partially remove clique $q$
from the data structure. The removal makes the elimination of higher-than-$k$
vertices belonging to it cheaper, because we will not have to traverse
it again. To remove $q$, we need each element of the list representing
$q$ to point to the element representing $q$ in the appropriate
vertex list. Because the vertex lists are doubly-linked, we can remove
elements from them in constant time.

Because each clique either represents a single nonzero of $A$ or
the nonzeros in a single column of a cancellation-free factor $L$,
the total size of the data structure is $\Theta(\eta(A)+\eta(L)+n)$.
If we do not need to store the column structures (for example, if
we are only interested in nonzero counts), we can remove completely
merged cliques. Because we create at most one list element for each
list element that we delete, the size of the data structure in this
scheme is bounded by the size of the original data structure, which
is only $\Theta(\eta(A)+n)$.

The number of operations in this algorithm is $\Theta(\eta(A)+\eta(L)+n)$,
which is often much less than the cost of the actual factorization,
\[
\phi(A)=\sum_{j=1}^{n}O\left(\eta^{2}\left(L_{j+1\colon n}\right)\right)\;.
\]
To see that this is indeed the case, we observe that each linked-list
element is touched by the algorithm only twice. An element of a clique
list is touched once when it is first created, and another time when
the clique is merged into another clique. Because the clique is removed
from vertex lists when it is merged, it is never referenced again.
An element of a vertex list is also touched only once after it is
created. Either the element is removed from the vertex's list before
the vertex is eliminated, or it is touched during the traveral of
the vertex's list.

The process of merging the cliques defines an important structure
that plays an important role in many sparse-matrix factorization algorithms,
the \emph{elimination tree} of $A$. To define the elimination tree,
we name some of the cliques. Cliques that are created when we initialize
the clique cover to represent the edge set of $G_{A}$ have no name.
Cliques that are formed when we eliminate a vertex are named after
the vertex. The elimination tree is a rooted forest on the vertices
$\{1,2,\ldots,n\}$. The parent $\pi(k)$ of vertex $k$ is the name
of the clique into which clique $k$ is merged, if it is. If clique
$k$ is not merged in the symbolic-elimination process, $k$ is a
root of a tree in the forest. There are many alternative definitions
of the elimination tree. The elimination tree has many applications.
In particular, it allows us to compute the number of nonzeros in each
row and column of a cancellation-free Cholesky factor in time almost
linear in $\eta(A)$, even faster than using symbolic elimination.

\section{Minimum-Degree Orderings}

The characterization of fill in Lemma~\ref{lemma: elimination game}
also suggests an ordering heuristic for reducing fill. If the elimination
of a yet-uneliminated vertex creates a clique whose size is the number
of the uneliminated neighbors of the chosen vertex, it makes sense
to eliminate the vertex with the fewer uneliminated neighbors. Choosing
this vertex minimizes the size of the clique that is created in the
next step and minimizes the arithmetic work in the next step. Ordering
heuristics based on this idea are called minimum-degree heuristics.

There are two problems in the minimum-degree idea. First, not all
the edges in the new cliques are new; some of them might be part of
$G_{A}$ or might have already been created by a previous elimination
step. It is possible to minimize not the degree but the actual new
fill, but this turns out to be more expensive algorithmically. Minimizing
fill in this way also turns out to produce orderings that are not
significantly better than minimum-degree orderings. Second, an optimal
choice for the next vertex to eliminate may be suboptimal globally.
There are families of matrices on which minimum-degree orderings generate
asymptotically more fill than the optimal ordering. Minimum-degree
and minimum-fill heuristics are greedy and myopic. They select vertices
for elimination one at a time, and the lack of global planning can
hurt them. This is why minimum fill is often not much better than
minimum fill.

Even though minimum-degree algorithms are theoretically known to be
suboptimal, they are very fast and often produce effective orderings.
On huge matrices nested-dissection orderings, which we discuss next,
are often more effective than minimum-degree orderings, but on smaller
matrices, minimum-degree orderings are sometimes more effective in
reducing fill and work.

Minimum-degree algorithms usually employ data structures similar to
the clique cover used by the symbolic elimination algorithm. The data
structure is augmented to allow vertex degrees to be determined or
approximated. Maintaining exact vertex degrees is expensive, since
a vertex cover does not represent vertex degrees directly. Many successful
minimum-degree algorithms therefore use degree approximations that
are cheaper to maintain or compute. Since the minimum-degree is only
a heuristic, these approximations are not neccessarily less effective
in reducing fill than exact degrees; sometimes they are more effective.

\section{Nested-Dissection Orderings for Regular Meshes}

Nested-dissection orderings are known to be approximately optimal,
and on huge matrices that can be significantly more effective than
other ordering heuristics. On large matrices, the most effective orderings
are often nested-dissection/minimum-degree hybrids.

Nested-dissection orderings are defined recursively using vertex subsets
called separators. 
\begin{defn}
Let $G=(V,E)$ be an undirected graph with $|V|=n$. Let $\alpha$
and $\beta$ be real positive constants, and let $f$ be a real function
over the positive integers. An $(\alpha,\beta,f)$ \emph{vertex separator}
in $G$ is a set $S\subseteq V$ of vertices that satisfies the following
conditions.

\begin{description}
\item [{Separation:}] There is a partition $V_{1}\cup V_{2}=V\setminus S$
such that for any $v\in V_{1}$ and $u\in V_{2}$, the edge $(u,v)\not\in E$.
\item [{Balance:}] $|V_{1}|,|V_{2}|\leq\alpha n$.
\item [{Size:}] $|S|\leq\beta f(n)$.
\end{description}
\end{defn}

Given a vertex separator $S$ in $G_{A}$, we order the rows and columns
of $A$ of the separator last, say in an arbitrary order (but if $G\setminus S$
contains many connected components then a clever ordering of $S$
can further reduce fill and work), and the rows and columns corresponding
to $V_{1}$ and $V_{2}$ first. By Lemma~\ref{lemma: fill paths},
this ensures that for any $v\in V_{1}$ and $u\in V_{2}$, the edge
$(u,v)$ is \emph{not} a fill edge, so $L_{u,v}=L_{v,u}=0$. Therefore,
the interleaving of vertices of $V_{1}$ and $V_{2}$ has no effect
on fill, so we can just as well order all the vertices of $V_{1}$
before all the vertices of $V_{2}$.

The function of the separator in the ordering is to ensure that $L_{u,v}=0$
for any $v\in V_{1}$ and $u\in V_{2}$. Any ordering in which the
vertices of $V_{1}$ appear first, followed by the vertices of $V_{2}$,
followed by the vertices of $S$, ensures that a $|V_{1}|$-by-$|V_{2}|$
rectangular block in $L$ does not fill. A good separator is one for
which this block is large. The size of $S$ determines the circumference
of this rectangular block, because half the circumference is $|V_{1}|+|V_{2}|=n-|S|$.
The size imbalance between $V_{1}$ and $V_{2}$ determines the aspect
ratio of this rectangle. For a given circumference, the area is maximized
the rectangle is as close to square as possible. Therefore, a small
balanced separator reduces fill more effectively than a large or unbalanced
separator.

By using separators in the subgraphs of $G_{A}$ induced by $V_{1}$
and $V_{2}$ to order the diagonal blocks of $A$ that correspond
to $V_{1}$ and $V_{2}$, we can avoid fill in additional blocks of
$L$. These blocks are usually smaller than the top-level $|V_{1}|$-by-$|V_{2}|$
block, but they are still helpful and significant in reducing the
total fill and work. If $|V_{1}|$ or $|V_{2}|$ are small, say smaller
than a constant threshold, we order the corresponding subgraphs arbitrarily. 

Let us see how this works on square two-dimensional meshes. To keep
the analysis simple, we use a cross-shaped separator that partitions
the mesh into four square or nearly-square subgraphs. We assume that
$G_{A}$ is an $n_{x}$-by-$n_{y}$ where $n_{x}n_{y}=n$ and where
$|n_{x}-n_{y}|\leq1$. The size of a cross-shaped separator in $G_{A}$
is $|S|=n_{x}+n_{y}-1<2\sqrt{n}$. To see this, we note that if $n_{x}=n_{y}=\sqrt{n}$
then $|S|=2\sqrt{n}-1$. Otherwise, without loss of generality $n_{x}=n_{y}-1<n_{y}$,
so $|S|=2n_{x}=2\sqrt{n_{x}n_{x}}<2\sqrt{n_{x}n_{y}}=2\sqrt{n}$.
The separator breaks the mesh into four subgraphs, each of which is
almost square (their $x$ and $y$ dimensions differ by at most $1$),
of size at most $n/4$. This implies that throught the recursion,
each size-$m$ subgraph that is produced by the nested-dissection
ordering has a $(0.25,0.5,\sqrt{m})$ $4$-way separator that we use
in the ordering.

We analyze the fill and work in the factorization by blocks of columns.
Let $S$ be the top-level separator and let $V_{1},V_{2},\ldots,V_{4}$
be the vertex sets of the separated subgraphs. We have 
\begin{eqnarray*}
\eta(L) & = & \eta\left(L_{\colon,S}\right)+\eta\left(L_{\colon,V_{1}}\right)+\eta\left(L_{\colon,V_{2}}\right)+\eta\left(L_{\colon,V_{3}}\right)+\eta\left(L_{\colon,V_{4}}\right)\\
 & \leq & \frac{|S|\left(|S|+1\right)}{2}+\eta\left(L_{\colon,V_{1}}\right)+\eta\left(L_{\colon,V_{2}}\right)+\eta\left(L_{\colon,V_{3}}\right)+\eta\left(L_{\colon,V_{4}}\right)\;.
\end{eqnarray*}
Bounding the fill in a block of columns that corresponds to one of
the separated subgraphs is tricky. We cannot simply substitute a similar
expression to form a reccurence. The matrix $A$ is square, but when
we analyze fill in one of these column blocks, we need to account
for fill in both the square diagonal block and in the subdiagonal
block. If $u\in S$ and $v\in V_{1}$ and $A_{u,v}\neq0$, then elements
in $L_{u,V_{1}}$ can fill. A accurate estimate of how much fill occurs
in such blocks of the factor is critical to the analysis of nested
dissection algorithms. If we use the trivial bound $\eta(L_{u,V_{1}})\leq n/4$,
we get an asymptotically loose bound $\eta(L)=O(n^{1.5})$ on fill.

To achieve an asymptotically tight bound, we set up a recurrence for
fill in a nested-dissection ordering of the mesh. Let $\bar{\eta}(m_{x},m_{y})$
be a bound on the fill in the columns corresponding to an $m_{x}$-by-$m_{y}$
submesh (with $|m_{x}-m_{y}|\leq1$). At most $2m_{x}+2m_{y}$ edges
connect the submesh to the rest of the mesh. Therefore we have
\[
\bar{\eta}(m_{x},m_{y})\leq\frac{\left(m_{x}+m_{y}-1\right)\left(m_{x}+m_{y}\right)}{2}+\left(m_{x}+m_{y}-1\right)\left(2m_{x}+2m_{y}\right)+4\bar{\eta}\left(\frac{m_{x}}{2},\frac{m_{y}}{2}\right)\;.
\]
The first term in the right-hand side bounds fill in the separator's
diagonal block. The second term, which is the crucial ingredient in
the analysis, bounds fill in the subdiagonal rows of the separator
columns. The third term bounds fill in the four column blocks corresponding
to the subgraphs of the separated submesh. If we cast the recurrence
in terms of $m=m_{x}m_{y}$ we obtain 
\[
\bar{\eta}(m)\leq\Theta\left(m\right)+4\bar{\eta}\left(\frac{m}{4}\right)\;.
\]
By the Master Theorem~CLR2ed\_Thm\_4.1, The solution of the recurrence
is $\bar{\eta}(m)=O(m\log m)$. Since $\eta(L)\leq\bar{\eta}(n)$,
we have $\eta(L)=O(n\log n)$. We can analyze work in the factorization
in a similar way. We set up a recurrence
\[
\bar{\phi}(m)\leq\Theta\left(m^{1.5}\right)+4\bar{\phi\left(\frac{m}{4}\right)\;,}
\]
whose solution leads to $\phi(A)=O(n\sqrt{n})$. It is possible to
show that these two bounds are tight and that under such a nested-dissection
ordering for a two-dimensional mesh, $\eta(L)=\Theta(n\log n)$ and
$\phi(A)=\Theta(n\sqrt{n})$.

For a three-dimensional mesh, a similar analysis yields $\eta(L)=\Theta(n^{4/3})$
and $\phi(A)=\Theta(n^{6/3})=\Theta(n^{2})$.

In practice, we can reduce the constants by stopping the recurrence
at fairly large subgraphs, say around $m=100$, and to order these
subgraphs using a minimum-degree ordering. This does not change the
asymptotic worst-case bounds on work and fill, but it usually improves
the actual work and fill counts. There are other nested-dissection/minimum-degree
hybrids that often improve the work and fill counts without hurting
the asymptotic worst-case bounds.
\begin{defn}
A class of graphs satisfies an $(\alpha,\beta,f)$ \emph{vertex-separator}
\emph{theorem} (or a separator theorem for short) if every $n$-vertex
graph in the class has an $(\alpha,\beta,f)$ vertex separator.
\end{defn}

\section{Generalized Nested-Dissection Orderings}

When $G_{A}$ is not a regular mesh, the analysis becomes much harder.
When applied to general graphs, the ordering framework that we described
in which a small balanced vertex separator is ordered last and the
connected components are ordered recursively is called \emph{generalized
nested dissection}.
\begin{defn}
A class of graphs satisfies an $(\alpha,\beta,f)$ \emph{vertex-separator}
\emph{theorem} (or a separator theorem for short) if every $n$-vertex
graph in the class has an $(\alpha,\beta,f)$ vertex separator.
\end{defn}

For example, planar graphs satisfy an $(2/3,\sqrt{8},\sqrt{n})$ separator
theorem. Since nested dissection orders subgraphs recursively, we
must ensure that subgraphs belong to the same class of graphs, so
that they can be ordered effectively as well. 
\begin{defn}
A class of graphs is said to be \emph{closed under subgraph} if every
subgraph of a graph in the class is also in the class.
\end{defn}

There are two ways to use small balanced vertex separators to order
an arbitrary graph. The first algorithm, which we call the \noun{lrt}
algorithm, guarantees an effective ordering for graphs belonging to
a class that satisfies a separator theorem and is closed under subgraph.
The algorithm receives an input a range $[i,j]$ of indices and graph
$G$ with $n$ vertices, $\ell$ of which may already have been ordered.
If $\ell>0$, then the ordered vertices have indices $j-\ell+1,\ldots,j$.
The algorithm assigns the rest of the indices, $i,\ldots,j-\ell$,
to the unnumbered vertices. The algorithm works as follows.
\begin{enumerate}
\item If $n$ is small enough, $n\leq(\beta(1-\alpha))^{2}$, the algorithm
orders the unnumbered vertices arbitrarily and returns.
\item The algorithm finds a small balanced vertex separator $S$ in $G$.
The separator separates the graph into subgraphs with vertex sets
$V_{1}$ and $V_{2}$. The two subgraphs may contain more than one
connected componet each.
\item The algorithm arbitrarily assigns indices $j-\ell-|S|+1,\ldots,j-\ell$
to the vertices of $S$.
\item The algorithm recurses on the subgraphs induced by $V_{1}\cup S$
and by $V_{2}\cup S$. The unnumbered vertices in the first subgraph
are assigned the middle range of $[i,j]$ and the second the first
range (starting from $i$).
\end{enumerate}
We initialize the algorithm by setting $\ell=0$ and $[i,j]=[1,n]$.

The second algorithm, called the \noun{gt} algorithm, finds a separator,
orders its vertices last, and then recurses on each connected component
of the graph with the separator removed; the vertices of each component
are ordered consecutively. In each recursive call the algorithm finds
a separator that separates the vertices of one component, ignoring
the rest of the graph. This algorithm does not guarantee an effective
orderings to any graph belonging to a class that satisfies a separator
theorem and is closed under subgraph; additional conditions on the
graphs are required, such as bounded degree or closure under edge
contractions. Therefore, we analyze the first algorithm.
\begin{thm}
Let $G$ be a graph belonging to a class that satisfies an $(\alpha,\beta,\sqrt{n})$
separator theroem and closed under subgraph. Ordering the vertices
of $G$ using the LRT algorithm leades to $O(n\log n)$ fill edges.
\end{thm}

\begin{proof}
We prove the theorem by induction on $n$. If $n\leq n_{0}$, the
theorem holds by setting the constant $c$ in the big-$O$ notation
high enough so that $cn\log n>n(n-1)/2$ for all $n\leq n_{0}$. Suppose
that the theorem holds for all graphs with fewer than $n$ vertices. 

The algorithm finds a separator $S$ that splits the graph into subgraphs
induced by $V_{1}$ and by $V_{2}$. We partition the fill edges into
four categories: (1) fill edges with two endpoints in $S$; (2) fill
edges with one endpoint in $S$ and the other in one of the $\ell$
already-numbered vertices; (3,4) fill edges with at least one endpoint
in $V_{2}$ or in $V_{2}$. Since there are no fill edges with one
endpoint in $V_{1}$ and the other in $V_{2}$, categories~3 and~4
are indeed disjoint.

The number of fill edges in Category~1 is at most $|S|(|S|-1)/2\leq\beta^{2}n/2$.

The number of fill edges in Category~2 is at most $|S|\ell\leq\beta\ell\sqrt{n}$.

Let $\bar{\eta}(n,\ell)$ be an upper bound on the number of fill
edges in the ordering produced by the algorithm on a graph with $n$
vertices, $\ell$ of which are already numbered.

The number of fill edges in Category~3 is at most $\bar{\eta}(|V_{1}|+|S|,\ell_{1})$,
where $\ell_{1}$ is the number of already-numbered vertices in $V_{1}\cup S$
\emph{after the vertices of $S$ have been numbered}. Note that $\ell_{1}$
may be smaller than $\ell+|S|$ because some of the $\ell$ vertices
that are initially numbered be in $V_{2}$. Similarly, the number
of fill edges in Category~4 is at most $\bar{\eta}(|V_{2}|+|S|,\ell_{2})$.

By summing the bounds for the four categories, we obtain a recurrence
on $\bar{\eta}$,
\begin{equation}
\bar{\eta}\left(n,\ell\right)\leq\frac{\beta^{2}n}{2}+\beta\ell\sqrt{n}+\bar{\eta}\left(|V_{1}|+|S|,\ell_{1}\right)+\bar{\eta}\left(|V_{2}|+|S|,\ell_{2}\right)\;.\label{eq:LRT fill recursion}
\end{equation}
We claim that 
\begin{equation}
\bar{\eta}\left(n,\ell\right)\leq c'(n+\ell)\log_{2}n+c''\ell\sqrt{n}\label{eq:LRT fill bound}
\end{equation}
for some constants $c'$ and $c''$. Since initially $\ell=0$, this
bound implies the $O(n\log n)$ bound on fill.

To prove the bound, we denote $n_{1}=|V_{1}|+|S|$ and $n_{2}=|V_{2}|+|S|$
and note the following bounds:
\[
\begin{array}{rcccl}
 &  & \ell_{1}+\ell_{2} & \leq & \ell+2\beta\sqrt{n}\\
n & \leq & n_{1}+n_{2} & \leq & n+\beta\sqrt{n}\\
(1-\alpha)n & \leq & n_{1},n_{2} & \leq & \alpha n+\beta\sqrt{n}
\end{array}
\]
The first inequality follows from the fact that every alread-numbered
vertex in the input to the two recursive calls is either one of the
$\ell$ initially-numbered vertices or a vertex of $S$. An initially-numbered
vertex that is not in $S$ is passed to only one of the recursive
calls; vertices in $S$ are be passed as already numbered to the two
calls, but there are at most $\beta\sqrt{n}$ of them. The second
inequality follows from the fact that the subgraphs passed to the
two recursive calls contain together all the vertices in $V=S\cup V_{1}\cup V_{2}$
and that $|S|\leq\beta\sqrt{n}$ vertices are passed to the both of
the recursive calls. The third inequality follows from the guarantees
on $|V_{1}|$, $|V_{2}|$, and $|S|$ under the separator theroem.

We now prove the claim (\ref{eq:LRT fill bound}) by induction on
$n$. For $n$ smaller than some constant size, we can ensure that
the claim holds simply by choosing $c'$ and $c''$ to be large enough.
In particular, $\bar{\eta}\left(n,\ell\right)\leq n(n-1)/2$, which
is smaller than the right-hand side of (\ref{eq:LRT fill bound})
for small enough $n$ and large enough $c'$ and $c''$.

For larger $n$, we use Equation~(\ref{eq:LRT fill recursion}) and
invoke the inductive assumption regarding the correctness of (\ref{eq:LRT fill bound}),
\begin{eqnarray*}
\bar{\eta}\left(n,\ell\right) & \leq & \frac{\beta^{2}n}{2}+\beta\ell\sqrt{n}+\bar{\eta}\left(|V_{1}|+|S|,\ell_{1}\right)+\bar{\eta}\left(|V_{2}|+|S|,\ell_{2}\right)\\
 & \leq & \frac{\beta^{2}n}{2}+\beta\ell\sqrt{n}+c'\left(n_{1}+\ell_{1}\right)\log_{2}n_{1}+c''\ell_{1}\sqrt{n_{1}}+c'\left(n_{2}+\ell_{2}\right)\log_{2}n_{2}+c''\ell_{2}\sqrt{n_{2}}\;.
\end{eqnarray*}
The rest of the proof only involves manipulations of the expression
in the second line to show that for large enough $c'$ and $c''$,
it is bounded by~(\ref{eq:LRT fill bound}). We omit the details.
\end{proof}
A similar analysis yields an analysis on arithmetic operations.
\begin{thm}
Let $G$ be a graph belonging to a class that satisfies an $(\alpha,\beta,\sqrt{n})$
separator theroem and closed under subgraph. Ordering the vertices
of $G$ using the LRT algorithm leads to $O(n^{1.5})$ arithmetic
operations in the algorithm.
\end{thm}

These results can be applied directly to matrices whose pattern graphs
are planar, and more generally to graphs that can be embedded in surfaces
with a bounded genus. Such graphs are closed under subgraph and satisfy
a $(2/3,\sqrt{8},\sqrt{n})$ separator theorem. Furthermore, the separator
in an $n$-vertex graph of this family can be found in $O(n)$ operations,
and it is possible to show that we can even find \emph{all} the separators
required in all the levels of the recursive algorithm in $O(n\log n)$
time.

For the \noun{gt} algorithm, which is somewhat easier to implement
and which finds separators faster (because the separator is not included
in the two subgraphs to be recursively partitioned), a separator theorem
is not suffiecient to guarantee comparable bounds on fill and work.
To ensure that the algorithm leads to similar asymptotic bounds, one
must also assume that the graphs are closed under edge contraction
or that they have a bounded degree. However, the two algorithms have
roughly the same applicability, because planar graphs and graphs that
can be embedded on surfaces with a bounded genus are closed under
edge contraction. It is not clear which of the two algorithms is better
in practice.

The fill and work results for both the \noun{lrt} and the \noun{gt}
algorithms can be extended to graphs that satisfy separator theorems
with separators smaller or larger than $\sqrt{n}$. See the original
articles for the asymptotic bounds.

Finally, we mention that an algorithm that finds approximately optimal
balanced vertex separators can be used to find a permutation that
approximately minimizes fill and work in sparse Cholesky. The algorithm
is similar in principle to the \noun{lrt} and \noun{gt} algorithms.
This result shows that up to a polylogarithmic factors, the quality
of vertex separators in a graph determines sparsity that we can achieve
in sparse Cholesky.

\section{Notes and References}

Bandwidth and envelope reduction orderings.

The idea of nested dissection, and the analysis of nested dissection
on regular meshes is due to George~XXX. The \noun{lrt} generalized-nested-dissection
algorithm and its analysis are due to Lipton, Rose, and Tarjan~XXX.
The \noun{gt} algorithm and its analysis are due to Gilbert and Tarjan~XXX.

\section*{Exercises}

\def\fact{\vcenter{\hbox{\includegraphics[width=1in]{matlab/plots/jik_chol_L.eps}}}}

\def\mat{\vcenter{\hbox{\includegraphics[height=1in]{matlab/plots/jik_chol_S.eps}}}}

\begin{xca}
The proof of Theorem~\ref{thm:ops in sparse cholesky} was essentially
based on the fact that the following algorithm is a correct implementation
of the Cholesky factorization:

\begin{lyxcode}
$S=A$

for~$j=1\colon n$

~~$L_{j,j}=\sqrt{S_{j,j}}$

~~$L_{j+1\colon n,j}=S_{j+1\colon n,j}/L_{j,j}$

~~for~$i=j+1\colon n$

~~~~for~$k=j+1\colon n$

~~~~~~$S_{i,k}=S_{i,k}-L_{i,j}L_{k,j}$

~~~~end

~~end

end
\end{lyxcode}
\end{xca}

\begin{enumerate}
\item Show that this implementation is correct.
\item Show that in each outer iteration, the code inside the inner loop,
the $k$ loop, performs $\eta^{2}(L_{j+1\colon n,j})$ nontrivial
subtractions (subtractions in which a nonzero value is subtracted).
\item This implementation of the Cholesky factorization is often called
$jik$ Cholesky, because the ordering of the loops is $j$ first (outermost),
then $j$, and finally $k$, the inner loop. In fact, all $6$ permutations
of the loop indices yield correct Cholesky factorizations; the expression
inside the inner loop should is the same in all the permutations.
Show this by providing $6$ appropriate \noun{Matlab} functions.
\item For each of the $6$ permutations, consider the middle iteration of
the outer loop. Sketch the elements of $A$, the elements of $L$,
and the elements of $S$ that are referenced during this iteration
of the outer loop. For the $jik$ loop ordering shown above, $A$
is not referenced inside outer loop, and for $L$ and $S$ the sketches
shown above.
\end{enumerate}
\begin{xca}
In this exercise we explore fill and work in the Cholesky factorization
of banded and low-profile matrices. We say that a matrix has a half-bandwidth
$k$ if for all $i<j-k$ we have $A_{j,i}=0$. 
\end{xca}

\begin{enumerate}
\item Suppose that $A$ corresponds to an $x$-by-$y$ two-dimensional mesh
whose vertices are ordered as in the previous chapter. What is the
half bandwidth of $A$? What happens when $x$ is larger than $y$,
and what happens when $y$ is larger? Can you permute $A$ to achieve
a minimal bandwidth in both cases?
\item Show that in the factorization of a banded matrix, all the fill occurs
within the band. That is, $L$ is also banded with the same bound
on its half bandwidth.
\item Compute a bound on fill and work as a function of $n$ and $k$.
\item In some cases, $A$ is banded but with a large bandwidth, but in most
of the rows and/or columns all the nonzeros are closer to the main
diagonal than predicted by the half-bandwidth. Can you derive an improved
bound that takes into account the local bandwidth of each row and/or
column? In particular, you need to think about whether a bound on
the rows or on the columns, say in the lower triangular part of $A$,
is more useful.
\item We ordered the vertices of our meshes in a way that matrices come
out banded. In many applications, the matrices are not banded, but
they can be symmetrically permuted to a banded form. Using an $x$-by-$y$
mesh with $x\gg y$ as a motivating example, develop a graph algorithm
that finds a permutation $P$ that clusters the nonzeros of $PAP^{T}$
near the main diagonal. Use a breadth-first-search as a starting point.
Consider the previous quesion in the exercise as you refine the algorithm.
\item Given an $x$-by-$y$ mesh with $x\gg y$, derive bounds for fill
and work in a generalized-nested-dissection factorization of the matrix
corresponding to the mesh. Use separators that approximately bisect
along the $x$ dimention until you reach subgraphs with a square aspect
ratio. Give the bounds in terms of $x$ and $y$.
\end{enumerate}
\begin{lyxcode}

\end{lyxcode}

%% file: chapter-laplacians.tex
 \setcounter{chapter}{3}

\chapter{\label{chapter:graphs-and-laplacians}Symmetric Diagonally-Dominant
Matrices and Graphs}

Support theory provides effective algorithms for constructing preconditioners
for diagonally-dominant matrices and effective ways to analyze these
preconditioners. This chapter explores the structure of diagonally-dominant
matrices and the relation between graphs and diagonallly-dominant
matrices. The next chapter will show how we can use the structure
of diagonally-dominant matrices to analyze preconditioners, and the
chapter that follows presents algorithms for constructing algorithms.

\section{Incidence Factorizations of Diagonally-Dominant Matrices}
\begin{defn}
\label{def:diagonally-dominant-matrix}A square matrix $A\in\mathbb{R}^{n\times n}$
is called \emph{diagonally-dominant} if for every $i=1,2,\ldots n$
we have
\[
A_{ii}\geq\sum_{{j=1\atop j\neq i}}^{n}\left|A_{ij}\right|\;.
\]
\end{defn}

Symmetric diagonally dominant matrices have symmetric factorizations
$A=UU^{T}$ such that each column of $U$ has at most two nonzeros,
and all nonzeros in each column have the same absolute values. We
now establish a notation for such columns. 
\begin{defn}
\label{def:edge-vertex-vectors}Let $1\leq i,j\leq n$, $i\neq j$.
A length-$n$ \emph{positive edge vector}, denoted $\left\langle i,-j\right\rangle $,
is the vector
\[
\left\langle i,-j\right\rangle =\begin{matrix}\strut\\
i\\
\strut\\
j\\
\strut
\end{matrix}\begin{bmatrix}\vdots\\
+1\\
\vdots\\
-1\\
\vdots
\end{bmatrix},\quad\left\langle i,-j\right\rangle _{k}=\left\{ \begin{array}{rl}
+1 & k=i\\
-1 & k=j\\
0 & \textrm{otherwise}.
\end{array}\right.
\]
A \emph{negative edge vector} $\left\langle i,j\right\rangle $ is
the vector
\[
\left\langle i,j\right\rangle =\begin{matrix}\strut\\
i\\
\strut\\
j\\
\strut
\end{matrix}\begin{bmatrix}\vdots\\
+1\\
\vdots\\
+1\\
\vdots
\end{bmatrix},\quad\left\langle i,j\right\rangle _{k}=\left\{ \begin{array}{rl}
+1 & k=i\\
+1 & k=j\\
0 & \textrm{otherwise}.
\end{array}\right.
\]
The reason for the assignment of signs to edge vectors will become
apparent later. A \emph{vertex vector} $\left\langle i\right\rangle $
is the unit vector
\[
\left\langle i\right\rangle _{k}=\left\{ \begin{array}{rl}
+1 & k=i\\
0 & \textrm{otherwise}.
\end{array}\right.
\]
\end{defn}

We now show that a symmetric diagonally dominant matrix can always
be expressed as a sum of outer products of edge and vertex vectors,
and therefore, as a symmetric product of a matrix whose columns are
edge and vertex vectors.
\begin{lem}
\label{lemma:diag-dom-rank-1-summations}Let $A\in\mathbb{R}^{n\times n}$
be a diagonally dominant symmetric matrix. We can decompose A as follows
\begin{eqnarray*}
A & = & \sum_{\substack{i<j\\
A_{ij}>0
}
}\left|A_{ij}\right|\left\langle i,j\right\rangle \left\langle i,j\right\rangle ^{T}\\
 &  & +\sum_{\substack{i<j\\
A_{ij}<0
}
}\left|A_{ij}\right|\left\langle i,-j\right\rangle \left\langle i,-j\right\rangle ^{T}\\
 &  & +\sum_{i=1}^{n}\left(A_{ii}-\sum_{\substack{j=1\\
j\neq i
}
}^{n}\left|A_{ij}\right|\right)\left\langle i\right\rangle \left\langle i\right\rangle ^{T}\\
 & = & \sum_{\substack{i<j\\
A_{ij}>0
}
}\left(\sqrt{\left|A_{ij}\right|}\;\left\langle i,j\right\rangle \right)\left(\sqrt{\left|A_{ij}\right|}\;\left\langle i,j\right\rangle \right)^{T}\\
 &  & +\sum_{\substack{i<j\\
A_{ij}<0
}
}\left(\sqrt{\left|A_{ij}\right|}\;\left\langle i,-j\right\rangle \right)\left(\sqrt{\left|A_{ij}\right|}\;\left\langle i,-j\right\rangle \right)^{T}\\
 &  & +\sum_{i=1}^{n}\left(\sqrt{A_{ii}-\sum_{\substack{j=1\\
j\neq i
}
}^{n}\left|A_{ij}\right|}\;\left\langle i\right\rangle \right)\left(\sqrt{A_{ii}-\sum_{\substack{j=1\\
j\neq i
}
}^{n}\left|A_{ij}\right|}\;\left\langle i\right\rangle \right)^{T}\;.
\end{eqnarray*}
\end{lem}

\begin{proof}
The terms in the summations in lines 1 and 4 are clearly equal (we
only distributed the scalars), and so are the terms in lines 2 and
5 and in lines 3 and 6. Therefore, the second equality holds.

We now show that the first equality holds. Consider the rank-$1$
matrix
\[
\left\langle i,-j\right\rangle \left\langle i,-j\right\rangle ^{T}=\begin{bmatrix}\ddots\\
 & +1 &  & -1\\
 &  & \ddots\\
 & -1 &  & +1\\
 &  &  &  & \ddots
\end{bmatrix}\;,
\]
and similarly,
\[
\left\langle i,j\right\rangle \left\langle i,j\right\rangle ^{T}=\begin{bmatrix}\ddots\\
 & +1 &  & +1\\
 &  & \ddots\\
 & +1 &  & +1\\
 &  &  &  & \ddots
\end{bmatrix}
\]
(in both matrices the four nonzeros are in rows and columns $i$ and
$j$). Suppose $A_{ij}>0$ for some $i\neq j$. The only contribution
to element $i,j$ from the three sums is from a single term in the
first sum, either the term $\left|A_{ij}\right|\left\langle i,j\right\rangle \left\langle i,j\right\rangle ^{T}$
or the term $\left|A_{ji}\right|\left\langle j,i\right\rangle \left\langle j,i\right\rangle ^{T}$,
depending on whether $i>j$. The $i,j$ element of this rank-$1$
matrix is $\left|A_{ij}\right|=A_{ij}$. If $A_{ij}<0$, a similar
argument shows that only $\left|A_{ij}\right|\left\langle i,-j\right\rangle \left\langle i,-j\right\rangle ^{T}$
(assuming $i<j$) contributes to the $i,j$ element, and that the
value of the element $-\left|A_{ij}\right|=A_{ij}$. The third summation
ensures that the values of diagonal elements is also correct.
\end{proof}
The matrix decompositions of this form play a prominent role in support
theory, so we give them a name:
\begin{defn}
\label{def:incidence-factorization}A matrix whose columns are scaled
edge and vertex vectors (that is, vectors of the forms $c\left\langle i,-j\right\rangle $,
$c\left\langle i,j\right\rangle $, and $c\left\langle i\right\rangle $)
is called an \emph{incidence matrix}. A factorization $A=UU^{T}$
where $U$ is an incidence matrix is called an \emph{incidence factorization}.
An incidence factorization with no zero columns, with at most one
vertex vector for each index $i$, with at most one edge vector for
each index pair $i,j$, and whose positive edge vectors are all of
the form $c\left\langle \min(i,j),-\max(i,j)\right\rangle $ is called
a \emph{canonical incidence factorization}.
\end{defn}

\begin{lem}
Let $A\in\mathbb{R}^{n\times n}$ be a diagonally dominant symmetric
matrix. Then $A$ has an incidence factorization $A=UU^{T}$, and
a unique canonical incidence factorization.
\end{lem}

\begin{proof}
The existence of the factorization follows directly from Lemma~\ref{lemma:diag-dom-rank-1-summations}.
We now show that the canonical incidence factorization is uniquely
determined by $A$. Suppose that $A_{ij}=0$. Then $U$ cannot have
a column which is a nonzero multiple of $\left\langle i,j\right\rangle $,
$\left\langle i,-j\right\rangle $, or $\left\langle -i,j\right\rangle $,
since if it did, there would be only one such column, which would
imply $\left(UU^{T}\right)_{ij}\neq0=A_{ij}$. Now suppost that $A_{ij}>0$.
Then $U$ must have a column $\sqrt{A_{ij}}\,\left\langle i,j\right\rangle $.
Similarly, if $A_{ij}<0$, then $U$ must have a column $\sqrt{-A_{ij}}\,\left\langle \min(i,j),-\max(i,j)\right\rangle $.
The uniqueness of the edge vectors implies that the scaled vertex
vectors in $U$ are also unique. 
\end{proof}

\section{Graphs and Their Laplacians Matrices}

We now define the connection between undirected graphs and diagonally-dominant
symmetric matrices. 
\begin{defn}
Let $G=(\{1,2,\ldots n\},E)$ be an undirected graph on the vertex
set $\{1,2,\ldots,n\}$ with no self loops or parallel edges. That
is, the edge-set $E$ consists of unordered pairs of unequal integers
$(i,j)$ such that $1\leq i,j\leq n$ and $i\neq j$. The degree of
a vertex $i$ is the numbers of edges incident on it. The \emph{Laplacian}
of $G$ is the matrix $A\in\mathbb{R}^{n\times n}$ such that
\[
A_{ij}=\left\{ \begin{array}{rl}
\textrm{degree}(i) & i=j\\
-1 & (i,j)\in E\text{ (the index pair is unordered)}\\
0 & \textrm{otherwise}.
\end{array}\right.
\]
\end{defn}

\begin{lem}
The Laplacian of an undirected graph is symmetric and diagonally dominant.
\end{lem}

\begin{proof}
The Laplacian is symmetric because the graph is undirected. The Laplacian
is diagonally dominant because the number of off-diagonal nonzero
entries in row $i$ is exactly $\textrm{degree}(i)$ and the value
of each such nonzero is $-1$.
\end{proof}
We now generalize the definition of Laplacians to undirected graphs
with positive edge weights.
\begin{defn}
Let $G=(\{1,2,\ldots n\},E,c)$ be a weighted undirected graph on
the vertex set $\{1,2,\ldots,n\}$ with no self loops, and with a
weight function $c:E\rightarrow\mathbb{R}\setminus\{0\}$. The \emph{weighted
Laplacian} of $G$ is the matrix $A\in\mathbb{R}^{n\times n}$ such
that
\[
A_{ij}=\left\{ \begin{array}{rl}
\sum_{(i,k)\in E}\left|c(i,k)\right| & i=j\\
-c(i,j) & (i,j)\in E\\
0 & \textrm{otherwise}.
\end{array}\right.
\]
If some of the edge weights are negative, we call $G$ a \emph{signed}
graph, otherwise, we simply call it a weighted graph.
\end{defn}

\begin{lem}
The Laplacian of a weighted undirected graph (signed or unsigned)
is symmetric and diagonally dominant.
\end{lem}

\begin{proof}
Again symmetry follows from the fact that the graph is undirected.
Diagonal dominance follows from the following equation, which holds
for any $i$, $1\leq i\leq n$:
\[
\sum_{\substack{j=1\\
j\neq i
}
}^{n}\left|A_{ij}\right|=\sum_{\substack{j=1\\
j\neq i\\
A_{ij}\neq0
}
}^{n}\left|A_{ij}\right|=\sum_{(i,j)\in E}^{n}\left|A_{ij}\right|=\sum_{(i,j)\in E}^{n}\left|c(i,j)\right|=A_{ii}\;.
\]
\end{proof}
If we allow graphs to have nonnegative vertex weights as well as edge
weights, then this class of graphs becomes completely isomorphic to
symmetric diagonally-dominant matrices.
\begin{defn}
\label{def:general-laplacians}Let $G=(\{1,2,\ldots n\},E,c,d)$ be
a weighted undirected graph on the vertex set $\{1,2,\ldots,n\}$
with no self loops, and with weight functions $c:E\rightarrow\mathbb{R}\setminus\{0\}$
and $d:\{1,\ldots,n\}\rightarrow\mathbb{R}_{+}\cup\{0\}$. The \emph{Laplacian}
of $G$ is the matrix $A\in\mathbb{R}^{n\times n}$ such that
\[
A_{ij}=\left\{ \begin{array}{rl}
d(i)+\sum_{(i,k)\in E}\left|c(i,k)\right| & i=j\\
-c(i,j) & (i,j)\in E\\
0 & \textrm{otherwise}.
\end{array}\right.
\]

A vertex $i$ such that $d(i)>0$ is called a \emph{strictly dominant}
vertex.
\end{defn}

\begin{lem}
The Laplacians of the graphs defined in Definition~\ref{def:general-laplacians}
are symmetric and diagonally dominant. Furthermore, these graphs are
isomorphic to symmetric diagonally-dominant matrices under this Laplacian
mapping.
\end{lem}

\begin{proof}
Symmetry again follows from the fact that the graph is undirected.
For any $i$ we have
\[
\sum_{\substack{j=1\\
j\neq i
}
}^{n}\left|A_{ij}\right|=\sum_{\substack{j=1\\
j\neq i\\
A_{ij}\neq0
}
}^{n}\left|A_{ij}\right|=\sum_{(i,j)\in E}^{n}\left|A_{ij}\right|=\sum_{(i,j)\in E}^{n}\left|c(i,j)\right|=A_{ii}-d(i)\;,
\]
so
\[
A_{ii}=d(i)+\sum_{\substack{j=1\\
j\neq i
}
}^{n}\left|A_{ij}\right|\geq\sum_{\substack{j=1\\
j\neq i
}
}^{n}\left|A_{ij}\right|
\]
because $d(i)\geq0$. This shows that Laplacians are diagonally dominant.
The isomorphism follows from the fact that the following expressions
uniquely determine the graph $(\{1,.\ldots,n\},E,c,d)$ associated
with a diagonally-dominant symmetric matrix:
\begin{eqnarray*}
(i,j)\in E & \text{iff} & i\neq j\text{ and }A_{ij}\neq0\\
c(i,j) & = & -A_{ij}\\
d(i) & = & A_{ii}-\sum_{\substack{j=1\\
j\neq i
}
}^{n}\left|A_{ij}\right|\;.
\end{eqnarray*}
\end{proof}
We prefer to work with vertex weights rather than allowing self loops
because edge and vertex vectors are algebraically different. As we
shall see below, linear combinations of edge vectors can produce new
edge vectors, but never vertex vectors. Therefore, it is convenient
to distinguish edge weights that correspond to scaling of edge vectors
from vertex weights that correspond to scaling of vertex vectors.

In algorithms, given an explicit representation of a diagonally-dominant
matrix $A$, we can easily compute an explicit representation of an
incidence factor $U$ (including the canonical incidence factor if
desired). Sparse matrices are often represented by a data structure
that stores a compressed array of nonzero entries for each row or
each column of the matrix. Each entry in a row (column) array stores
the column index (row index) of the nonzero, and the value of the
nonzero. From such a representation of $A$ we can easily construct
a sparse representation of $U$ by columns. We traverse each row of
$A$, creating a column of $U$ for each nonzero in the upper (or
lower) part of $A$. During the traversal, we can also compute all
the $d(i)$'s. The conversion works even if only the upper or lower
part of $A$ is represented explicitly.

We can use the explicit representation of $A$ as an implicit representation
of $U$, with each off-diagonal nonzero of $A$ representing an edge-vector
column of $U$. If $A$ has no strictly-dominant rows, that is all.
If $A$ has strictly dominant rows, we need to compute their weights
using a linear traversal of $A$.

\section{Laplacians and Resistive Networks}

Weighted (but unsigned) Laplacians model voltage-current relationships
in electrical circuits made up of resistors. This fact has been used
in developing some support preconditioners. This section explains
how Laplacians model resistive networks.

Let $G_{A}=(\{1,\ldots,n\},E,c)$ be an edge-weighted graph and let
$A$ be its Laplacian. We view $G_{A}$ as an electrical circuit with
$n$ nodes (connection points) and with $m=|E|$ resistors. If $(i,j)\in E$,
then there is a resistor with capacitance $c(i,j)$ between nodes
$i$ and $j$ (equivalently, with resistence $1/c(i,j)$).

Given a vector $x$ of node potentials, the voltage drop across resistor
$(i,j)$ is $|x_{i}-x_{j}|$. If the voltage drop is nonzero, current
will flow across the resistor. Current flow has a direction, so we
need to orient resistors in order to assign a direction to current
flows. We arbitrarily orient resistors from the lower-numbered node
to the higher-numbered one. In that direction, the directional voltage
drop across $(i,j)$ is 
\[
x_{\min(i,j)}-x_{\max(i,j)}=\left\langle \min(i,j),-\max(i,j)\right\rangle ^{T}x\;.
\]
We denote $E=\{(i_{1},j_{1}),(i_{2},j_{2}),\ldots,(i_{m},j_{m})\}$
and we denote
\[
\tilde{U}=\left[\begin{array}{ccc}
\left\langle \min(i_{1},j_{1}),-\max(i_{1},j_{1})\right\rangle  & \cdots & \left\langle \min(i_{m},j_{m}),-\max(i_{m},j_{m})\right\rangle \end{array}\right]\;.
\]
Using this notation, the vector $\tilde{U}^{T}x$ is the vector of
all directional voltage drops. The directional current flow across
resistor $(i,j)$ is $c(i,j)(x_{\min(i,j)}-x_{\max(i,j)})$. We denote
by $C$ the diagonal matrix with $C_{k,k}=c(i_{k},j_{k})$. Then the
vector $C\tilde{U}^{T}x$ is the vector of directional current flow
across the $m$ resistors.

Next, we compute nodal currents, the net current flowing in or out
of each node. If we set the potential of node $i$ to $x_{i}$, this
means that we have connected node $i$ to a voltage source, such as
a battery. The voltage source keeps the potential at $i$ at exactly
$x_{i}$ volts. To do so, it might need to send current into $x_{i}$
or to absorb currect flowing out of $x_{i}$. These current flows
also have a direction: we can either compute the current flowing into
$i$, or the current flowing out of $i$. We will arbitrarily compute
the current flowing into $i$. How much current flows into node $i$?
Exactly the net current flowing into it from its neighbors in the
circuit, minus the current flowing from it to its neighbors. Let $(i_{k},j_{k})\in E$.
If $i_{k}<j_{k}$, then $\left(C\tilde{U}^{T}x\right)_{k}$ represents
the current $c(i_{k},j_{k})(x_{\min(i_{k},j_{k})}-x_{\max(i_{k},j_{k})})$
flowing from $i_{k}$ to $j_{k}$, which is positive if $x_{i_{k}}>x_{j_{k}}$.
If $i_{k}>j_{k}$ then current flowing from $i_{k}$ to $j_{k}$ will
have a negative sign (even though it too is flowing out of $i_{k}$if
$x_{i_{k}}>x_{j_{k}}$) and we have to negate it before we add it
to the total current flowing from $i_{k}$ to its neighbors, which
is exactly the net current flowing into $i_{k}$ from the voltage
source. That is, the total current from a node $i$ to its neighbors
is
\[
\sum_{\substack{(i_{k},j_{k})\in E\\
i_{k}=i
}
}(-1)^{i>j_{k}}\left(C\tilde{U}^{T}x\right)_{k}=\tilde{U}_{i,\colon}C\tilde{U}^{T}x\;.
\]
Therefore, the vector of net node currents is exactly $\tilde{U}C\tilde{U}^{T}=A$.
We can also compute the total power dissipated by the circuit. We
multiply the current flowing across each resistor by the voltage drop
and sum over resistors, to obtain $\left(x^{T}\tilde{U}\right)\left(C\tilde{U}^{T}x\right)=x^{T}\tilde{U}C\tilde{U}^{T}x$.

We summarize the various voltage-current relationships:
\[
\begin{array}{ll}
x & \text{nodal voltages}\\
\tilde{U}^{T}x & \text{voltage drops across resistors}\\
C\tilde{U}^{T}x & \text{current flows across resistors}\\
\tilde{U}C\tilde{U}^{T}x=Ax & \text{net nodal currents}\\
x^{T}\tilde{U}C\tilde{U}^{T}x=x^{T}Ax & \text{total power dissipated by the circuit}
\end{array}
\]
These relationships provide physical interpretations to many quantities
and equations that we have already seen. Given a vector $x$ of nodal
voltages, $Ax$ is the corresponding vector of nodal currents, where
$A$ is the Laplacian of the circuit. Conversely, given a vector $b$
of nodal currents, solving $Ax=b$ determines the corresponding nodal
voltages. The maximal eigenvalue of $A$ measures the maximum power
that the circuit can dissipate under a unit vector of voltages, and
the minimal nonzero eigenvalue measures the minimum power that is
dissipated under a unit voltage vector that is orthogonal to the constant
vector (which spans the null space of $A$).

Laplacians with strictly-dominant rows arise when the boundary conditions,
the known quantities in the circuit, include a mixture of nodal voltages
and nodal currents. Suppose that we know the currents in all the nodes
except for one, say $n$, where we know the voltage, not the current.
Let us assume that the voltage at node $n$ is $5$ Volts. We want
to compute the voltages at nodes $1$ to $n-1$. We cannot use the
linear system $\tilde{U}C\tilde{U}^{T}x=Ax=b$ directly, because we
do not know $b_{n}$, and on the other hand, we do know that $x_{n}=5$.
Therefore, we drop the last row from this linear system, since we
do not know its right-hand side. We have
\[
\begin{bmatrix}A_{1,1} & \cdots & A_{1,n-1} & A_{1,n}\\
\vdots &  & \vdots & \vdots\\
A_{n-1,1} & \cdots & A_{n-1,n-1} & A_{n-1,n}
\end{bmatrix}\begin{bmatrix}x_{1}\\
\vdots\\
x_{n-1}\\
5
\end{bmatrix}=\begin{bmatrix}b_{1}\\
\vdots\\
b_{n-1}
\end{bmatrix}
\]
or
\[
\begin{bmatrix}A_{1,1} & \cdots & A_{1,n-1}\\
\vdots &  & \vdots\\
A_{n-1,1} & \cdots & A_{n-1,n-1}
\end{bmatrix}\begin{bmatrix}x_{1}\\
\vdots\\
x_{n-1}
\end{bmatrix}=\begin{bmatrix}b_{1}\\
\vdots\\
b_{n-1}
\end{bmatrix}-5\begin{bmatrix}A_{1,n-1}\\
\vdots\\
A_{n-1,n-1}
\end{bmatrix}\;.
\]
We have obtained a new square and symmetric coefficient matrix and
a known right-hand side. The new matrix is still diagonally dominant,
but now has strictly-dominant rows: if $A_{k,n}\neq0$, then row $k$
in the new coefficient matrix is now strictly-dominant, since we removed
$A_{k,n}$ from it. If we know the voltages at other nodes, we repeat
the process and drop more rows and columns, making the remaining coefficient
matrix even more dominant.

Given two resistive networks with the same number of unknown-voltage
nodes, what is the interpretation of a path embedding $\pi$? An embedding
of $G_{A}$ in $G_{B}$ shows, for each edge $(i,j)$ in $G_{A}$,
a path between $i$ and $j$ in $G_{B}$. That path can carry current
and its conductance, which is the inverse of the sum of resistances
along the path, serves as a lower bound on the conductance between
$i$ and $j$ in $G_{B}$. Intuitively, an embedding allows us to
show that $G_{B}$ is, up to a certain factor, ``as good as'' $G_{A}$,
in the sense that for a given vector of voltages, currents in $G_{B}$
are not that much smaller than the currents in $G_{A}$.

%% file: chapter-augmented-mst.tex
 \setcounter{chapter}{6}

\chapter{\label{chapter:augmented-mst}Augmented Spanning-Tree Preconditioners}

It's time to construct a preconditioner! This chapter combinatorial
algorithms for constructing preconditioners that are based on augmenting
spanning trees with extra edges. 

\section{Spanning Tree Preconditioners}

We start with the simplest support preconditioner, a maximum spanning
tree for weighted Laplacians. The construction of the preconditioner
$B$ aims to achieve three goals:
\begin{itemize}
\item The generalized eigenvalues $(A,B)$ should be at least $1$.
\item The preconditioner should be as sparse as possible and as easy to
factor as possible.
\item The product of the maximum dilation and the maximum congestion should
be low, to ensure that generalized eigenvalues of the pencil $(A,B)$
are not too large. (We can also try to achieve low stretch.)
\end{itemize}
These objectives will not lead us to a very effective preconditioner.
It usually pays to relax the second objective and make the preconditioner
a little denser in order to achive a smaller $\kappa(A,B)$. But here
we strive for simplicity, so we stick with these objectives.

We achieve the first objective using a simple technique. Given $A$,
we compute its canonical incidence factor $U$. We construct $V$
by dropping some of the columns of $U$. If we order the columns of
$U$ so that the columns that we keep in $V$ appear first, then 
\[
V=U\left(\begin{array}{c}
I\\
0
\end{array}\right)\;.
\]
This is a subset preconditioner, so 
\[
\lambda(A,B)\geq\sigma^{-1}(B,A)\geq\left\Vert \left(\begin{array}{c}
I\\
0
\end{array}\right)\right\Vert _{2}^{-2}=1^{-2}=1\;.
\]
We now study the second objective, sparsity in $B$. By Lemma~\ref{lemma:preconditioner-connectedness},
we should $G_{B}$ as connected as $G_{A}$. That is, we cannot drop
so many columns from $U$ that connected components of $G_{A}$ become
disconnected in $G_{B}$. How sparse can we make $G_{B}$ under this
constraint? A spanning forest of $G_{A}$. That is, we can drop edges
from $G_{A}$ until no cycles remain in $G_{B}$. If there are cycles,
we can clearly drop an edge. If there are no cycles, dropping an edge
will disconnect a connected component of $G_{A}$. A spanning subgraph
with no cycles is called a spanning forest. If $G_{A}$ is connected,
the spanning forest is a spanning tree. A spanning forest $G_{B}$
of $G_{A}$ is always very sparse; the number of edges is $n$ minus
the number of connected components in $G_{A}$. The weighted Laplacian
$B$ of a spanning forest $G_{B}$ can be factored into triangular
factors in $\Theta(n)$ time, the factor requires $\Theta(n)$ words
of memory to store, and each preconditioning step requires $\Theta(n)$
arithmetic operations.

Which edges should we drop, and which edges should we keep in the
spanning forest? If $G_{B}$ is a spanning forest of $G_{A}$, than
for every edge $(i_{1},i_{2})$ there is exactly one path in $G_{B}$
between $i_{1}$ and $i_{2}$. Therefore, if $G_{B}$ is a spanning
forest, then there is a unique path embedding $\pi$ of the edges
of $G_{A}$ into paths in $G_{B}$. The dropping policy should try
to minimize the maximum congestion and dilation of the embedding (or
the stretch/crowding of the embedding). The expressions for congestion,
dilation, stretch, and crowsing are sums of ratios whose denominators
are absolute values or squares of absolute values of entries of $B$.
Therefore, to reduce the congestion, dilation, and stretch, we can
try to drop edges $(i_{1},i_{2})$ that correspond to $A_{i_{1},i_{2}}$with
small absolute values, and keep ``heavy'' edges that correspond
to entries of $A$ with large absolute values.

One class of forests that are easy to construct and that favor heavy
edges are \emph{maximum spanning forests}. A maximum spanning forest
maximizes the sum of the weights $-A_{i_{1},i_{2}}$ of the edges
of the forest. One property of maximum spanning forest is particularly
useful for us.
\begin{lem}
Let $G_{B}$ be a maximum spanning forest of the weighted (but unsigned)
graph $G_{A}$, and let $\pi(i_{1},i_{\ell})=(i_{1},i_{2},\ldots,i_{\ell})$
be the path with endpoints $i_{1}$ and $i_{\ell}$ in $G_{B}$. Then
for $j=1,\ldots,\ell-1$ we have $|A_{i_{1},i_{2}}|\leq|A_{i_{j},i_{j+1}}|$.
\end{lem}

\begin{proof}
Suppose for contradition that for some $j$, $|A_{i_{1},i_{\ell}}|>|A_{i_{j},i_{j+1}}|$.
If we add $(i_{1},i_{\ell})$ to $G_{B}$, we create a cycle. If we
then drop $(i_{j},i_{j+1})$, the resulting subgraph again becomes
a spanning forest. The total weight of the new spannign forest is
$|A_{i_{1},i_{\ell}}|-|A_{i_{j},i_{j+1}}|>0$ more than that of $G_{B}$,
contradicting the hypothesis that $G_{B}$ is a maximum spanning forest.
\end{proof}
It follows that in all the summations that constitute the congestion,
dilation, stretch and crowding, the terms are bounded by $1$. This
yields the following result.
\begin{lem}
Let $A$ and $B$ be weighted Laplacians with identical row sums,
such that $G_{B}$ is a maximum spanning forest of $G_{A}$. Then
\[
\kappa(A,B)\leq(n-1)m\;,
\]
where $n$ is the order of $A$ and $B$ and $m$ is the number of
nonzeros in the strictly upper triangular part of $A$.
\end{lem}

\begin{proof}
Let $W$ correspond to the path embedding of $G_{A}$ in $G_{B}$.
For an edge $(i_{1},i_{2})$ in $G_{A}$ we have
\[
\textrm{dilation}_{\pi}(i_{1},i_{2})=\sum_{{(j_{1,},j_{2})\atop (j_{1,},j_{2})\in\pi(i_{1},i_{2})}}\sqrt{\frac{A_{i_{1},i_{2}}}{B_{j_{1},j_{2}}}}\leq\sum_{{(j_{1,},j_{2})\atop (j_{1,},j_{2})\in\pi(i_{1},i_{2})}}1\leq n-1\;.
\]
The rightmost inequality holds because the number of edges in a simple
path in a graph is at most $n-1$, the number of vertices minus one.
The total number of paths that use a single edge in $G_{B}$ is at
most $m$, the number of edges in $G_{A}$, and hence the number of
paths (in fact, its not hard to see that the number of paths is at
most $m-(n-2)$). Therefore, 
\begin{eqnarray*}
\kappa(A,B) & \leq & \sigma(A,B)/\sigma(B,A)\\
 & \leq & \sigma(A,B)/1\\
 & \leq & \left(\max\left\{ 1,\max_{(i_{1},i_{2})\in G_{A}}\textrm{dilation}_{\pi}(i_{1},i_{2})\right\} \right)\left(\max\left\{ 1,\max_{(j_{1},j_{2})\in G_{A}}\textrm{congestion}_{\pi}(j_{1},j_{2})\right\} \right)\\
 & \leq & (n-1)m\;.
\end{eqnarray*}
A similar argument shows that the stretch of an edge is at most $n-1$,
and since there are at most $m$ edges, $\|W\|_{F}^{2}$ is at most
$(n-1)m+n$.
\end{proof}
Algorithms for constructing minimum spanning trees and forests can
easily be adapted to construct maximum spanning forests. For example,
Kruskal's algorithm starts out with no edges in $G_{B}$. It sorts
the edges of $G_{A}$ by weight and processes from heavy to light.
For each edge, the algorithm determines whether its endpoints are
in the same connected component of the current forest. If the endpoints
are in the same component, then adding the edge would close a cycle,
so the edge is dropped from $G_{B}$. Otherwise, the edge is added
to $G_{B}$. This algorithm requires a union-find data structure to
determine whether two vertices belong to the same connected components.
It is easy to implement the algorithm so that it performs $O(m\log m)$
operations. The main data structure in another famous algorithm, Prim's,
is a priority queue of vertices, and it can be implemented so that
it performs $O(m+n\log n)$ operations. If $A$ is very sparse, Prim's
algorithm is faster.

When we put together the work required to construct a maximum-spanning-forest
preconditioner, to factor it, the condition number of the preconditioned
system, and the cost per iteration, we can bound the total cost of
the linear solver.
\begin{thm}
Let $A$ be a weighted Laplacians of order $n$ with $n+2m$ nonzeros.
Then a minimal-residual preconditioned Krylov-subspace method with
a maximum-spanning-forest preconditioner can solve a consistent linear
system $Ax=b$ using $O((n+m)\sqrt{nm})$ operations.
\end{thm}

\begin{proof}
Constructing the preconditioner requires $O(m+n\log n)$ work. Computing
the Cholesky factorization of the preconditioner requires $\Theta(n)$
work. The cost per iteration is $\Theta(n+m)$ operations, because
$A$ has $\Theta(n+m)$ nonzeros and the factor of the preconditioner
only $\Theta(n)$. The condition number of the preconditioned system
is bounded by $nm$, so the number of iteration to reduce the relative
residual by a constant factor is $O(\sqrt{nm})$. Thus, the total
solution cost is
\[
O(m+n\log n)+\Theta(n)+O((n+m)\sqrt{nm})=O((n+m)\sqrt{nm})\;.
\]
\end{proof}
Can we do any better with spanning forest preconditioner? It seems
that the congesion-dilation-product bound cannot give a condition-number
bound better than $O(mn)$.\marginpar{Is there a lower bound?} But
the stretch bound can. Constructions for \emph{low-stretch trees}
can construct a spanning forest $G_{B}$ for $G_{A}$ such that 
\[
\sum_{(i_{1},i_{2})\in G_{A}}\textrm{stretch}_{\pi}(i_{1},i_{2})=O\left((m+n)(\log n\log\log n)^{2}\right)\;,
\]
where $\pi$ is the embedding of edges of $G_{A}$ into paths in $G_{B}$.
Like maximum spanning forests, lowe-stretch forests also favor heavy
edges, but their optimization criteria are different. The number of
operations required to construct such a forest is $O((m+n)\log^{2}n)$.
The constructoin details are considerably more complex than those
of maximum spanning forests, so we do not give them here. The next
theorem summarizes the total cost to solve a linear system with a
low-stretch forest preconditioner.
\begin{thm}
Let $A$ be a weighted Laplacians of order $n$ with $n+2m$ nonzeros.
A minimal-residual preconditioned Krylov-subspace method with a low-stretch-forest
preconditioner can solve a consistent linear system $Ax=b$ using
$O((n+m)^{1.5}(\log n\log\log n)$ operations.
\end{thm}

\begin{proof}
Constructing the preconditioner requires $O((m+n)\log^{2}n)$ work.
Computing the Cholesky factorization of the preconditioner requires
$\Theta(n)$ work. The cost per iteration is $\Theta(n+m)$. The condition
number of the preconditioned system is $O((m+n)(\log n\log\log n)^{2})$.
Thus, the total solution cost is
\begin{multline*}
O((m+n)\log^{2}n)+\Theta(n)+O\left((n+m)\sqrt{(m+n)(\log n\log\log n)^{2}}\right)=\\
O\left((n+m)^{1.5}(\log n\log\log n\right)\;.
\end{multline*}
\end{proof}

\section{Vaiyda's Augmented Spanning Trees}

The analysis of spanning-tree preconditioners shows how they can be
improved. The construction of the preconditioner is cheap, the factorization
of the preconditioner is cheap, each iteration is cheap, but the solver
performs many iterations. This suggests that it might be better to
make the preconditioner a bit denser. This would make the construction
and the factorization more expensive, and it would make every iteration
more expensive. But if we add edges cleverly, the number of iterations
can be dramatically reduced.

The following algorithm adds edges to a maximum spanning tree in an
attempt to reduce the bounds on both the congestion and the dilation.
The algorithm works in two phases. In the first phase, the algorithm
removes edges to partition the tree into about $t$ subtrees of roughly
the same size. In the second phase, the algorithm adds the heaviest
edge between every two subtrees. The value $t$ is a parameter that
controls the density of the preconditioner. When $t$ is small, the
preconditioner remains a tree or close to a tree. As $t$ grows, the
preconditioner becomes denser and more expensive to construct and
to factor, the cost of every iteration grows, but the number of iterations
shrinks.

The partitioning algorithm, called \noun{TreePartition}, is recursive
and works as follows. It maintains an integer vector $s$. The element
$s_{i}$ is the number of vertices in the subtree rooted at $i$.
We begin by calling \noun{TreePartition} on the root of the maximum
spanning tree. The algorithm starts processing

\begin{algorithm}
\noindent \noun{TreeParition}(vertex $i$, integer array $s$)\\
~~$s_{i}\leftarrow1$\\
~~for each child $j$ of $i$\\
~~~~if \noun{$(s_{j}>n/t+1)$}\\
\noun{~~~~~~TreePartition}($j$,$s$)

\caption{}
\end{algorithm}

\section{Notes and References}

The low-stretch forests are from Elkin-Emek-Spielman-Teng, STOC 2005.
This is an improvement over Alon-Karp-Peleg-West.

%% file: chapter-linear-algebra.tex
\chapter{Linear Algebra Notation and Definitions}

This appendix provides notation, definitions and some well known results,
mostly in linear algebra. We assume that all matrices and vectors
are real unless we explicitly state otherwise.

A square matrix is symmetric if $A_{ij}=A_{ji}$.

\section{Eigenvalues and Eigenvectors}

A square matrix is \emph{positive definite} if $x^{T}Ax>0$ for all
$x$ and \emph{positive semidefinite} if $x^{T}Ax\geq0$ for all $x$.

The \emph{eigendecomposition} of a square matrix is a factorization
$A=V\Lambda V^{-1}$ where $\Lambda$ is diagonal. The columns of
$V$ are called \emph{eigenvectors} of $A$ and the diagonal elements
of $\Lambda$ are called the \emph{eigenvalues} of $A$. A real matrix
may have complex eigenvalues. The expression $\Lambda(A)$ denotes
the set of eigenvalues of $A$.

Not all matrices have an eigendecomposition, but every symmetric matrix
has one. In particular, the eigenvalues of symmetric matrices are
real and their eigenvectors are orthogonal to each other. Therefore,
the eigendecomposition of symmetric matrices is of the form $A=V\Lambda V^{T}$.
The eigenvalues of positive definite matrices are all positive, and
the eigenvalues of positive semidefinite matrices are all non-negative.

For any matrix $V$, the product $VV^{T}$ is symmetric and positive
semidefinite.

\section{The Singular Value Decomposition}

Every matric $A\in\mathbb{C}^{m\times n}$, $m\geq n$ has a \emph{singular
value decomposition} (SVD) $A=U\Sigma V^{*}$ where $U\in\mathbb{C}^{m\times n}$
with orthonormal columns, $V\in\mathbb{C}^{n\times n}$ with orthonormal
columns, and $\Sigma\in\mathbb{R}^{n\times n}$ is non-negative and
diagonal. (This decomposition is sometimes called the \emph{reduced}
SVD, the full one being with a rectangular $U$ and an $m$-by-$n$
$\Sigma$.) The columns of $U$ are called \emph{left singular vectors},
the columns of $V$ are called right singular vectors, and the diagonal
elements of $\Sigma$ are called \emph{singular values}. The singular
values are non-negative. We denote the set of singular values of $A$
by $\Sigma(A)$.

\section{Generalized Eigenvalues}

Preconditioning involves two matrices, the coefficient matrix and
the preconditioners. The convergence of iterative linears solvers
for symmetric semidefinite problems depends on the generalized eigenvalues
of the pair of matrices. A pair $(S,T)$ of matrices is also called
a pencil.
\begin{defn}
Let $S$ and $T$ be $n$-by-$n$ complex matrices. We say that a
scalar $\lambda$ is a \emph{finite generalized eigenvalue} of the
matrix pencil (pair) $(S,T)$ if there is a vector $v\neq0$ such
that 
\[
Sv=\lambda Tv
\]
 and $Tv\neq0$. We say that $\infty$ is a \emph{infinite generalized
eigenvalue} of $(S,T)$ if there exist a vector $v\neq0$ such that
$Tv=0$ but $Sv\neq0$. Note that $\infty$ is an eigenvalue of $(S,T)$
if and only if $0$ is an eigenvalue of $(T,S)$. The finite and infinite
eigenvalues of a pencil are \emph{determined eigenvalues} (the eigenvector
uniquely determines the eigenvalue). If both $Sv=Tv=0$ for a vector
$v\neq0$, we say that $v$ is an \emph{indeterminate eigenvector},
because $Sv=\lambda Tv$ for any scalar $\lambda$.
\end{defn}

We order from smallest to largest. We will denote the $k$th eigenvalue
of $S$ by $\lambda_{k}(S)$, and the $k$th determined generalized
eigenvalue of $(S,T)$ by $\lambda_{k}(S,T)$. Therefore $\lambda_{1}(S)\leq\dots\leq\lambda_{l}(S)$
and $\lambda_{1}(S,T)\leq\dots\leq\lambda_{d}(S,T)$, where $l$ is
the number of eigenvalues $S$ has, and $d$ is the number of determined
eigenvalues that $(S,T)$ has.
\begin{defn}
A pencil $(S,T)$ is \emph{Hermitian positive semidefinite} (H/PSD)
if $S$ is Hermitian, $T$ is positive semidefinite, and $\mbox{null}(T)\subseteq\mbox{null}(S)$. 
\end{defn}

The generalized eigenvalue problem on H/PSD pencils is, mathematically,
a generalization of the Hermitian eigenvalue problem. In fact, the
generalized eigenvalues of an H/PSD can be shown to be the eigenvalues
of an equivalent Hermitian matrix. The proof appears in the Appendix.
Based on this observation it is easy to show that other eigenvalue
properties of Hermitian matrices have an analogy for H/PSD pencils.
For example, an H/PSD pencil, $(S,T)$, has exactly $\mbox{rank}(T)$
determined eigenvalues (counting multiplicity), all of them finite
and real.

A useful tool for analyzing the spectrum of an Hermitian matrix is
the \emph{Courant-Fischer Minimax Theorem}~\cite{MatrixComp}.
\begin{thm}
\label{thm:CF}(Courant-Fischer Minimax Theorem) Suppose that $S\in\mathbb{C}^{n\times n}$
is an Hermitian matrix, then 
\[
\lambda_{k}(S)=\min_{\dim(U)=k}\max_{x\in U}\frac{x^{*}Sx}{x^{*}x}
\]
 and 
\[
\lambda_{k}(S)=\max_{\dim(V)=n-k+1}\min_{x\in V}\frac{x^{*}Sx}{x^{*}x}\,.
\]
\end{thm}

As discussed above, the generalized eigenvalue problem on H/PSD pencils
is a generalization of the eigenvalue problem on Hermitian matrices.
Therefore, there is a natural generalization of Theorem~\ref{thm:CF}
to H/PSD pencils, which we refer to as the \emph{Generalized Courant-Fischer
Minimax Theorem.} We now state the theorem. For completeness the proof
appears in the Appendix.
\begin{thm}
\label{thm:gen-CF}(Generalized Courant-Fischer Minimax Theorem) Suppose
that $S\in\mathbb{C}^{n\times n}$ is an Hermitian matrix and that
$T\in\mathbb{C}^{n\times n}$ is an Hermitian positive semidefinite
matrix such that $\text{\emph{null}}(T)\subseteq\text{\emph{null}}(S)$.
For $1\leq k\leq\mbox{\emph{rank}}(T)$ we have 
\begin{align*}
\lambda_{k}(S,T) & =\min_{\begin{array}{c}
\dim(U)=k\\
U\perp\mbox{\emph{null}}(T)
\end{array}}\max_{x\in S}\frac{x^{*}Sx}{x^{*}Tx}
\end{align*}
 and 
\[
\lambda_{k}(S,T)=\max_{\begin{array}{c}
\dim(V)=\mbox{rank}(T)-k+1\\
V\perp\mbox{\emph{null}}(T)
\end{array}}\min_{x\in S}\frac{x^{*}Sx}{x^{*}Tx}\,.
\]
\end{thm}